\documentclass[10pt]{amsart}

\usepackage{amsfonts}
\usepackage{amssymb}

\usepackage{hyperref}
\hypersetup{colorlinks=true,citecolor=blue,linkcolor=blue,urlcolor=blue,
pdfstartview=FitH, pdfauthor=Djordje Milicevic, pdftitle=Sub-Weyl subconvexity for Dirichlet L-functions to prime power moduli}

\newtheorem{lemma}{Lemma}
\newtheorem{theorem}{Theorem}
\newtheorem{definition}{Definition}

\def\ord{\operatorname{ord}\nolimits}

\begin{document}

\title[Sub-Weyl subconvexity]{Sub-Weyl subconvexity for Dirichlet $L$-functions to prime power moduli}
\author{Djordje Mili\'cevi\'c}
\address{Bryn Mawr College, Department of Mathematics, 101 North Merion Avenue, Bryn Mawr, PA 19010, U.S.A.}
\email{dmilicevic@brynmawr.edu}

\begin{abstract}
We prove a subconvexity bound for the central value $L(\frac12,\chi)$ of a Dirichlet $L$-function of a character $\chi$ to a prime power modulus $q=p^n$ of the form $L(\frac12,\chi)\ll p^rq^{\theta+\epsilon}$ with a fixed $r$ and $\theta\approx 0.1645<\frac16$, breaking the long-standing Weyl exponent barrier. In fact, we develop a  general new theory of estimation of short exponential sums involving $p$-adically analytic phases, which can be naturally seen as a $p$-adic analogue of the method of exponent pairs. This new method is presented in a ready-to-use form and applies to a wide class of well-behaved phases including many that arise from a stationary phase analysis of hyper-Kloosterman and other complete exponential sums.
\end{abstract}

\maketitle

\bigskip\bigskip

\section{Introduction and statement of results}

One of the principal questions about $L$-functions is the size of their critical values. In this paper, we address an instance of the subconvexity problem, which we describe below, and break a long-standing barrier known as the Weyl exponent for central values of certain Dirichlet $L$-functions.

In the case of the Riemann zeta-function, the distribution of values of $\zeta(1/2+it)$ for large $t$ is of central interest; see Titchmarsh \cite{Titchmarsh}. From the functional equation and the Phragm\'en-Lindel\"of principle, it follows that
\begin{equation}
\label{ZetaSubconvexity}
\left|\zeta(1/2+it)\right|\ll(1+|t|)^{\theta+\epsilon}
\end{equation}
with $\theta=\frac14$. The Lindel\"of hypothesis, the statement that \eqref{ZetaSubconvexity} holds with $\theta=0$, is a consequence of the celebrated Riemann hypothesis and lies very much out of reach of current methods, but an estimate of the form \eqref{ZetaSubconvexity} where $\theta<\frac14$ has important implications. It was proved by Hardy and Littlewood by using Weyl differencing that \eqref{ZetaSubconvexity} holds with $\theta=\frac16$. This exponent was lowered by Walfisz \cite{Walfisz} in 1924 to $\theta=\frac{163}{988}\approx 0{.}1650$; many subsequent papers slowly improved the result to the current value $\theta=\frac{32}{205}\approx 0{.}1561$, due to Huxley \cite{Hux05}.

For an automorphic representation $\pi$ of $\text{GL}(n)$, the statement that
\begin{equation}
\label{SubconvexityLfunctionsEquation}
|L(1/2,\pi)|\ll C(\pi)^{\theta+\epsilon},
\end{equation}
where $C(\pi)$ is the analytic conductor of $\pi$ as defined by Iwaniec and Sarnak \cite{IwaniecSarnakPerspectives} and $\theta=\frac14$, is known as the convexity bound and follows from the basic analytic properties of $L(s,\pi)$. The \textit{subconvexity conjecture} states that such a bound always holds for some $\theta<\frac14$. Proving a subconvex estimate for any given $L$-function requires deep arithmetic considerations and can have important arithmetic, geometric, or dynamical consequences; see surveys \cite{IwaniecSarnakPerspectives,MichelParkCity}. Many cases of subconvexity on lower-dimensional groups have been proved, often with exponents $\theta$ very close to $\frac14$. A breakthrough paper of Michel and Venkatesh \cite{MichelVenkatesh} contains a fully general subconvexity estimate for $\textrm{GL}(2)$ $L$-functions (with $\theta$ close to $\frac14$) by a geometric method and references to previous results. In some cases, $\theta=\frac16$ was proved, and such a result goes under the name of \textit{Weyl exponent}.

In the case of a Dirichlet $L$-function of a character $\chi$ modulo $q$, the corresponding statement that
\begin{equation}
\label{ChiSubconvexity}
|L(1/2,\chi)|\ll q^{\theta+\epsilon}
\end{equation}
is known only with $\theta=\frac3{16}=0{.}1875$, due to Burgess \cite{Burgess}. That the Weyl exponent $\theta=\frac16$ is not known for this family is a major source of frustration. However, building on the ideas of Postnikov \cite{Postnikov}, Barban, Linnik and Tshudakov \cite{BarbanLinnikChudakov} proved estimates allowing them to take $\theta=\frac16$ when considering Dirichlet $L$-functions $L(s,\chi)$ associated to characters $\chi$ modulo $p^n$, where $p$ is a fixed prime and $n\to\infty$. This result was generalized by Heath-Brown \cite{HeathBrownHybrid} to a hybrid bound that contains \eqref{ChiSubconvexity} with an exponent $\frac16\leqslant\theta<\frac14$ assuming that the modulus $q$ has a divisor $d$ in a suitable range, with $\theta=\frac16$ for moduli $q$ having a divisor $d\asymp q^{1/3+o(1)}$ (including all sufficiently powerful moduli). Using a very different approach, Conrey and Iwaniec~\cite{ConreyIwaniec} obtained the Weyl exponent $\theta=\tfrac16$ in the case when $\chi$ is a real (that is, quadratic) character.

The first main result of this article is the following theorem.

\begin{theorem}
\label{MainResult}
Let $\theta>\theta_0\approx 0{.}1645$ be given. There is an $r\geqslant 0$ such that
\[ L\left(\frac12,\chi\right)\ll p^r\cdot q^{\theta}(\log q)^{1/2}\]
holds for every Dirichlet character $\chi$ to any prime power modulus $q=p^n$.
\end{theorem}

In particular, we see that, for sufficiently large $n\geqslant n_0$, Theorem~\ref{MainResult} yields the subconvexity bound \eqref{ChiSubconvexity} with $\theta<\frac16$.
We stress that, even though our method is $p$-adic, the implied constant and the values of $r$ and $n_0$ in Theorem~\ref{MainResult} depend only on the value of $\theta$ and are \textit{universal across all primes $p$} and all prime powers $q=p^n$. This is the first family of $L$-functions since Walfisz's 1924 result for the Riemann zeta-function in which a better exponent than $\frac16$ has been obtained.

\medskip

As the principal device of this paper, we develop a theory of estimation of exponential sums of the form 
\begin{equation}
\label{GeneralSum}
\sum_{M<m\leqslant M+B}e\left(\frac{f(m)}{p^n}\right),
\end{equation}
where $f(t)$ is an analytic function on the ring $\mathbb{Z}_p$ of $p$-adic integers satisfying certain conditions. Here, and throughout the paper, $e(x)$ denotes $e^{2\pi ix}$ and its obvious unique extension from $\cup_{k\in\mathbb{Z}}p^k\mathbb{Z}$ to a $\mathbb{Z}_p$-periodic function on $\mathbb{Q}_p$. In Definition~\ref{ClassesDefinition} (section~\ref{DataPairsSection}, below), we specify a class of ($p$-adically analytic) power series $\mathbf{F}$, which includes multiples $a\log_p(1+p^{\kappa}t)$ of the $p$-adic logarithm, and to which our estimates apply. Roughly speaking, series $f$ in $\mathbf{F}$ satisfy $f'(t)=p^w\omega'(1+p^{\kappa}\omega t)^{-y}+p^w\gamma_0+p^{u+w}g(t)$ with suitable parameters (which vary and are suppressed in this introduction) and a power series $g$ satisfying suitable conditions (ensuring it does not interfere with the first, leading term). We call a pair of non-negative real numbers $(k,\ell)$ a $p$-adic exponent pair if, roughly, for every $f\in\mathbf{F}$ as above, every sufficiently large $n$, and every $0<B\leqslant p^{n-w-\kappa}$,
\[ \sum_{M<m\leqslant M+B}e\left(\frac{f(m)}{p^n}\right)\ll p^r\left(\frac{p^{n-w-\kappa}}B\right)^kB^{\ell}\big(\log p^{n-w-\kappa}\big)^{\delta}, \]
with some $r$, $\delta$ depending on the exponent pair and the parameters implied in $f$. We will only need $\delta\in[0,1]$ in the exponent pairs we construct. In fact, we will be rather more precise and talk about \textit{$p$-adic exponent data} in order to track all dependencies explicitly; see Definition~\ref{ExponentPairsDefinition} (section~\ref{DataPairsSection}, below).

The heart of our method of estimating sums of the form \eqref{GeneralSum} is contained in Theorems \ref{Bprocess} and \ref{Aprocess}, which we term $B$- and $A$-processes, respectively. An immediate consequence of these results is the following compact statement.

\begin{theorem}
\label{ABprocessesTheorem}
If $(k,\ell)$ is a $p$-adic exponent pair, then so are
\[ A(k,\ell)=\left(\frac{k}{2(k+1)},\frac{k+\ell+1}{2(k+1)}\right)\quad\text{and}\quad B(k,\ell)=\left(\ell-\frac12,k+\frac12\right). \]
\end{theorem}

Starting from the ``trivial'' $p$-adic exponent pair $(0,1)$, we obtain with use of Theorem~\ref{ABprocessesTheorem} further pairs: $B(0,1)=(\frac12,\frac12)$ (which corresponds to a variant of the P\'olya-Vinogradov inequality), $AB(0,1)=(\frac16,\frac23)$, and infinitely many more, including, for example, $ABA^3B(0,1)=(\frac{11}{82},\frac{57}{82})$.

Section~\ref{DataPairsSection}, among other things, also presents the intuition behind the class $\mathbf{F}$ and Theorem~\ref{ABprocessesTheorem} and describes a typical use of our method as well as further examples of phases in $\mathbf{F}$ that naturally arise in analytic number theory.

\medskip

We step back for a moment to reflect on the analogy with the Archimedean aspect. The parallel between the subconvexity problem \eqref{ZetaSubconvexity} in the $t$-aspect and and \eqref{ChiSubconvexity} in the \textit{``depth'' aspect} (level $p^n$, $n\to\infty$) is particularly natural from the adelic point of view: one focuses on ramification at a single archimedean or non-archimedean place (or at a few places at once in hybrid bounds). The best available improvements on the bound \eqref{ZetaSubconvexity} are obtained by estimating the exponential sum $\sum_{N<n\leqslant N+M}n^{it}$ in short intervals ($M\ll(1+|t|)^{1/2+\epsilon}$). The method of exponent pairs of van der Corput \cite{vdCorput1922}, Phillips \cite{Phillips}, and Rankin \cite{RankinExponent}, a fundamental tool in the theory of exponential sums, relies on the iteration of two ``processes'', the (``Archimedean'') $A$- and $B$-process, which exploit the arithmetic structure by transforming a given exponential sum into rather different sums with different ranges of summation. Graham and Kolesnik \cite{GK} give an excellent survey of the theory of exponential sums. Note that transformations of $p$-adic exponent pairs given by Theorem~\ref{ABprocessesTheorem} formally coincide with those provided by the Archimedean $A$- and $B$-processes.

The ``$q$-analogues'' of Weyl differencing allowed previous researchers to establish the Weyl exponent in the context of estimating $L(\frac12,\chi)$ with a character $\chi$ to a powerful modulus and the associated exponential sums~\cite{Postnikov,BarbanLinnikChudakov,FujiiGallagherMontgomery,HeathBrownHybrid}. In Theorem~\ref{Aprocess}, which establishes the $A$-process as a recursive process relying on a $q$-analogue of the Weyl-van der Corput inequality (Lemma~\ref{WeylDifferencing}), we embrace a different paradigm of $f$ as a $p$-adic analytic function rather than a finite (essentially cubic in the works just referenced) polynomial. This allows us to obtain more general results, leaves us flexibility for iterative estimates, and brings to the fore the analogy with the Archimedean situation, while also presenting some serious difficulties which we overcome in the proof. Vinogradov's method was also applied by Gallagher \cite{GallagherProgressions} and Iwaniec \cite{IwaniecZerosPowerful} in the study of zero-free regions for $L(s,\chi)$ near the line $\text{Re}\,s=1$ and the prime number theorem in arithmetic progressions to powerful moduli. Note that iterations of the $A$-process alone yield exponent pairs $(k,\ell)$ in which $k$ is very small and $\ell$ is very close to 1; such estimates are suitable in ranges relevant to the behavior close to the edge of the critical strip.

Theorem~\ref{Bprocess} establishes the analogue of the $B$-process. Along with Lemmas~\ref{ImplicitFunctionTheorem} and \ref{StatPhaseFourier} and Theorem~\ref{SummationFormula} which we develop in the course of proving it, this result appears to have no nontrivial precedents in the literature. Our approach involves a careful application of $p$-adic Poisson summation, with the Fourier transform $\hat{\mathbf{e}}_f(s)$ given by a complete exponential sum as in \eqref{EquationFourierTransform}. We analyze such a sum using the $p$-adic analogue of the method of stationary phase (Lemma~\ref{StatPhase}), which expresses it as a sum of contributions over all approximate critical points. In Lemma~\ref{ImplicitFunctionTheorem}, we show that all such points indeed arise from actual, non-singular $p$-adic critical points and develop a $p$-adic implicit function theorem to express the critical points through analytic functions. This analysis culminates with Lemma~\ref{StatPhaseFourier}, in which we evaluate $\hat{\mathbf{e}}_f(s)$ in appropriate ranges. We show that contributions from all approximate critical points can be collected via the $p$-adic Gaussian, and find that, for an $f\in\mathbf{F}$, $\hat{\mathbf{e}}_f(s)$ vanishes unless $s$ lies in a certain arithmetic progression (roughly) of the form $a_0+p^{\kappa}t$, in which case an extremely handsome formula
\[ \hat{\mathbf{e}}_f\big(a_0+p^{\kappa}t\big)=\epsilon p^{(n-w+\kappa)/2}e\left(\frac{\breve{f}(t)}{p^n}\right) \]
holds, with some $\breve{f}\in\mathbf{F}$. As a particularly pleasing application, we obtain, in Theorem~\ref{SummationFormula}, a summation formula in which an exponential sum involving $f$ is related to its ``dual sum'', an exponential sum involving $\breve{f}$, with a long original sum giving rise to a short dual sum and conversely. The statement of the $B$-process, Theorem~\ref{Bprocess}, follows when the existing pair $(k,\ell)$ is applied to the dual sum. In fact, the summation formula turns out to be extremely versatile, and we use it in the proof of the $A$-process (Theorem~\ref{Aprocess}) to obtain tighter estimates.

Exponential sums of the form \eqref{GeneralSum} enter the estimation of the central value $L(1/2,\chi)$ via the approximate functional equation. As we will see in section~\ref{LFunctionSection}, every character $\chi$ modulo $p^n$ satisfies $\chi(1+p^{\kappa}t)=e\big(a\log_p(1+p^{\kappa}t)/p^n\big)$ for some $a\in\mathbb{Z}_p$, with $a\in\mathbb{Z}^{\times}_p$ corresponding to primitive characters (here, we can take $\kappa=1$ for odd $p$). After splitting the Dirichlet polynomials according to classes modulo $p^{\kappa}$ and applying a $p$-adic exponent pair $(k,\ell)$, the best value of the exponent $\theta$ which can be obtained in the estimate of Theorem~\ref{MainResult} is given by $\frac{k+\ell}2-\frac14$; see Theorem~\ref{EstimateInTermsOfkl}. (In fact, while the factor $(\log q)^{1/2}$ in Theorem~\ref{MainResult} is not needed with the present slick formulation, we keep it there so that the values of $r$ and $\theta$ arising from the $p$-adic exponent data apply verbatim without modification.) In particular, the trivial pair $(0,1)$ recovers the convexity bound $\theta=\frac14$, the pair $(\frac16,\frac23)$ gives the Weyl exponent $\theta=\frac16$, while already the pair $(\frac{11}{82},\frac{57}{82})$ gives $\theta=\frac{27}{164}\approx 0{.}1646$, breaking the Weyl exponent barrier in this family. In light of Theorem~\ref{ABprocessesTheorem}, the set of $p$-adic exponent pairs obtainable by the $p$-adic $A$- and $B$-processes coincides with the classical situation. Rankin \cite{RankinExponent} found the infimum of $(k+\ell)$ over all exponent pairs obtainable by $A$- and $B$-processes; his result gives the value of $\theta_0\approx 0{.}1645$ in Theorem~\ref{MainResult}.

With trivial modifications, our proof also yields the estimate $L(1/2+it,\chi)\ll (1+|t|)^Ap^rq^{\theta}(\log q)^{1/2}$ with $A=\tfrac54$, applicable along the entire critical line; see the remark after the proof of Theorem~\ref{EstimateInTermsOfkl} for details. A hybrid bound also subconvex in $t$ or even of sub-Weyl strength in both $t$- and $q$-aspects would be very interesting, but we do not pursue it here. 

\medskip

In addition to its intrinsic interest and the context into which it puts the method of exponential sums, the importance of Theorem~\ref{MainResult} lies in how it informs our understanding of the various aspects of the subconvexity problem (including the $t$-aspect, the ``depth'' aspect with which we are concerned, and the $q$-aspect) and of the available methods. We prefer to think of our $A$- and $B$-processes not as static estimates but as dynamical ways to \textit{transform} (possibly incurring inequalities) a sum into (possibly a number of) other sums, which can in turn be transformed time and again, exploiting and transcoding the arithmetic structure present in the original sum. In this light, the fact that the analogous steps can be used in the transformations of $p$-adic and Archimedean sums indicates a deep analogy of their built-in, ``genetic'' arithmetic structures.

From a generalist point of view (such as Selberg class), it is generally believed \cite{MichelVenkatesh} that the analytic behavior of $L$-functions is controlled in a universal fashion by~the conductors $C(\pi)$. Theorem~\ref{MainResult} points at intrinsic features of the depth aspect and helps shed light on the structure that distinguishes between those families of $L$-functions in which the Weyl subconvexity exponent $\theta=\frac1{6}$ is available through current techniques from those in which the naturally obtained exponent is Burgess's $\theta=\frac3{16}$. The universality of these exponents and techniques which allow one to break them and obtain better estimates toward the Lindel\"of hypothesis were principal research themes of a 2006 workshop at the American Institute of Mathematics~\cite{RicottaOpenProblems}. Subsequent to the current paper, in \cite{BlomerMilicevic}, Blomer and the author consider the subconvexity problem $L(1/2,f\otimes\chi)\ll_f(q^2)^{\theta+\epsilon}$ for character twists of a $\text{GL}(2)$ $L$-function, in which $\theta=\frac3{16}$ is currently the best known result in general, and develop further $p$-adic tools to obtain the Weyl exponent $\theta=\frac16$ in depth aspect and corresponding estimates for twisted sums of Hecke eigenvalues. For other recent striking examples of the distinctive r\^ole played by the square-full direction in analytic number theory, see \cite{Hiary,NelsonPitaleSaha,TemplierLargeValues,Vishe}.

The close of this introduction is a good place to open several questions suggested by our work. A number of subconvexity, nonvanishing, and moments-related problems for $L$-functions have so far found stronger answers in the $t$-aspect than in the $q$-aspect. The results of the present paper and \cite{BlomerMilicevic} indicate that the analogy with the depth aspect carries over in some of them; it will be interesting to see further ways in which it intervenes and how far it goes. Quantitatively stronger or hybrid (adelic in a sense) versions of Theorem~\ref{MainResult} would also appear seriously interesting; the author has obtained some positive results in the initial investigations in this direction. Finally, our results establish a theory of short exponential sums involving $p$-adically analytic fluctuations independent of the specific application to Theorem~\ref{MainResult}. There are many applications of the method of (Archimedean) exponential sums to problems other than estimates of $L$-functions (such as in the geometry of numbers), and our results are general enough to be appropriate analogues of the machinery that is needed to break the canonical exponents in most of the better-known of these applications; it appears extremely intriguing to investigate whether some of these questions have appropriate $p$-adic analogues.

Several notations will be used throughout the paper. For a positive integer $i$, we write
\[ (y)_i=y(y-1)\cdots(y-i+1). \]
For an $y\in\mathbb{Q}^{+}$, let $\iota(y)=\max\big(0,\ord_p(y^{-1})\big)$ and $\iota'(y)=\max\big(0,\ord_py\big)$, so that $\ord_py=\iota'(y)-\iota(y)$. We also write simply $\iota$ and $\iota'$ for $\iota(y)$ and $\iota'(y)$, respectively, when the value of $y$ is unambiguous form the context. We denote $\varepsilon(y)=1$ if $\ord_py\neq 0$ and $\varepsilon(y)=0$ if $\ord_py=0$. We write $f\ll g$ or $f=\text{O}(g)$ to denote that $|f|\leqslant Cg$ for some constant $C$, or, equivalently, that $\limsup(|f|/g)<+\infty$.

\section{Preliminaries on \texorpdfstring{$p$-adic}{p-adic} analysis}
\label{Preliminaries}

In this section, we collect facts about $p$-adic exponential, logarithmic, and power series and prove several auxiliary results related to these $p$-adic series which will be useful in our later capstone estimates. The reader is encouraged to postpone details of proofs for the second reading. Much of the pain in this section comes from the occasional need, inherent in the method of exponent pairs, to deal with power series of the form $(1+p^{\kappa}t)^y$ even when $\text{ord}_py\neq 0$, and our desire to minimize losses while doing so.

Throughout this section, all formal power series have coefficients in $\mathbb{Q}_p$ unless specified otherwise. For such a series $a(t)=\sum_{k=0}^{\infty}a_kt^k$, we follow the notation of \cite{RobertpAdicAnalysis} and denote its radius of convergence by
\begin{equation}
\label{RadiusOfConvergenceEquation}
r_a=\sup\{r\geqslant 0:\lim |a_k|_pr^k=0\}=\big(\limsup |a_k|_p^{1/k}\big)^{-1}
\end{equation}
and its growth modulus by
\[ M_ra=M_r(a)=\max |a_k|_pr^k\quad (0\leqslant r<r_a). \]
Note that we may very well have $\log_pr_a\in\mathbb{R}\setminus\mathbb{Z}$ even though each $\log_p|a_k|_p$ is an integer. We will write $M_ra\doteq |a_{k_0}|_pr^{k_0}$ if there is a unique $k_0\in\mathbb{N}$ achieving the maximum and the value of $k_0$ is clear from the context; such radii $r$ are called regular.

We record the following standard fact.

\begin{lemma}
\label{LocalInjectionLemma}
Let $f(t)=\sum_{k=0}^{\infty}a_kt^k=a_0+tf_1(t)$ be a formal power series with $a_1\neq 0$, and let $0<r<r_f$ be such that $M_rf_1\doteq|a_1|_p$. Then, for every $x$,$y$ with $|x|_p,|y|_p\leqslant r$, we have $|f(x)-f(y)|_p=|a_1|_p|x-y|_p$. In other words, for every $x$, $y$ with $|x|_p,|y|_p\leqslant r$,
\[ f(x)\equiv f(y)\pmod {p^j|a_1|_p^{-1}}\iff x\equiv y\pmod{p^j}. \]
\end{lemma}

\begin{proof}
The proof is simple. We have that, for every $k\geqslant 2$,
\begin{align*}
|a_k(x^k-y^k)|_p&=|a_k(x-y)(x^{k-1}+x^{k-2}y+\dots+y^{k-1})|_p\\
&\leqslant |a_k|_pr^{k-1}\cdot |x-y|_p<|a_1|_p|x-y|_p.
\end{align*}
Therefore,
\[ |f(x)-f(y)|_p=\left|\sum_{k=1}^{\infty}a_k(x^k-y^k)\right|_p=|a_1|_p|x-y|_p. \qedhere \]
\end{proof}

For two power series $f(t)$ and $g(t)$ such that $g(0)=0$, one can define purely formally the power series $(f\circ g)(t)=f(g(t))$ obtained by formal substitution. On the other hand, for any power series $a(t)=\sum_{k=0}^{\infty}a_kt^k$, we can define its derivative series $Da(t)=a'(t)=\sum_{k=1}^{\infty}ka_kt^{k-1}$. The usual rules for differentiation hold, including the Sum and Product rules, as well as the Chain Rule,
\begin{equation}
\label{ChainRuleEquation}
D(f\circ g)(t)=Df(g(t))Dg(t),
\end{equation}
valid for any two power series $f$ and $g$ with $g(0)=0$ \cite[page 289]{RobertpAdicAnalysis}.

We will repeatedly use the following standard proposition, which gives a sufficient condition for this substitution to correspond to numerical substitution in convergent $p$-adic power series:

\begin{lemma}
\label{NumericalSubstitution}
Let $f$ and $g$ be two convergent power series with $g(0)=0$. If $|x|<r_g$ and $M_{|x|}(g)<r_f$, then $r_{f\circ g}>|x|$ and the numerical evaluation of the composite $f\circ g$ can be made according to
\[ (f\circ g)(x)=f(g(x)). \]
\end{lemma}

\begin{proof}
This statement is from \cite[page 294]{RobertpAdicAnalysis}.
\end{proof}

In particular, consider the power series
\[ \varepsilon(x)=\exp_p(x)=\sum_{k=0}^{\infty}\frac1{k!}x^k=1+\varepsilon_0(x),\quad \lambda(x)=\log_p(1+x)=\sum_{k=1}^{\infty}\frac{(-1)^{k-1}}{k}x^k. \]
Recall that
\[ \text{ord}_p(k!)=\left\lfloor\frac kp\right\rfloor+\left\lfloor\frac k{p^2}\right\rfloor+\cdots<\frac{k}{p-1},\quad\text{so that}\quad\text{ord}_p(k!)\leqslant\frac{k-1}{p-1}. \]
It is therefore seen that
\begin{alignat*}{3}
&r_{\varepsilon}=r_{\varepsilon_0}=r_p, &\quad &M_r\varepsilon_0\doteq r &\,\,\, &\text{for all}\,\,r<r_p,\\
&r_{\lambda}=1, &\quad &M_r\lambda\doteq r &\,\,\, &\text{for all}\,\,r<r_p,
\end{alignat*}
where $r_p=p^{-\rho_p}$, $\rho_p=1/(p-1)$. Moreover, if  $\mathcal{O}$ is any complete valuation ring extension of $\mathbb{Z}_p$ (possibly $\mathcal{O}=\mathbb{Z}_p$) and $K$ is the field of fractions of $\mathcal{O}$, and if $B_r=\{t\in\mathcal{O}:|t|_p<r\}$, then $\varepsilon_0,\lambda:B_{r_p}\to B_{r_p}$ are isometries such that, according to Lemma \ref{NumericalSubstitution}, $\varepsilon_0\circ\lambda=\lambda\circ\varepsilon_0=\text{id}_{B_{r_p}}$.

Of particular interest to us will be the power series $\pi^y(x)$, defined for $y\in K^{\times}$ as
\[ \pi^y(x)=1+\pi^y_0(x)=\varepsilon(y\lambda(x))=\sum_{k=0}^{\infty}\binom ykx^k. \]
It is easy to see that the radius of convergence $r_{\pi^y}=r_{\pi^y_0}$ equals $\infty$ if $y\in\mathbb{N}_0$, $1$ if $\iota(y)=0$ and $y\not\in\mathbb{N}_0$, and $r_pp^{-\iota(y)}$ if $\iota(y)>0$. In any case, for $r<r_pp^{-\iota(y)}$, the above composition is also valid as a numerical evaluation by Lemma \ref{NumericalSubstitution}, and $M_r\pi^y_0\doteq |y|_pr$. Moreover, $\pi^y_0:B_{r_pp^{-\iota(y)}}\to B_{r_pp^{-\iota'(y)}}$ is an isometry such that $\pi^{1/y}_0\circ\pi^y_0=\text{id}_{B_{r_pp^{-\iota(y)}}}$. We write $\pi^y(x)=(1+x)^y$. We have that
\begin{equation}
\label{PowerFunctionMultiplicativity}
\big((1+x_1)(1+x_2)\big)^y=(1+x_1)^y(1+x_2)^y
\end{equation}
for all $x_1,x_2\in B_{r_pp^{-\iota(y)}}$.

In particular, the equation $(1+x)^y=1+t$ has a solution $x\in B_{r_pp^{-\iota(y)}}$ if and only if $t\in B_{r_pp^{-\iota'(y)}}$, in which case the solution is unique and given by $1+x=(1+t)^{1/y}$.

Among all power series $a(t)=\sum_{k=0}^{\infty}a_kt^k$ with coefficients $a_k\in\mathbb{Z}_p$, we consider the following subsets:
\begin{align*}
\label{ClassesOfPowerSeries}
\mathbf{I}_0(\mathbb{Z}_p)&=\left\{a(t):a_k\in\mathbb{Z}_p\,(k\geqslant 0),\,\,\lim |a_k|_p=0\right\},\\
\mathbf{I}(\mathbb{Z}_p)&=\mathbb{Z}_p+pt\mathbf{I}_0(\mathbb{Z}_p)=\left\{a(t):a_0\in\mathbb{Z}_p,\,\,a_k\in p\mathbb{Z}_p\,(k\geqslant 1),\,\,\lim|a_k|_p=0\right\},\\
\mathbf{I}^{\times}(\mathbb{Z}_p)&=\mathbb{Z}^{\times}_p+pt\mathbf{I}_0(\mathbb{Z}_p)=\left\{a(t):a_0\in\mathbb{Z}^{\times}_p,\,\,a_k\in p\mathbb{Z}_p\,(k\geqslant 1),\,\,\lim|a_k|_p=0\right\},\\
\mathbf{I}^1(\mathbb{Z}_p)&=(1+p\mathbb{Z}_p)+pt\mathbf{I}_0(\mathbb{Z}_p)\hphantom{{}^{\kappa\kappa}}=1+p\mathbf{I}_0(\mathbb{Z}_p),\\ \mathbf{I}^1_{\kappa}(\mathbb{Z}_p)&=(1+p^{\kappa}\mathbb{Z}_p)+p^{\kappa}t\mathbf{I}_0(\mathbb{Z}_p)=1+p^{\kappa}\mathbf{I}_0(\mathbb{Z}_p).
\end{align*}
We see that all power series in the ring $\mathbf{I}_0(\mathbb{Z}_p)$ define analytic functions $\mathbb{Z}_p\to\mathbb{Z}_p$, that $\mathbf{I}(\mathbb{Z}_p)$ is a subring of $\mathbf{I}_0(\mathbb{Z}_p)$, and that $\mathbf{I}^{\times}(\mathbb{Z}_p)$ is the group of invertible elements of $\mathbf{I}(\mathbb{Z}_p)$. We note for reference that obviously
\begin{alignat*}{2}
&r_a\geqslant 1 &\quad &\text{for all }a\in\mathbf{I}_0(\mathbb{Z}_p),\\
&M_r(a)\leqslant 1 &\quad &\text{for all }r\leqslant 1,\,\,a\in\mathbf{I}_0(\mathbb{Z}_p),\\
&M_r(a)\doteq 1 &\quad &\text{for all }r\leqslant 1,\,\,a\in\mathbf{I}^{\times}(\mathbb{Z}_p).
\end{alignat*}

Let $y\in\mathbb{Q}^{\times}_p$ and an integer $\kappa\geqslant 1+\iota'(2)$ be arbitrary, and let $\iota=\iota(y)$, $\iota'=\iota'(y)$, so that $\kappa+\iota'=\kappa+\iota+\text{ord}_py$. Then the power series $\pi^y_{[\kappa+\iota]}(x)=1+\pi^y_{[\kappa+\iota]0}(x)=(1+p^{\kappa+\iota}x)^y$ satisfies
\[ M_r\pi^y_{[\kappa+\iota]0}\doteq|y|_pp^{-\kappa-\iota}r=p^{-\kappa-\iota'}r\]
for every $r<p^{\kappa}r_p$ (in particular for $r=1$), and $\pi^y_{[\kappa+\iota]}$ belongs to $\mathbf{I}^1_{\kappa+\iota'}(\mathbb{Z}_p)$.

We will also consider, for any given $\lambda\in\mathbb{R}_{\geqslant 0}$, the following subspaces of $\mathbf{I}_0(\mathbb{Z}_p)$,
\begin{equation}
\label{LambdaDefinitions}
\begin{aligned}
\mathbf{I}_0[\lambda](\mathbb{Z}_p)&=\{a(t)\in\mathbf{I}_0(\mathbb{Z}_p):\text{ord}_pa_k\geqslant\lceil k\lambda\rceil\,\,(k\in\mathbb{N}_0)\},\\
\mathbf{I}^n_0[\lambda](\mathbb{Z}_p)&=t\mathbf{I}_0[\lambda](\mathbb{Z}_p),\\
\mathbf{I}^1_{\kappa}[\lambda](\mathbb{Z}_p)&=(1+p^{\kappa}\mathbb{Z}_p)+p^{\kappa}\mathbf{I}^n_0[\lambda](\mathbb{Z}_p).
\end{aligned}
\end{equation}
For example, if $\lambda\in\mathbb{N}_0$, then $\mathbf{I}_0[\lambda](\mathbb{Z}_p)$ consists of power series of the form $a(t)=a_1(p^{\lambda}t)$ for some $a_1(t)\in\mathbf{I}_0(\mathbb{Z}_p)$. It is clear that each $\mathbf{I}_0[\lambda](\mathbb{Z}_p)$ is a ring, that $\mathbf{I}^n_0[\lambda](\mathbb{Z}_p)$ is an $\mathbf{I}_0[\lambda](\mathbb{Z}_p)$-module, and that, when $\kappa\geqslant\lambda$, $\mathbf{I}^1_{\kappa}[\lambda](\mathbb{Z}_p)$ is a subgroup of $\mathbf{I}_0[\lambda](\mathbb{Z}_p)^{\times}$. It is also clear that $r_a\geqslant p^{\lambda}$ for every $a(t)\in\mathbf{I}_0[\lambda](\mathbb{Z}_p)$.

The relevance of these classes for us stems from the fact that, as is easily verfiied,
\[ \pi^y_{[\kappa+\iota]}\in
\begin{cases}\mathbf{I}^1_{\kappa+\iota'}[\kappa-\rho_p](\mathbb{Z}_p),&\text{ord}_py\neq 0,\\
\mathbf{I}^1_{\kappa+\iota'}[\kappa](\mathbb{Z}_p),&\text{ord}_py=0.\end{cases} \]
We can write $\pi^y_{[\kappa+\iota]}\in\mathbf{I}^1_{\kappa+\iota'}[\kappa-\rho_p(y)](\mathbb{Z}_p)$, where $\rho_p(y)$ equals $\rho_p$ if $\text{ord}_py\neq 0$ and $0$ otherwise.

For every $a(t)\in\mathbf{I}^1_{\kappa+\iota}(\mathbb{Z}_p)$, $a(t)=a_0+p^{\kappa+\iota}ta_1(t)$, $a_0\in(1+p^{\kappa+\iota}\mathbb{Z}_p)$, $a_1(t)\in \mathbf{I}_0(\mathbb{Z}_p)$, we can consider $a(t)^y=a_0^y\big(1+p^{\kappa+\iota}a_0^{-1}ta_1(t)\big)^y$, with the latter power defined by formal substitution. This power series $a(t)^y$ belongs to $\mathbf{I}^1_{\kappa+\iota'}(\mathbb{Z}_p)$. According to Lemma \ref{NumericalSubstitution}, the values $a(t)^y$ can be numerically evaluated as compositions for $t\in\mathbb{Z}_p$. In particular, for every two $a(t),b(t)\in\mathbf{I}^1_{\kappa+\iota}(\mathbb{Z}_p)$, we have according to \eqref{PowerFunctionMultiplicativity} the equality of values
\[ (a(t)b(t))^y=a(t)^yb(t)^y \]
for every $t\in\mathbb{Z}_p$. Consequently, both sides of this equation must also agree as power series in $\mathbf{I}^1_{\kappa+\iota'}(\mathbb{Z}_p)$. Similarly, let $a(t)\in\mathbf{I}^1_{\kappa+\iota}(\mathbb{Z}_p)$, and let $b(t)=a(t)^y\in\mathbf{I}^1_{\kappa+\iota'}(\mathbb{Z}_p)$. For every $t\in\mathbb{Z}_p$, we have an equality of values $a(t)^y=b(t)$ in $1+p^{\kappa+\iota'}\mathbb{Z}_p$. Therefore, we must also have
\[ a(t)=b(t)^{1/y} \]
as an equality of values $a(t)=b(t)^{1/y}$ in $1+p^{\kappa+\iota}\mathbb{Z}_p$ for every $t\in\mathbb{Z}_p$, and therefore also as an equality of series in $\mathbf{I}^1_{\kappa+\iota}(\mathbb{Z}_p)$.

Finally, we comment on the compositions of series of the form \eqref{LambdaDefinitions}. Suppose that $a(t)=\sum_{k=0}^{\infty}a_kt^k\in\mathbf{I}_0[\lambda_a](\mathbb{Z}_p)$ and $b(t)=t\sum_{k=0}^{\infty}b_kt^k\in\mathbf{I}^n_0[\lambda_b](\mathbb{Z}_p)$, where  $\lambda_a>0$. For every $r<p^{\lambda_b}$, $M_r(b)\leqslant r$, so that numerical substitution in $a(b(t))$ is allowed for all $r<\min(p^{\lambda_a},p^{\lambda_b})$ according to Lemma~\ref{NumericalSubstitution}. Moreover, from the formal substitution
\[ a(b(t))=\sum_{k=0}^{\infty}a_kt^k\bigg(\sum_{\ell=0}^{\infty}b_{\ell}t^{\ell}\bigg)^k=\sum_{k=0}^{\infty}\bigg(\sum_{k=k_0+\ell_1+\dots+\ell_{k_0}}a_{k_0}b_{\ell_1}\dots b_{\ell_{k_0}}\bigg)t^k, \]
it is clear that $(a\circ b)\in\mathbf{I}_0[\min(\lambda_a,\lambda_b)](\mathbb{Z}_p)$. If, in addition, $a(t)\in\mathbf{I}^1_{\kappa}[\lambda_a](\mathbb{Z}_p)$ for some $\kappa\geqslant\lambda_a$, then it follows from above that $(a\circ b)(t)\in\mathbf{I}^1_{\kappa}[\min(\lambda_a,\lambda_b)](\mathbb{Z}_p)$. In particular, if $a(t)\in\mathbf{I}^1_{\kappa+\iota}[\lambda](\mathbb{Z}_p)$, then $a(t)^y\in\mathbf{I}^1_{\kappa+\iota'}[\min(\kappa-\rho_p(y),\lambda)](\mathbb{Z}_p)$.

The following two lemmas~\ref{RaiseToPower} and \ref{PowerSeriesTaylorExpansion} will  be useful in obtaining successive convergents to the solution of an implicit function problem in Lemma~\ref{ImplicitFunctionTheorem}.

\begin{lemma}
\label{RaiseToPower}
Let $y\in\mathbb{Q}^{\times}_p$ and $\kappa\in\mathbb{N}$ be arbitrary, and let $\iota=\iota(y)$, $\iota'=\iota'(y)$. Let further $a(t)$ and $b(t)$ be two power series with $a(t)\in\mathbf{I}^1_{\kappa+\iota}[\lambda_a](\mathbb{Z}_p)$ and $b(t)\in\mathbf{I}^n_0[\lambda_b](\mathbb{Z}_p)$, $\lambda_a,\lambda_b\geqslant 0$. Then there exists a power series $\tilde{b}(t)\in\mathbf{I}^n_0[\min(\kappa-\rho_p(y),\lambda_a,\lambda_b)](\mathbb{Z}_p)$ such that
\[ \big(a(t)+p^{\kappa+\iota}b(t)\big)^y=a(t)^y+p^{\kappa+\iota'}\tilde{b}(t). \]
\end{lemma}

\begin{proof}
We may assume that $\lambda_a\leqslant\kappa+\iota$. Note that
\[ \big(a(t)+p^{\kappa+\iota}b(t)\big)^y=\Big[a(t)\big(1+p^{\kappa+\iota}a(t)^{-1}b(t)\big)\Big]^y=a(t)^y\big(1+p^{\kappa+\iota}a(t)^{-1}b(t)\big)^y. \]
As pointed above, we have that $a(t)^y\in\mathbf{I}^1_{\kappa+\iota'}[\min(\kappa-\rho_p(y),\lambda_a)](\mathbb{Z}_p)$, as well as $a(t)^{-1}\in\mathbf{I}^1_{\kappa+\iota}[\lambda_a](\mathbb{Z}_p)$, $a(t)^{-1}b(t)\in\mathbf{I}^n_0[\min(\lambda_a,\lambda_b)](\mathbb{Z}_p)$, and so
\[ \big(1+p^{\kappa+\iota}a(t)^{-1}b(t)\big)^y\in\mathbf{I}^1_{\kappa+\iota'}[\min(\kappa-\rho_p(y),\lambda_a,\lambda_b)](\mathbb{Z}_p). \]
We can thus take
\[ \tilde{b}(t)=a(t)^y\frac{\big(1+p^{\kappa+\iota}a(t)^{-1}b(t)\big)^y-1}{p^{\kappa+\iota'}}. \qedhere \]
\end{proof}

We continue with a discussion regarding formal substitution in Taylor series. We start with an easy observation \cite[Corollary on p.76]{RobertpAdicAnalysis} that, if $b_{ik}\in\mathbb{Q}_p$ ($i,k\in\mathbb{N}_0$) are such that $\lim_{\max(i,k)\to\infty}|b_{ik}|_p=0$, then $\sum_{i=0}^{\infty}\big(\sum_{k=0}^{\infty}b_{ik}\big)=\sum_{k=0}^{\infty}\big(\sum_{i=0}^{\infty}b_{ik}\big)$.

For a power series $a(t)=\sum_{k=0}^{\infty}a_kt^k$, we can also consider its $i^{\textnormal{th}}$ derivative $D_ia(t)=a^{(i)}(t)=i!\sum_{k=i}^{\infty}\binom kia_kt^{k-i}$. The series for $D_ia$ converges on the disk of convergence $D$ of $a$, and its sum agrees with the (analytic) $i^{\text{th}}$ derivative of $a$ on $D$; in fact, it is immediate from \eqref{RadiusOfConvergenceEquation} that $r_{a^{(i)}}=r_a$. Moreover, for every $x,b\in D$, we have an equality of values
\begin{equation}
\label{TaylorSeriesValues}
f(x)=\sum_{i=0}^{\infty}\frac{f^{(i)}(b)}{i!}(x-b)^i,
\end{equation}
since the order of summation can be exchanged with $b_{ik}=\binom kia_kb^{k-i}(x-b)^i$ ($k\geqslant i$) \cite[Proposition 3.22 on page 87]{Katok}.

On the other hand, suppose that $f_0,f_1,f_2,\dots$ is a sequence of formal power series in $\mathbf{I}_0(\mathbb{Z}_p)$, with $f_i(t)=\sum_{k=0}^{\infty}a_{ik}t^k$. If, for every fixed $k\in\mathbb{N}_0$, $\lim|a_{ik}|_p=0$, then we can define the formal sum $f(t)=\sum_{k=0}^{\infty}\big(\sum_{i=0}^{\infty}a_{ik}\big)t^k$. If $x\in\mathbb{Q}_p$ is such that
\begin{equation}
\label{DoubleExchangeCriterion}
\lim_{\max(i,k)\to\infty}|a_{ik}|_p|x|^k_p=0,
\end{equation}
then all $f_i(x)$ and $f(x)$ converge, and in fact we have an equality of values $f(x)=\sum_{i=0}^{\infty}f_i(x)$.

Finally, we will also consider, for an $i\in\mathbb{N}_0$, the class
\[\mathbf{I}_{0,i}[\lambda](\mathbb{Z}_p)=\{a(t)\in\mathbf{I}_0(\mathbb{Z}_p):\text{ord}_pa_k\geqslant\lceil(k+i)\lambda\rceil\,\,(k\in\mathbb{N}_0)\}, \]
and, analogously, $\mathbf{I}^n_{0,i}[\lambda](\mathbb{Z}_p)=t\mathbf{I}_{0,i}[\lambda](\mathbb{Z}_p)$. It is easy to see that, if $a(t)\in\mathbf{I}_0[\lambda](\mathbb{Z}_p)$, then $a^{(i)}(t)/i!\in\mathbf{I}_{0,i}[\lambda](\mathbb{Z}_p)$. It is also easy to see that, if $a(t)\in\mathbf{I}_{0,i_a}[\lambda_a](\mathbb{Z}_p)$ and $b(t)\in\mathbf{I}^n_{0,i_b}[\lambda_b](\mathbb{Z}_p)$, then $(a\circ b)\in a(0)+\mathbf{I}_{0,i_a+i_b}[\min(\lambda_a,\lambda_b)](\mathbb{Z}_p)$ (and the term $a(0)$ may be omitted if $i_b=0$). Also, if $a(t)\in\mathbf{I}_{0,i}[\lambda_a](\mathbb{Z}_p)$ and $b(t)\in t^j\mathbf{I}_0[\lambda_b](\mathbb{Z}_p)$, then $a(t)b(t)\in t^{\max(j-i,0)}\mathbf{I}_{0,\max(i-j,0)}[\min(\lambda_a,\lambda_b)](\mathbb{Z}_p)$.

We use these observations in the proof of the final lemma of this section.

\begin{lemma}
\label{PowerSeriesTaylorExpansion}
Let $f\in\mathbf{I}_0(\mathbb{Z}_p)$, $g,h\in\mathbf{I}^n_0(\mathbb{Z}_p)$, $u\in\mathbb{N}_0$. Then
\[ f\big(g(t)+p^uh(t)\big)=\sum_{i=0}^{\infty}\frac{f^{(i)}(g(t))}{i!}p^{ui}h(t)^i. \]
In particular, if $f\in\mathbf{I}_0[\lambda_f](\mathbb{Z}_p)$, $g\in\mathbf{I}^n_0[\lambda_g](\mathbb{Z}_p)$, $h\in\mathbf{I}^n_0[\lambda_h](\mathbb{Z}_p)$, then
\[ f\big(g(t)+p^uh(t)\big)=f(g(t))+p^uf_1(t)  \]
for some $f_1\in\mathbf{I}^n_{0,1}[\min(\lambda_f,\lambda_g,\lambda_h)](\mathbb{Z}_p)$.
\end{lemma}

\begin{proof}

We have seen that, if $a(t)\in\mathbf{I}_0(\mathbb{Z}_p)$ and $b(t)\in\mathbf{I}^n_0(\mathbb{Z}_p)$, then $(a\circ b)\in\mathbf{I}_0(\mathbb{Z}_p)$. From the discussion above, both sides of the first equality exist as formal power series in $\mathbf{I}_0(\mathbb{Z}_p)$, numerical substitution is allowed in all terms for $t\in p\mathbb{Z}_p$, and the values of both sides agree for all $t\in p\mathbb{Z}_p$ (in fact for all $t$ for which $g(t),h(t)\in\mathbb{Z}_p$ and numerical substitution is allowed); therefore, they must agree as power series.

Suppose additionally that $f\in\mathbf{I}_0[\lambda_f](\mathbb{Z}_p)$, $g\in\mathbf{I}_0^n[\lambda_g](\mathbb{Z}_p)$, and $h\in\mathbf{I}_0^n[\lambda_h](\mathbb{Z}_p)$. Since $f^{(i)}(t)/i!\in\mathbf{I}_{0,i}[\lambda_f](\mathbb{Z}_p)$, we have that $f^{(i)}(g(t))/i!\in\mathbf{I}_{0,i}[\min(\lambda_f,\lambda_g)](\mathbb{Z}_p)$, and so
\[ f_1(t)=\sum_{i=1}^{\infty}\frac{f^{(i)}(g(t))}{i!}p^{u(i-1)}h(t)^i\in\mathbf{I}^n_{0,1}[\min(\lambda_f,\lambda_g,\lambda_h)](\mathbb{Z}_p),\]
since $p^{u(i-1)}h(t)^i/t\in t^{i-1}\mathbf{I}_0[\lambda_h](\mathbb{Z}_p)$.
\end{proof}

\section{\texorpdfstring{$p$-adic}{p-adic} exponent data and pairs}
\label{DataPairsSection}

Exponential sums of the shape \eqref{GeneralSum} cannot, of course, be non-trivially estimated entirely independently of the arithmetic structure of $f$. In this section, we define a class of functions to which our method suitably applies as well as the principal parameters of our estimates, $p$-adic exponent data and $p$-adic exponent pairs, derive some of their general properties, and give examples illustrating our definitions and their typical uses. Occasionally, and for illustrative purposes only, we reference in this section statements and equations from later sections, but, of course, all actual definitions and propositions are independent of the later material. Additional useful intuition, examples, and explanations can be found in section~\ref{LFunctionSection}.

\medskip

We may, in light of \eqref{TaylorSeriesValues}, think of $f(t)$ as a power series in $t$. Of particular interest to us will be the case when $f(t)$ is a constant multiple of the $p$-adic logarithm $\log_p(1+pt)$ and $B$ is relatively short compared to $p^n$. The method we develop, however, applies to estimation of sums of type \eqref{GeneralSum} with a rather general $f$, as we discuss below. This is a very pleasing aspect of our method, although it is not entirely a matter of choice, for our recursive process produces many other $f$, in addition to the $p$-adic logarithm, which we need to be able to handle. Definition \ref{ClassesDefinition} gives a universe of power series in which we find it convenient to formulate our results.

\begin{definition}
\label{ClassesDefinition}
Let $w\in\mathbb{Z}$, $u,\kappa\in\mathbb{N}$ with $\kappa\geqslant 1+\iota'(2)$, $\lambda\in \rho_p\mathbb{N}$, $y\in\mathbb{Q}^{+}$, and let $\iota=\iota(y)$, $\iota'=\iota'(y)$, $\omega,\omega'\in\mathbb{Z}^{\times}_p$. We say that a power series $f\in\mathbb{Q}^{\times}_p\mathbf{I}_0(\mathbb{Z}_p)$ belongs to class $\mathbf{F}(w,y,\kappa,\lambda,u,\omega,\omega')$ if
\begin{equation}
\label{ClassFDefinition}
f'(t)=p^w\omega'\big(1+p^{\iota+\kappa}\omega t\big)^{-y}+p^w\gamma_0+p^{u+w}g(t)
\end{equation}
for some $\gamma_0\in\mathbb{Z}_p$ and $g\in\mathbf{I}_0[\lambda](\mathbb{Z}_p)$. We say that $f$ belongs to class $\mathbf{F}(w,y,\kappa,\lambda,u)$ if $f\in\mathbf{F}(w,y,\kappa,\lambda,u,\omega,\omega')$ for some $\omega,\omega'\in\mathbb{Z}^{\times}_p$.
\end{definition}

The condition that $\lambda\in\rho_p\mathbb{N}$ (rather than simply $\lambda\in\mathbb{R}^{+}$) is used only to obtain sharper estimates in the proof of Lemma~\ref{StatPhaseFourier} and could easily be dispensed with, but the values of $\lambda$ naturally obtained from our iterative method lie in this discrete set anyway. We will sometimes use the symbol $\infty$ in place of $\lambda$ and $u$ and say that $f$ belongs to the class $\mathbf{F}(w,y,\kappa,\infty,\infty,\omega,\omega')$ if $f$ satisfies Definition~\ref{ClassesDefinition} for arbitrarily large values of $\lambda$ and $u$, which is to say that \eqref{ClassFDefinition} holds with $g=0$.

The class $\mathbf{F}(w,y,\kappa,\lambda,u)$ is wholly unnecessarily restrictive. In fact, for each particular application of our method to an exponential sum involving a power series, say, in $\mathbf{I}_0(\mathbb{Z}_p)$, we really only need a finite list of non-vanishing conditions of the kind that are discussed, for example, in \cite{Hux05}. Such non-vanishing conditions are tedious but straightforward to write down in each particular instance. However, writing these conditions out in full generality appears very involved, and the class of functions we consider amply suffices for the subconvexity application. Our Definitions~\ref{ClassesDefinition} and Definition~\ref{ExponentPairsDefinition} should be compared with their archimedean counterparts in \cite[p. 30-31]{GK}.

Note that, for every series $f\in\mathbf{F}(w,y,\kappa,\lambda,u)$ and every $n\geqslant w$, 
\[ \mathbf{e}_f(m)=e\left(\frac{f(m)}{p^n}\right) \]
defines a function $\mathbf{e}_f:\mathbb{Z}\to\mathbb{C}$ which is periodic with period $p^{n-w}$. Indeed, by \eqref{TaylorSeriesValues}, we have, for every $q\in\mathbb{Z}$, an equality of values
\[ f\big(m+p^{n-w}q\big)=f(m)+f'(m)p^{n-w}q+\sum_{r=2}^{\infty}\frac{f^{(r)}(m)}{r!}p^{(n-w)r}q^r. \]
Since $f'(m)\in p^w\mathbb{Z}_p$, while $\ord_pr!\leqslant\lfloor(r-1)\rho_p\rfloor$ and
\[ \ord_pf^{(r)}(m)\geqslant w+\min\big((r-1)\kappa,u+\lceil (r-1)\lambda\rceil\big) \]
for every $r\geqslant 2$, we conclude that $f\big(m+p^{n-w}q\big)-f(m)\in p^n\mathbb{Z}_p$, from which the periodicity of $\mathbf{e}_f$ follows.

\medskip

In this paper, we develop machinery to estimate sums of the form \eqref{GeneralSum} whose general term $\mathbf{e}_f(m)=e\big(f(m)/p^n\big)$ is a periodic function arising from some $f\in\mathbf{F}(w,y,\kappa,\lambda,u)$. We give several examples illustrating varied situations in which such arithmetic sums arise as well as the r\^ole of various parameters in Definition~\ref{ClassesDefinition}.

Character sums
\[ S_{\chi}(M,B)=\sum_{M<m\leqslant M+B}\chi(m), \]
for a Dirichlet character $\chi$ modulo $q=p^n$ are of classical interest and of direct relevance to our subconvexity application; we discuss them in section~\ref{LFunctionSection}. For definiteness, suppose that $\chi$ is primitive. According to Lemma~\ref{ParametrizationOfCharacters}, there is an $a_0\in\mathbb{Z}_p^{\times}$ such that
\[ \chi\big(1+p^{\kappa_1}t\big)=e\left(\frac{a_0\log_p(1+p^{\kappa_1}t)}{p^n}\right) \]
for every $t\in\mathbb{Z}$, where $\kappa_1=1+\iota'(2)$. Splitting our character sum into classes modulo $p^{\kappa}$ for a suitable $\kappa\geqslant\kappa_1$ and fixing, for every $1\leqslant c\leqslant p^{\kappa}$ such that $p\nmid c$, an integer $c'$ with $cc'\equiv 1\pmod{p^n}$, we have that
\begin{equation}
\label{SwitchToInnerSum}
S_{\chi}(M,B)=\sum_{1\leqslant c\leqslant p^{\kappa},\,p\nmid c}\chi(c)\sum_{(M-c)/p^{\kappa}<t\leqslant (M+B-c)/p^{\kappa}}e\left(\frac{a_0\log_p(1+p^{\kappa}c't)}{p^n}\right).
\end{equation}
Note that, for any $\omega\in\mathbb{Z}^{\times}_p$, since $\log_p(1+p^{\kappa}\omega t)\in p^{\kappa}\mathbf{I}_0(\mathbb{Z}_p)$ and $\left[\log_p(1+p^{\kappa}\omega t)\right]'=p^{\kappa}\omega\big(1+p^{\kappa}\omega t\big)^{-1},$ we have that $\log_p\big(1+p^{\kappa}\omega t\big)\in\mathbf{F}(\kappa,1,\kappa,\infty,\infty,\omega,\omega)$. In particular, we have that the phase $f_c(t)=a_0\log_p(1+p^{\kappa}c't)$ satisfies
\begin{equation}
\label{FcClass}
f_c\in\mathbf{F}(\kappa,1,\kappa,\infty,\infty,c',a_0c'),
\end{equation}
so the inner sum above can be treated using our techniques.

We see already in this example the need for the parameter $\kappa$ in \eqref{ClassFDefinition}. A given arithmetic summand, such as $\chi(m)$ in this example, may exhibit its true local behavior as the exponential with a phase expressed by a well-behaved $p$-adic power series when restricted to a suitable $p$-adic neighborhood, such as the arithmetic progression $c+p^{\kappa}m$ with $\kappa\geqslant\kappa_1=1+\iota'(2)$. On the other hand, the phase of our exponential is also properly a polynomial in $t$, and sometimes it can be convenient to take a larger value of $\kappa$ to obtain a lower-degree polynomial. In our case, the choices $\kappa>n/2+\text{O}(1)$ and $\kappa>n/3+\text{O}(1)$ produce exponential sums with a linear and quadratic phase, respectively, but of course they also require the splitting of the original sum into more pieces; we postpone the discussion of the relative utility of such choices to section~\ref{LFunctionSection}. This first example also showcases the flexibility given by the extra parameters $w$, $\omega$, and $\omega'$ in \eqref{ClassFDefinition}. As the discussion of periodicity of $\mathbf{e}_f$ indicates, changing the value of $w$ is effectively equivalent to changing the modulus to $p^{n-w}$, so, while this flexibility could just as well be achieved by adjusting $n$, and while the value of $w$ does change through the application of $A$- and $B$-processes, we find it convenient to keep $n$ as a fixed parameter and track the changes in $w$ separately. Finally the inclusion of $\iota$ in the exponent to $p^{\iota+\kappa}$ is simply a natural normalization in light of the properties of the $p$-adic power function $\pi^y(x)$ discussed in section \ref{Preliminaries} and is responsible for the elegant statement of Lemma~\ref{ImplicitFunctionTheorem}.

Throughout our method, we think of the term $p^w\omega'\big(1+p^{\iota+\kappa}\omega t\big)^{-y}$ as the main term in \eqref{ClassFDefinition}, and we track the remaining terms to ensure that they do not interfere with the leading term. The extra flexibility afforded by allowing these smaller terms is both pleasing for the scope of our method and essential; we proceed to explain one of their sources and the r\^ole of parameters $u$ and $\lambda$ in controlling them. Our $A$-process relies on a version of Weyl differencing and reduces estimation of the sum \eqref{GeneralSum} with $f\in\mathbf{F}(w,y,\kappa,\lambda,u)$ to sums involving a phase of the form
\[ f_{\chi,h}(t)=f\big(t+p^{\chi}h\big)-f(t)=p^{\chi}hf'(t)+p^{2\chi}h^2\sum_{r=2}^{\infty}p^{(r-2)\chi}h^{r-2}\frac{f^{(r)}(t)}{r!}; \]
see Lemma~\ref{AprocessClassesOfFunctions}. For example, if $f(t)=f_c(t)=a_0\log_p(1+p^{\kappa}c't)$ as in the inner sum in \eqref{SwitchToInnerSum}, $(p^{\chi}hf'(t))'=a_0c'{}^2p^{\chi+2\kappa}h(1+p^{\kappa}c't)^{-2}$. The infinite sum contributes a secondary term (whose derivative is $p^{u+w}g(t)$ in \eqref{ClassFDefinition}) which we must keep carrying while ensuring that it does not interfere with the main term, especially in light of the Implicit Function Theorem (Lemma~\ref{ImplicitFunctionTheorem}); this separation is the r\^ole of the parameter $u$. Moreover, the quantity $u+\lfloor\lambda\rfloor-\kappa-\iota'$ turns out to control both the new value of $u$ for the phase $f_{\chi,h}$ in Lemma~\ref{AprocessClassesOfFunctions} and the success of Lemmas~\ref{ImplicitFunctionTheorem} and \ref{StatPhaseFourier}. The parameter $\lambda$ is a measure of decay of coefficients of $g(t)$ (recall that, for $\lambda\in\mathbb{N}_0$, $g_0(p^{\lambda}t)\in\mathbf{I}_0[\lambda](\mathbb{Z}_p)$ for every $g_0\in\mathbf{I}_0(\mathbb{Z}_p)$, and compare with the form of the leading term) and, in a sense, helps the secondary term keep pace with the extra factor of $p^{\kappa}$ that the main term inherits with each differentiation. 

As our third example, we consider sums of Kloosterman sums. According to Sali\'e's classical evaluation (see \cite[p.322]{IwaniecKowalski}, \cite[Sec.7]{BlomerMilicevic}), the Kloosterman sum $S(m_1,m_2,q)$, for an odd prime power $q=p^n$ ($n\geqslant 2$) and $p\nmid m_1m_2$, vanishes unless $\big(\frac{m_1m_2}p\big)=1$, in which case it is explicitly given as
\[ S(m_1,m_2,q)=\mathop{\sum\nolimits^{\ast}}_{x\bmod q}e\left(\frac{m_1\bar{x}+m_2x}q\right)=q^{1/2}\sum_{\pm}\epsilon(\pm\ell,q)e\left(\pm\frac{2\ell}q\right), \]
where $\pm\ell$ are the two points satisfying the stationary phase condition $\ell^2\equiv m_1m_2\pmod{q}$ (cf. Lemma~\ref{StatPhase}), and $\epsilon(a,p^n)$ is the explicit unit factor as in Lemma~\ref{Gauss}, which depends only on $p$, the class of $a\bmod p$, and the parity of $n$. We refer the reader to \cite{BlomerMilicevic} for a more refined discussion of $p$-adic square roots and content ourselves here with the observation that, if $\pm\ell=\pm\ell(c)$ are the solutions to $\ell^2\equiv c\pmod{p^n}$, then $\ell(c+p^{\kappa}t)\equiv\pm\ell(c)(1+p^{\kappa}c't)^{1/2}\pmod{p^n}$ for every $\kappa\geqslant 1$, $t\in\mathbb{Z}$. We thus have, for example,
\begin{align*}
\sum_{M<m\leqslant M+B}&S(1,m,q)\\
&=q^{1/2}\sum_{\pm}\sum_{1\leqslant c\leqslant p^{\kappa},\,p\nmid c}\epsilon(\pm\ell(c),q)\sum_{(M-c)/p^{\kappa}<t\leqslant(M+B-c)/p^{\kappa}}e\left(\pm\frac{f_{c,\frac12}(t)}{p^n}\right),
\end{align*}
with the phase $f_{c,\frac12}(t)=2\ell(c)(1+p^{\kappa}c't)^{1/2}$ in the class $\mathbf{F}(\kappa,\frac12,\kappa,\infty,\infty,c',\pm 2\ell(c)c')$ of Definition~\ref{ClassesDefinition}. We remark that the good analytic behavior of derivatives of solutions to the stationary phase equation is not accidental and is instead genetic to the corresponding implicit function problem.

Many other cases of complete exponential sums to prime power moduli, such as for example hyper-Kloosterman sums, similarly give rise to exponentials with $p$-adically analytic phases satisfying the conditions of Definition~\ref{ClassesDefinition}. This involves an explicit evaluation of stationary points, which leads to an implicit function problem that can, under rather general conditions, be solved within the class $\mathbf{F}$; see Lemma~\ref{ImplicitFunctionTheorem}. In particular, this procedure also ultimately powers the duality approach in the proof of the $B$-process in section~\ref{Bprocess}. Solution of the implicit function problem is another important source of the secondary term in \eqref{ClassesDefinition} in applications. For another involved and hands-on example, in which the phase $f\in\mathbf{F}$ arises from a repeated explicit evaluation by the $p$-adic method of stationary phase, see \cite{BlomerMilicevic}.

\medskip

Finally we discuss translational invariance in the classes $\mathbf{F}(w,y,\kappa,\lambda,u)$. For a power series $f\in\mathbb{Q}^{\times}_p\mathbf{I}_0(\mathbb{Z}_p)$ and any $t,M\in\mathbb{Z}_p$, we have by \eqref{TaylorSeriesValues} an equality of values
\[ f(M+t)=\sum_{i=0}^{\infty}\frac{f^{(i)}(M)}{i!}t^i. \]
We may therefore consider $f(M+t)$ as a power series, which we denote by $f_M(t)$. Note that the formal derivative of $f_M$ agrees with the translation of the derivative $f'$, that is, $(f_M)'(t)=(f')_M(t)=f'(M+t)$. The following lemma is a simple but important verification.

\begin{lemma}
\label{TranslationInvariance}
Each of the classes $\mathbf{I}_0(\mathbb{Z}_p)$, $\mathbf{I}_{0,j}[\lambda](\mathbb{Z}_p)$, and $\mathbf{F}(w,y,\kappa,\lambda,u)$ is invariant under translations, that is, if $f$ belongs to one of these classes $\mathbf{C}$, then $f_M\in\mathbf{C}$ for every $M\in\mathbb{Z}_p$.
\end{lemma}

\begin{proof}
Suppose that $f=\sum_{k=0}^{\infty}c_kt^k\in\mathbf{I}_0(\mathbb{Z}_p)$, so that $\lim|c_k|_p=0$, and $M\in\mathbb{Z}_p$. Then $f^{(i)}(M)/i!=\sum_{k=0}^{\infty}\binom{k+i}ic_{k+i}M^k$, so that
\[ \left|\frac{f^{(i)}(M)}{i!}\right|_p\leqslant\sup_{k\geqslant 0}\left|\binom{k+i}ic_{k+i}M^k\right|_p\leqslant\sup_{k\geqslant i}|c_k|_p\to 0\quad (i\to\infty). \]
This shows that $f_M$ is a power series with integral coefficients and that, in fact, $f_M\in\mathbf{I}_0(\mathbb{Z}_p)$. If, moreover, $f\in\mathbf{I}_{0,j}[\lambda](\mathbb{Z}_p)$, then the above estimate shows that
\[ \ord_p\big(f^{(i)}(M)/i!\big)\geqslant\inf_{k\geqslant i}\ord_pc_k\geqslant\inf_{k\geqslant i}\lceil\lambda(k+j)\rceil=\lceil\lambda(i+j)\rceil, \]
so that $f_M\in\mathbf{I}_{0,j}[\lambda](\mathbb{Z}_p)$ too.

Now, let $f\in\mathbf{F}(w,y,\kappa,\lambda,u,\omega,\omega')$, so that $f'$ satisfies \eqref{ClassFDefinition} with $g\in\mathbf{I}_0[\lambda](\mathbb{Z}_p)$. Then, using \eqref{PowerFunctionMultiplicativity}, we have for every $t\in B_{p^{-\kappa}}$ an equality of values
\begin{align*}
f'(M+t)
&=p^w\omega'\left(1+p^{\iota+\kappa}\omega(M+t)\right)^{-y}+p^w\gamma_0+p^{u+w}g(M+t)\\
&=p^w\omega'\left(1+p^{\iota+\kappa}\omega M\right)^{-y}\left[1+p^{\iota+\kappa}\omega\left(1+p^{\iota+\kappa}\omega M\right)^{-1}t\right]^{-y}\\
&\qquad+p^w\gamma_0+p^{u+w}g_M(t).
\end{align*}
The right-hand side of this equality is a power series which must coincide with $(f_M)'$. We have already proved that $f_M\in\mathbb{Q}^{\times}_p\mathbf{I}_0(\mathbb{Z}_p)$ and that $g_M\in\mathbf{I}_0[\lambda](\mathbb{Z}_p)$, so that $f_M\in\mathbf{F}\big(w,y,\kappa,\lambda,u,\omega(1+p^{\iota+\kappa}\omega M)^{-1},\omega'(1+p^{\iota+\kappa}\omega M)^{-y}\big)$.
\end{proof}

We now define $p$-adic exponent data and pairs. Let $P$ denote the set of prime numbers. For any sets $X$, $Y$, and any family of subsets $X_p\subset X$ ($p\in P$), let $\mathbf{J}(X_p;Y)$ be the set of all functions $g:\mathbb{Q}^{+}\times\bigsqcup_{p\in P}(\{p\}\times X_p)\to Y$ such that, for every $y\in\mathbb{Q}^{+}$, there is a finite subset $P_0(y)\subset P$ and a function $g_0:(P\setminus P_0(y))\times X\to Y$ such that $g(y,p,x)=g_0(p,x)$ for every $p\in P\setminus P_0(y)$ and every $x\in X_p$.

In particular, write $\mathbf{J}(Y):=\mathbf{J}(\emptyset;Y)$ for the set of functions $g(y,p):\mathbb{Q}^{+}\times P\to Y$ with the above properties, and $\mathbf{J}_1(Y):=\mathbf{J}(\mathbb{N}'_p\times\rho_p\mathbb{N};Y)$ (with $X=\mathbb{R}^{+}$ and $\mathbb{N}'_p=\iota'(2)+\mathbb{N}$) for the set of such functions $g(y,p,\kappa,\lambda):\mathbb{Q}^{+}\times\bigsqcup_{p\in P}(\{p\}\times\mathbb{N}'_p\times\rho_p\mathbb{N})\to Y$.

Classes $\mathbf{J}(Y)$ and $\mathbf{J}_1(Y)$ for appropriate $Y$ are suitable universes for certain components of $p$-adic exponent data in Definition~\ref{ExponentPairsDefinition}. For example, with variables keeping their meaning from Definition~\ref{ClassesDefinition}, $\kappa_0(y,p)\in\mathbf{J}(\mathbb{N})$ is the smallest value of $\kappa$ to which our datum applies, and it may well differ from its generic value for some exceptional $(y,p)$. For example, if a datum is obtained using our $A$- and $B$-processes, and if one of the iterations involves the power series $\pi_{[\kappa]}^{7y+2}(x)$, then it may be necessary to require a higher value of $\kappa$ for those pairs $(y,p)$ for which $\ord_p(7y+2)\neq 0$; when estimating \eqref{GeneralSum}, this simply corresponds to a finer initial splitting as in \eqref{SwitchToInnerSum}. We require $p$-adic exponent data to be universal in that they ultimately apply to all values of $y$ and $p$. However, as our examples demonstrate, in a typical application, we need to estimate a sum involving a phase with one specific value of $y$. What really matters, then, is that our method applies in a uniform (tightest possible) way for all primes outside a finite exceptional set (which may depend on $y$), and with a possible finite adjustment of initial conditions at the exceptional primes; this is precisely the content of our definitions, with $Y$ denoting the universe of assumed values and with $\mathbf{J}_1(Y)$ also taking into account the possible dependence of other parameters on specific values of $\kappa$ and $\lambda$. The reader interested in applications may treat  quantities in classes $\mathbf{J}(Y)$ and $\mathbf{J}_1(Y)$ simply as explicit ``expressions'' in terms of other parameters $y$, $p$, $\kappa$, $\lambda$, knowing that their form suffices for the purpose of estimating any given exponential sum to which our method applies.

We note on the side that, in all $p$-adic exponent data we produce, the functions in corresponding classes $\mathbf{J}(X_p;Y)$ actually satisfy the following stronger uniformity condition in $y$ and $p$: there is a finite set $P_0\subset P$, a finite set of non-vanishing linear forms $l_i(y)=a_iy+b_i$ ($1\leqslant i\leqslant I$), and functions $g_0:\mathbb{Z}^I\times P\times X\to Y$ and $g'_0:\mathbb{Z}^I\times X\to Y$, such that $f(y,p,x)=g_0\big((\ord_pl_i(y))_{i=1}^I,p,x\big)$ for every $y\in\mathbb{Q}^{+}$, $p\in P$, and $x\in X_p$, as well as $g_0(z,p,x)=g'_0(z,x)$ for every $z\in\mathbb{Z}^I$, $p\in P\setminus P_0$, and $x\in X_p$.

We are now ready for the main definition.

\begin{definition}
\label{ExponentPairsDefinition}
Let $\mathbf{Q}$ be the set of all quintuples
\begin{equation}
\label{Quadruple}
q=\big(k,\,\ell,\,\,r,\,\delta,\,\,(n_0,u_0,\kappa_0,\lambda_0)\big)
\end{equation}
where $k,\ell\in\mathbb{R}$, $0\leqslant k\leqslant\frac12\leqslant\ell\leqslant 1$, $r\in\mathbf{J}_1(\mathbb{R})$, $\delta\in\mathbb{R}^{+}_0$, $n_0,u_0\in\mathbf{J}_1(\mathbb{N})$, $\kappa_0\in\mathbf{J}(\mathbb{N})$, $\lambda_0\in\mathbf{J}(\mathbb{R}^{+}_0)$, and $n_0(y,p,\kappa,\lambda)>\kappa+\iota'(y)$.

We say that a quintuple $q\in\mathbf{Q}$ as in \eqref{Quadruple} is a $p$-adic exponent datum if, for every $p\in P$, $y\in\mathbb{Q}^{+}$, $w\in\mathbb{Z}$, $\kappa\in\mathbb{N}$ with $\kappa\geqslant 1+\iota'(2)$, $\lambda\in\rho_p\mathbb{N}$, $n,u\in\mathbb{N}$ such that
\[ \kappa\geqslant\kappa_0(y,p),\quad\lambda\geqslant\lambda_0(y,p),\quad n\geqslant w+n_0(y,p,\kappa,\lambda),\quad u\geqslant u_0(y,p,\kappa,\lambda), \]
and for every $f\in\mathbf{F}(w,y,\kappa,\lambda,u)$, $M\in\mathbb{Z}$, and $0<B\leqslant p^{n-w-\kappa-\iota'}$, we have the estimate
\begin{equation}
\label{SharpDef}
\sum_{M<m\leqslant M+B}e\left(\frac{f(m)}{p^n}\right)\ll p^{r}\left(\frac{p^{n-w-\kappa-\iota'}}B\right)^kB^{\ell}\big(\log p^{n-w-\kappa-\iota'}\big)^{\delta},
\end{equation}
where $r=r(y,p,\kappa,\lambda)$, and the implied constant depends only on the datum $q$.

We say that a pair
\[ \pi=(k,\ell) \]
of non-negative numbers is a $p$-adic exponent pair if $q=\big(k,\ell,r,\delta,(n_0,u_0,\kappa_0,\lambda_0)\big)$ is a $p$-adic exponent datum for some $r\in\mathbf{J}_1(\mathbb{R})$, $\delta\in\mathbb{R}^{+}_0$, $n_0,u_0\in\mathbf{J}_1(\mathbb{N})$, $\kappa_0\in\mathbf{J}(\mathbb{N})$, $\lambda_0\in\mathbf{J}(\mathbb{R}^{+}_0)$.
\end{definition}

Every $p$-adic exponent datum $q$ carries two kinds of quantities: the values of $k$, $\ell$, $r$, and $\delta$ describe the upper bound \eqref{SharpDef}, while $(n_0,u_0,\kappa_0,\lambda_0)$ can be thought of as ``initial conditions'' that control the moduli $p^n$ and the classes $\mathbf{F}(w,y,\kappa,\lambda,u)$ of phases $f$ to which this estimate applies. We emphasize that, while the implied constant in \eqref{SharpDef} may be different from one $p$-adic exponent datum to another, it is, for a given datum $q$, absolute, and \eqref{SharpDef} holds uniformly across all other parameters, including $p$, $y$, $w$, $\kappa$, $\lambda$, $n$, $u$, $f$, $M$, and $B$.

Note that $(0,1)$ is trivially a $p$-adic exponent pair, as $(0,1,0,0,(\kappa+\iota'+1,1,1+\iota'(2),\rho_p))$ is a $p$-adic exponent datum. Further, note that the estimate on the right-hand side of \eqref{SharpDef} is an increasing function of $B$, so that, when applying this estimate, we may freely use an upper bound on $B$ instead of its exact value. We also mention that any $\delta\in\mathbf{J}_1(\mathbb{R}^{+}_0)$ would suffice for applications; we ask for $\delta\in\mathbb{R}^{+}_0$ simply because this will be the case in all $p$-adic exponent data we construct.

We now describe a typical use of Definition~\ref{ExponentPairsDefinition}. The estimate \eqref{SharpDef} holds uniformly in all parameters. In a typical application, $p^n$ is the principal parameter, $B$ is a certain power of $p^n$ (depending on $\kappa$), and, upon choosing an allowable $\kappa$, the phase $f$ and hence $w$, $y$, $\lambda$, $u$ are all fixed. In a depth-aspect problem, the $p$-adic exponent pair $(k,\ell)$ controls the principal power dependence of our estimate on $p^n$, so, in practice, one first chooses the pair $(k,\ell)$ to optimize this dependence (for specific relative sizes of $B$ and $p^n$ and the desired type of result) and then considers the corresponding datum. For example, the pair $(\tfrac12,\tfrac12)$, given by the $p$-adic exponent datum \eqref{ExponentDatumHalfHalf}
\[ \omega_{1/2}=\Big(\tfrac12,\tfrac12,0,1,\big(\kappa+\iota'+1+\iota'(12),\max(\kappa-\lfloor\lambda\rfloor+\iota'+1,1),1+\iota'(4),\rho_p\big)\Big), \]
 is well suited for very long sums \eqref{GeneralSum}, yielding in \eqref{SharpDef} the upper bound
 \[ \sum_{M<m\leqslant M+B}e\left(\frac{f(m)}{p^n}\right)\ll\big(p^{n-w-\kappa-\iota'}\big)^{1/2}\log p^{n-w-\kappa-\iota'}, \]
valid for all $p^n$ and $f\in\mathbf{F}(w,y,\kappa,\lambda,u)$ with $\kappa\geqslant 1+\iota'(4)$, $\lambda\geqslant\rho_p$, $n\geqslant w+\kappa+\iota'+1+\iota'(12)$, $u\geqslant\max(\kappa-\lfloor\lambda\rfloor+\iota'+1,1)$ and for all $M\in\mathbb{Z}$ and $0<B\leqslant p^{n-w-\kappa-\iota'}$, which is uniform in $B$ and can be seen as a variant of the P\'olya-Vinogradov inequality. Section~\ref{LFunctionSection} contains a supply of explicit $p$-adic exponent data yielded by our method that one can choose from, as we do in the course of proving the sub-Weyl subconvex bound \eqref{ExplicitEstimate}. Each of these $p$-adic exponent pairs arises from $(0,1)$ by finitely many $A$- and $B$-processes, which in turn give rise to successive $p$-adic exponent data $q$. With each application, the quantities in $q$ change and become fairly complicated (cf.~the statement of Lemma~\ref{Aprocess}), but they always take a dramatically simpler form away from finitely many $p$, such as for $p\not\in\{2,3\}$, $p\nmid y$ in the case of $\omega_{1/2}$. For such generic $p$, the original sum is split as in \eqref{SwitchToInnerSum} with $\kappa\geqslant\kappa_0$ (the latter being a constant for a fixed $y$ and non-execeptional $p$), and, assuming that the (generally mild) ``separation'' conditions $\lambda\geqslant\lambda_0$ and $u\geqslant u_0$ are met, the inner sum is estimated by \eqref{SharpDef}. Since the $p$-adic exponent datum used ultimately applies to all $p$, this proof is then easily adjusted at the finitely many exceptional primes, without necessarily impacting the final result. We refer the reader to the proof of Theorem~\ref{EstimateInTermsOfkl} and the discussion around \eqref{ExplicitEstimate} for a sample execution of this approach. Finally, all these calculations simplify even further if one is willing to simply treat constants in certain exponents of $p$ (such as $\kappa$, $n_0$, $u_0$, $r$) as $\text{O}(1)$, a shortcut that we do not take but that would be perfectly acceptable in a purely depth-aspect problem (when $p$ is considered fixed).

\medskip

We proceed to comment on why the conditions $0\leqslant k\leqslant\frac12\leqslant\ell\leqslant 1$ are included in Definition~\ref{ExponentPairsDefinition} and collect some additional useful information along the way. Consider $f(t)=p^{w-\kappa-\iota}(-y+1)^{-1}\big(1+p^{\kappa+\iota}t\big)^{-y+1}$ for $y\neq 1$, and $f(t)=p^{w-\kappa}\log_p\big(1+p^{\kappa}t\big)$ for $y=1$. Let
\[ S(a)=\sum_{M<m\leqslant M+B}e\left(\frac{af(m)}{p^n}\right). \]
Note that, for $a\in\mathbb{Z}^{\times}_p$, $af(t)\in\mathbf{F}(w,y,\kappa,\infty,\infty,1,a)$. We have that
\[ \sum_{a\in(\mathbb{Z}/p^n\mathbb{Z})^{\times}}|S(a)|^2=\mathop{\sum\sum}_{M<m_1,m_2\leqslant M+B}\sum_{a\in(\mathbb{Z}/p^n\mathbb{Z})^{\times}}e\left(\frac{a(f(m_1)-f(m_2))}{p^n}\right). \]
Recall from section~\ref{Preliminaries} that $M_r\pi^0_{-y+1}\doteq |-y+1|_pr$ for all $r<r_pp^{-\iota}$ and $M_r\lambda\doteq r$ for all $r<r_p$. It follows easily that $\ord_p(f(m_1)-f(m_2))=w+\ord_p(m_1-m_2)$ for every $m_1,m_2\in\mathbb{Z}_p$. Therefore, if $B\leqslant p^{n-w-1}$ (and so certainly throughout the range $1\leqslant B\leqslant p^{n-w-\kappa-\iota'}$), the inner sum vanishes unless $m_1=m_2$, so that
\[ \sum_{a\in(\mathbb{Z}/p^n\mathbb{Z})^{\times}}|S(a)|^2=\varphi(p^n)\cdot B. \]
It follows that $|S(a)|\geqslant B^{1/2}$ for at least one $a\in(\mathbb{Z}/p^n\mathbb{Z})^{\times}$. It follows that, if an estimate of the form \eqref{SharpDef} is to hold for all $B$ in some interval $I\subseteq[1,p^{n-w-\kappa-\iota'}]$, we must have
\begin{equation}
\label{NoBetterThanSqrt}
p^r\left(\frac{p^{n-w-\kappa-\iota'}}B\right)^kB^{\ell}\big(\log p^{n-w-\kappa-\iota'}\big)^{\delta}\gg B^{1/2}
\end{equation}
throughout the entire range $B\in I$. This conclusion (``no better than square root cancellation'') will be used several times. In particular, with the choice $B=1$ and $n-w=n_0$, we have that $p^{r+(n_0-\kappa-\iota')k}(\log p^{n_0-\kappa-\iota'})^{\delta}\gg 1$. On the other hand, taking $B=p^{n-w-\kappa-\iota'}$, we see that the defining property \eqref{SharpDef} cannot hold with $\ell<\frac12$.

On the other hand, consider the behavior when $f(t)$ is as above, $M$ and $B$ are arbitrary but fixed, and $n\to\infty$. An elementary application of Dirichlet's Box Principle shows that we can find an $a\in(\mathbb{Z}/p^n\mathbb{Z})\setminus(p^n\mathbb{Z})$ such that
\begin{align*}
p^{-w}\big(af(M+1),af(M+2),&\dots,af(M+B)\big)\\
&\in \mathbb{Q}^{/B/}_p+p^{n-w}\mathbb{Z}^B_p+\big([0,p^{n-w}/\lfloor p^{n/B}\rfloor]\cap\mathbb{Z}\big)^B,
\end{align*}
where $\mathbb{Q}^{/B/}_p=\{(q,\dots,q)\in\mathbb{Q}^B_p:q\in\mathbb{Q}_p\}$. For this choice of $a$, we have that $|S(a)|=|B+\text{O}(Bp^{-n/B})|\gg B$; on the other hand, $af(t)\in\mathbf{F}(\ord_pa+w,y,\kappa,\infty,\infty,1,a|a|_p)$ and $\ord_pa\leqslant n-\lfloor n/B\rfloor$, so that $p^{n-(\ord_pa+w)-\kappa-\iota'}\geqslant p^{\lfloor n/B\rfloor-w-\kappa-\iota'}$ in \eqref{SharpDef}. Taking $n\to\infty$, we see that no estimate of the form \eqref{SharpDef} can hold with $k<0$.

We have seen how, for two different reasons (not entirely unlike the heuristics behind large sieve estimates), every $p$-adic exponent datum that is to satisfy \eqref{SharpDef} must have $k\geqslant 0$ and $\ell\geqslant\frac12$. Finally, there is no need to consider data with ($k\geqslant 0$ and) $\ell>1$, since such an estimate would be worse than that provided by the trivial datum $(0,1,0,0,(\kappa+\iota'+1,1,1+\iota'(2),\rho_p))$ in most ranges. Similarly, there is no need to consider data with $k>\frac12$ (and $\ell\geqslant\frac12$) since the estimate obtained would be worse than that provided by the first non-trivial $p$-adic exponent datum \eqref{ExponentDatumHalfHalf}.

\medskip

The following Lemma~\ref{EquivalenceSharpSmooth}, which states that the exponent datum condition may be verified (with a minimal loss) over either sharp or smooth cutoff functions, will be convenient. We will need several notations. Let $C^1_0(\mathbb{R})$ denote the set of continuously differentiable functions $h:\mathbb{R}\to\mathbb{C}$ such that $\lim_{|t|\to\infty}|t|^N(|h(t)|+|h'(t)|)=0$ for every $N\in\mathbb{N}$. For an $h\in C^1_0(\mathbb{R})$, denote
\begin{equation}
\label{DefinitionHStar}
\|h\|_{\star}=\inf_{t_0\in\mathbb{R}}\int_{-\infty}^{\infty}\big(|t-t_0|+1\big)|h'(t)|\,\text{d}t.
\end{equation}
Note that the quantity $\|h\|_{\star}$ is invariant under translations, that is, each of the translates $h_x(t)=h(t+x)$ ($x\in\mathbb{R}$) has $\|h_x\|_{\star}=\|h\|_{\star}$.

Let $\mathcal{C}=(\mathbf{C}_i)_{i\in I}$ be a family of classes $\mathbf{C}_i$ of power series in $\mathbb{Q}^{\times}_p\mathbf{I}_0(\mathbb{Z}_p)$, each of which is invariant under translations in the sense of Lemma~\ref{TranslationInvariance}, and let $0\leqslant k\leqslant\frac12\leqslant\ell\leqslant 1$, $\delta\in\mathbb{N}_0$, $n_0:I\to\mathbb{N}_0$, $w:I\to\mathbb{Z}$, $r:I\to\mathbb{R}$. We say that a triple $\tau=(k,\ell,(r,w,n_0))$ satisfies the condition $H(\delta)$ if the estimate
\[ \sum_{M<m\leqslant M+B}e\left(\frac{f(m)}{p^n}\right)\ll p^{r(i)}\left(\frac{p^{n-w(i)}}B\right)^kB^{\ell}\big(\log p^{n-w(i)}\big)^{\delta} \]
holds, with a uniform implied constant depending only on $\tau$ and $\delta$ (so, explicitly \textit{not} on $i\in I$), for every $i\in I$, every $f\in\mathbf{C}_i$, and every $n\geqslant n_0(i)$, $M\in\mathbb{Z}$, and $0<B\leqslant p^{n-w(i)}$. We will also write the above condition with $r$ and $w$ in place of $r(i)$ and $w(i)$ for brevity. We say that $\tau$ satisfies the condition $H(\delta)^{sq}$ if, additionally, the right hand side of the above bound is $\gg B^{1/2}$ uniformly for every $i\in I$ and all $0<B\leqslant p^{n-w(i)}$.

As the example most important for us, a quintuple $q=(k,\ell,r,\delta,(n_0,u_0,\kappa_0,\lambda_0))$ is a $p$-adic exponent datum if and only if, for the collection $\mathcal{C}=\{\mathbf{F}(w,y,\kappa,\lambda,u):w\in\mathbb{Z},y\in\mathbb{Q}^{+},\kappa\geqslant\kappa_0,\lambda\geqslant\lambda_0,u\geqslant u_0\}$, the triple $(k,\ell,(r,w+\kappa+\iota'(y),n_0))$ satifies the condition $H(\delta)$. We have already seen in \eqref{NoBetterThanSqrt} that, for these triples, $H(\delta)$ automatically implies $H(\delta)^{sq}$. Recall that each of the classes $\mathbf{F}(w,y,\kappa,\lambda,u)$ is invariant under translations by Lemma~\ref{TranslationInvariance}.

With the same notation, we say that $\tau$ satisfies the condition $H_{\text{sm}}(\delta)$ if the estimate
\[ \sum_{m\in\mathbb{Z}}e\left(\frac{f(m)}{p^n}\right)h\left(\frac mB\right)\ll c(h)\cdot p^{r(i)}\left(\frac{p^{n-w(i)}}B\right)^kB^{\ell}\big(\log p^{n-w(i)}\big)^{\delta} \]
holds, with a uniform implied constant depending only on $\tau$ and $\delta$ and with $c(h)$ depending only on the cutoff function $h$, for every $i\in I$, every $f\in\mathbf{C}_i$, and every $n\geqslant n_0(i)$, $0<B\leqslant p^{n-w(i)}$, and $h\in C^1_0(\mathbb{R})$. We say that $\tau$ satisfies the condition $H_{\text{sm}}^{\sharp}(\delta)$ if the above holds with
\[ c(h)=\|h\|_{\star}. \]
The conditions $H_{\text{sm}}(\delta)^{sq}$ and $H_{\text{sm}}^{\sharp}(\delta)^{sq}$ are defined analogously. Finally, denote
\[ \delta_{1/2}=\begin{cases} 1,&(k,\ell)=(\tfrac12,\tfrac12),\\ 0,&\text{else}; \end{cases}\qquad \delta_{01}=\begin{cases} 1,&(k,\ell)=(0,1),\\ 0,&\text{else}.\end{cases} \]

\begin{lemma}
\label{EquivalenceSharpSmooth}
Let $\mathcal{C}=(\mathbf{C}_i)_{i\in I}$ be a family of classes $\mathbf{C}_i$ of power series in $\mathbb{Q}^{\times}_p\mathbf{I}_0(\mathbb{Z}_p)$, each of which is invariant under translations in the sense of Lemma~\ref{TranslationInvariance}. Then, for every triple $\tau=(k,\ell,(r,w,n_0))$, $0\leqslant k\leqslant\frac12\leqslant\ell\leqslant 1$, $n_0:I\to\mathbb{N}_0$, $w:I\to\mathbb{Z}$, $r:I\to\mathbb{R}$, and for every $\delta\in\mathbb{N}_0$, we have the following implications:
\[ H(\delta)\implies H^{\sharp}_{\textnormal{sm}}(\delta)\implies H_{\textnormal{sm}}(\delta),\qquad H_{\textnormal{sm}}(\delta)^{sq}\implies H(\delta+\delta_{1/2})^{sq}. \]
\end{lemma}

\begin{proof}
Suppose that $H_\text{sm}(\delta)^{sq}$ holds. Fix a smooth, compactly supported cutoff function $\phi\in C^{\infty}_c(\mathbb{R})$ with the following properties:
\begin{itemize}
\item $0\leqslant\phi(x)\leqslant 1$ for all $x$.
\item $\phi(x)=0$ for all $x\not\in\left[0,\frac34\right]$.
\item $\phi(x)+\phi\left(\frac{1+x}2\right)=1$ for all $0\leqslant x\leqslant\frac12$.
\end{itemize}
Using $\phi$, we define smooth, compactly supported cutoff functions $\phi_i\in C^{\infty}_c(\mathbb{R})$ ($i\in\mathbb{N}_0$) as follows: let
\[ \phi_0(x)=\begin{cases} 1-\phi(|x|),&|x|\leqslant\frac12,\\ 0,&|x|>\frac12, \end{cases} \]
and, for $i\geqslant 1$, let $\phi_i(x)=\phi\left(2^{i-1}|x|-(2^{i-1}-1)\right)$. Then, for all $i\geqslant 1$,$0\leqslant\phi_i(x)\leqslant 1$ for all $x$, $\phi_i(x)=0$ for all $x$ with $|x|\not\in\left[1-\frac1{2^{i-1}},1-\frac1{2^{i+1}}\right]$, and $\phi_{i-1}(x)+\phi_i(x)=1$ for all $x$ with $|x|\in\left[1-\frac1{2^{i-1}},1-\frac1{2^i}\right]$. Therefore, the cutoff function $\tilde{\phi}_i(x)=\sum_{j=0}^i\phi_j(x)$ satisfies:
\begin{itemize}
\item $0\leqslant\tilde{\phi}_i(x)\leqslant 1$ for all $x$,
\item $\tilde{\phi}_i(x)=0$ for all $x$ with $|x|\geqslant 1-\frac1{2^{i+1}}$,
\item $\tilde{\phi}_i(x)=1$ for all $x$ with $|x|\leqslant 1-\frac1{2^i}$.
\end{itemize}

Now, let $f\in\mathbf{C}_i$, and let $n\geqslant n_0(i)$, $M\in\mathbb{Z}$, and $0<B\leqslant p^{n-w(i)}$ be given. Since the class $\mathbf{C}_i$ is closed under translations, we have that $f_{M'}\in\mathbf{C}_i$ for every $M'\in\mathbb{Z}$. Write $B=2^{\beta}C+B_1$, where $0\leqslant B_1<2^{\beta}$, and $\beta\in\mathbb{N}$ will be suitably chosen later. Then
\begin{align*}
\sum_{M<m\leqslant M+B}&e\left(\frac{f(m)}{p^n}\right)
=\sum_{m\in\mathbb{Z}}e\left(\frac{f(m)}{p^n}\right)\tilde{\phi}_{\beta-1}\left(\frac{m-M-2^{\beta-1}C}{2^{\beta-1}C}\right)+\text{O}(B_1+C)\\
&=\sum_{m\in\mathbb{Z}}e\left(\frac{f_{M+2^{\beta-1}C}(m)}{p^n}\right)\phi_0\left(\frac{m}{2^{\beta-1}C}\right)+\text{O}\left(\frac{B}{2^{\beta}}+2^{\beta}\right)\\
&\hphantom{{}={}}+\sum_{m\in\mathbb{Z}}\sum_{j=1}^{\beta-1}e\left(\frac{f_{M+2^{\beta-1}C}(m)}{p^n}\right)\phi\left(\frac{|m|-(2^{\beta-1}-2^{\beta-j})C}{2^{\beta-j}C}\right)\\
&=\sum_{m\in\mathbb{Z}}e\left(\frac{f_{M+2^{\beta-1}C}(m)}{p^n}\right)\phi_0\left(\frac{m}{2^{\beta-1}C}\right)\\
&\hphantom{{}={}}+\sum_{j=1}^{\beta-1}\sum_{m\in\mathbb{Z}}e\left(\frac{f_{M+(2^{\beta}-2^{\beta-j})C}(m)}{p^n}\right)\phi\left(\frac{m}{2^{\beta-j}C}\right)\\
&\hphantom{{}={}}+\sum_{j=1}^{\beta-1}\sum_{m\in\mathbb{Z}}e\left(\frac{f_{M+2^{\beta-j}C}(m)}{p^n}\right)\phi\left(-\frac{m}{2^{\beta-j}C}\right)+\text{O}\left(\frac{B}{2^{\beta}}+2^{\beta}\right).
\end{align*}
Since each translate of $f$ belongs to $\mathbf{C}_i$ and $2^{\beta-j}C\leqslant B\leqslant p^{n-w(i)}$, we may apply the condition $H_{\text{sm}}(\delta)$ to see that the sum of the first three summands is at most
\begin{align*}
&\ll p^r\left(\frac{p^{n-w}}{2^{\beta-1}C}\right)^k(2^{\beta-1}C)^{\ell}(\log p^{n-w})^{\delta}+\sum_{j=1}^{\beta-1}p^r\left(\frac{p^{n-w}}{2^{\beta-j}C}\right)^k(2^{\beta-j}C)^{\ell}(\log p^{n-w})^{\delta}\\
&\ll p^r\left(\frac{p^{n-w}}B\right)^kB^{\ell}(\log p^{n-w})^{\delta}(1+\delta_{1/2}\beta),
\end{align*}
recalling that $\ell\geqslant k$, with $\ell>k$ unless $k=\ell=\frac12$. Finally, we choose $\beta$ so that $2^{\beta}\asymp B^{1/2}$; then $\beta\asymp\log B\ll\log p^{n-w}$. We thus obtain
\begin{align*}
\sum_{M<m\leqslant M+B}e\left(\frac{f(m)}{p^n}\right)
&\ll p^r\left(\frac{p^{n-w}}B\right)^kB^{\ell}(\log p^{n-w})^{\delta+\delta_{1/2}}+B^{1/2}\\
&\ll p^r\left(\frac{p^{n-w}}B\right)^kB^{\ell}(\log p^{n-w})^{\delta+\delta_{1/2}},
\end{align*}
as desired, since the first term dominates in light of the condition $H(\delta)^{sq}$. This shows that $\tau$ satisfies $H(\delta+\delta_{1/2})^{sq}$, proving the implication $H_{\text{sm}}(\delta)^{sq}\implies H(\delta+\delta_{1/2})^{sq}$.

Suppose that $H(\delta)$ holds. Let $h\in C^1_0(\mathbb{R})$ be arbitrary, and let $f\in\mathbf{C}_i$, $n\geqslant n_0(i)$, $M\in\mathbb{Z}$, and $0<B\leqslant p^{n-w(i)}$ be given. Fix a $t_0\in\mathbb{R}$, and let
\[ \tilde{S}(t)=\begin{cases}
\sum_{t_0B\leqslant m\leqslant t}e\big(f(m)/p^n\big), &t>t_0B,\\
0, &t=t_0B,\\
-\sum_{t\leqslant m<t_0B}e\big(f(m)/p^n\big), &t<t_0B.
\end{cases} \]
We can break the sum defining $\tilde{S}(t)$ into at most $|t-t_0B|/B+1$ blocks of size at most $B$. Using the condition $H(\delta)$ to estimate each of the blocks, we find that
\[ \tilde{S}(t)\ll \left(\frac{|t-t_0B|}B+1\right)\cdot p^r\left(\frac{p^{n-w}}B\right)^kB^{\ell}(\log p^{n-w})^{\delta}. \]
Using summation by parts, we estimate
\begin{align*}
\sum_{m\in\mathbb{Z}}&e\left(\frac{f(m)}{p^n}\right)h\left(\frac mB\right)\\
&=\int_{\mathbb{R}}h\left(\frac tB\right)\,\text{d}\tilde{S}(t)=-\frac1B\int_{\mathbb{R}}\tilde{S}(t)h'\left(\frac tB\right)\,\text{d}t\\
&\ll\frac1B\int_{\mathbb{R}}\left(\frac{|t-t_0B|}B+1\right)\left|h'\left(\frac tB\right)\right|\,\text{d}t\cdot p^r\left(\frac{p^{n-w}}B\right)^kB^{\ell}(\log p^{n-w})^{\delta}\\
&=\int_{\mathbb{R}}\big(|t-t_0|+1\big)|h'(t)|\,\text{d}t\cdot p^r\left(\frac{p^{n-w}}B\right)^kB^{\ell}(\log p^{n-w})^{\delta}.
\end{align*}
This estimate is valid, with a uniform implied constant, for every $t_0\in\mathbb{R}$. Taking the infimum of the right-hand side over all $t_0\in\mathbb{R}$, we find that $H_{\text{sm}}^{\sharp}(\delta)$ holds. This proves that $H(\delta)\implies H^{\sharp}_{\text{sm}}(\delta)$ and completes the entire proof, since $H^{\sharp}_{\text{sm}}(\delta)\implies H_{\text{sm}}(\delta)$ is trivial.
\end{proof}

\section{\texorpdfstring{$B$-process}{B-process}}
\label{BprocessSection}

We are now ready for the proof of the $B$-process, which relies on Poisson summation to replace an exponential sum with a short dual sum, and on the method of stationary phase and the implicit function theorem to evaluate the dual sum. Although historical precedent would have us presenting the $A$-process first, we find that this order of exposition allows us to obtain tighter estimates.

The following lemma is a version of the $p$-adic analogue of the method of stationary phase and the starting point for the analysis of the Fourier transform $\hat{\mathbf{e}}_{f}(s)$ (defined below in \eqref{EquationFourierTransform}). A variant of this method (as well as of Lemma~\ref{Gauss}) can be found in \cite[sections 3.5 and 12.3]{IwaniecKowalski}, but we include it for completeness as Lemma~\ref{StatPhase} and fine-tune the statement and proof to our particular situation.

\begin{lemma}[Method of stationary phase]
\label{StatPhase}
Let $p$ be a prime, let $f\in\mathbb{Q}^{\times}_p\mathbf{I}(\mathbb{Z}_p)$ be such that $f'\in(\mathbb{Z}_p+p^{\mu}t\mathbf{I}_0(\mathbb{Z}_p))$ for some $\mu\in\mathbb{N}_0$, and let $n,j\in\mathbb{N}$ be such that $j\leqslant n-1$ and
\[ 2(n-j)+\mu\geqslant n+\iota'(2). \] Then
\[ \sum_{m\bmod p^n}e\left(\frac{f(m)}{p^n}\right)=\sum_{\substack{m\bmod p^n\\f'(m)\equiv 0\bmod{p^j}}}e\left(\frac{f(m)}{p^n}\right). \]
\end{lemma}

\begin{proof}
We can write
\[ S=\sum_{m\bmod{p^n}}e\left(\frac{f(m)}{p^n}\right)=p^{-j}\sum_{m\bmod{p^n}}\sum_{k\bmod{p^j}}e\left(\frac{f(m+p^{n-j}k)}{p^n}\right).\]
We can use Taylor's expansion \eqref{TaylorSeriesValues} to write
\[ f(m+p^{n-j}k)=f(m)+p^{n-j}f'(m)k+\sum_{r=2}^{\infty}\frac1{r!}p^{r(n-j)}f^{(r)}(m)k^r. \]
We claim that, under our conditions, all terms in the rightmost sum are divisible by $p^n$. Indeed, writing $f'(t)=b_0+\sum_{k=1}^{\infty}p^{\mu}b_kt^k$, we have $f^{(r)}(t)=p^{\mu}\sum_{k=0}^{\infty}b_{k+r-1}(k+r-1)_{r-1}t^k$, so that
\[ \text{ord}_p\left(\frac1{r!}p^{r(n-j)}f^{(r)}(m)k^r\right)\geqslant\big(2(n-j)+\mu\big)+\big((r-2)(n-j)-\text{ord}_pr\big). \]
That the right-hand side is divisible by $p^n$ is now immediate for $r=2$ and $r=3$; for $r\geqslant 4$, the claim follows from $p^{r-2}\geqslant 2^{r-2}\geqslant r$.

It follows that
\[ S=p^{-j}\sum_{m\bmod{p^n}}e\left(\frac{f(m)}{p^n}\right)\sum_{k\bmod{p^j}}e\left(\frac{f'(m)k}{p^j}\right). \]
The inner sum equals $p^j$ if $f'(m)\equiv 0\pmod{p^j}$ and vanishes otherwise. This gives the desired equality.
\end{proof}

The method of stationary phase, in some variation of that presented in Lemma~\ref{StatPhase}, goes back at least to Sali\'e \cite{SalieSP}. A simple instance of this method is the classical evaluation of the Gaussian sum, which we record for reference. We will in fact only use the most elementary case $n\in\{0,1\}$ of this Lemma ($n\in\{2,3\}$ for $p=2$).

\begin{lemma}[Gauss]
\label{Gauss}
For a prime $p$, $n\in\mathbb{N}$, and $a\in\mathbb{Z}$ with $p\nmid a$, let
\[ \tau_a(p^n)=\sum_{m\bmod{p^n}}e\left(\frac{am^2}{p^n}\right). \]
Then
\[ \tau_a(p^n)=p^{\frac{n+\iota'(2)}2}\epsilon(a,p^n), \]
where $\epsilon(a,p^n)$ is a unit factor given explicitly as
\[ \epsilon(a,p^n)=
\begin{cases}
1, & p\neq 2,\,\, 2\mid n,\\
\big(\frac ap\big), &p\equiv 1\bmod 4,\,\,2\nmid n,\\
\big(\frac ap\big)i, &p\equiv 3\bmod 4,\,\,2\nmid n;
\end{cases}
\quad
\epsilon(a,2^n)=
\begin{cases}
0, &p=2,\,\,n=1,\\
\frac{1+i^a}{\sqrt2}, &2\mid n,\,\,n\geqslant 2,\\
\big(\frac 2a\big)\frac{1+i^a}{\sqrt2}, &2\nmid n,\,\,n\geqslant 3.
\end{cases}
\]
\end{lemma}

\begin{proof}
This is adapted to our notation from \cite[Theorems 1.5.1 and 1.5.2 and Proposition 1.5.3 on page 26]{GaussJacobiSums}.
\end{proof}

In the following lemma, we develop the $p$-adic implicit function theorem which we will use to characterize the critical points in the exponential sum \eqref{EquationFourierTransform}.

\begin{lemma}[Implicit function theorem]
\label{ImplicitFunctionTheorem}
Let $f\in\mathbf{F}(w,y,\kappa,\lambda,u,\omega,\omega')$, and assume that
\[ u>\kappa-\lfloor\lambda\rfloor+\iota',\quad \tilde{\lambda}=\min(\kappa-\rho_p(y),\lambda)>0. \]
Let $\tilde{g}_0=f'(0)p^{-w}-\omega'\in\mathbb{Z}_p$. Then there is a power series $\tilde{f}\in\mathbf{I}^n_0[\tilde{\lambda}](\mathbb{Z}_p)$ such that
\begin{equation}
\label{ImplicitFunctionClasses}
\tilde{f}(t)=p^{-\iota-\kappa}\omega^{-1}\big(1+p^{\iota'+\kappa}t\big)^{-1/y}-p^{-\iota-\kappa}\omega^{-1}+p^{u+\lfloor\lambda\rfloor-\kappa-\iota'}\tilde{g}(t)
\end{equation}
for some $\tilde{g}\in\mathbf{I}^n_0[\tilde{\lambda}](\mathbb{Z}_p)$ and such that
\begin{equation}
\label{ImplicitFunctionDefiningProperty}
f'\big(\tilde{f}(t)\big)p^{-w}=\tilde{g}_0+\omega'\big(1+p^{\iota'+\kappa}t\big).
\end{equation}

Moreover, for $j\geqslant\iota'+\kappa$ and $s'\in\mathbb{Z}_p$, the congruence
\[ f'(m)p^{-w}\equiv s'\pmod{p^j} \]
has solutions $m\in\mathbb{Z}_p$ if and only if $s'\equiv \tilde{g}_0+\omega'\pmod{p^{\iota'+\kappa}}$. In this case, writing $j'=j-\iota'-\kappa$ and $s'=\tilde{g}_0+\omega'\big(1+p^{\iota'+\kappa}t\big)$ for some $t\in\mathbb{Z}_p$ unique modulo $p^{j'}$, the above congruence holds if and only if
\[ m\equiv\tilde{f}(t)\pmod{p^{j'}}. \]
\end{lemma}

\begin{proof}
Recall that
\[ f'(t)p^{-w}=\omega'\big(1+p^{\iota+\kappa}\omega t\big)^{-y}+\gamma_0+p^ug(t), \]
where $g\in\mathbf{I}_0[\lambda](\mathbb{Z}_p)$. We can write
\[ \gamma_0+p^ug(t)=\gamma_0+p^ug(0)+p^{u+\lfloor\lambda\rfloor-\kappa}p^{\kappa}tg_1(t)=\tilde{g}_0+p^{u'+\kappa}tg_1(t), \]
where $u'=u+\lfloor\lambda\rfloor-\kappa>\iota'$ and $g_1\in\mathbf{I}_0[\lambda](\mathbb{Z}_p)$. We will now construct a power series $\tilde{f}$ such that
\begin{align*}
f'(\tilde{f}(t))p^{-w}
&=\tilde{g}_0+\omega'\big(1+p^{\iota+\kappa}\omega\tilde{f}(t)\big)^{-y}+p^{u'+\kappa}\tilde{f}(t)g_1(\tilde{f}(t))\\
&=\tilde{g}_0+\omega'\big(1+p^{\iota'+\kappa}t\big),
\end{align*}
with all numerical substitutions allowed for $t\in\mathbb{Z}_p$.

Let $\omega''=\omega'{}^{-1}$. Define a sequence of power series $\tilde{f}_k$ as follows:
\begin{equation}
\label{Definitionfk}
\begin{aligned}
\tilde{f}_0(t)&=\frac{\big(1+p^{\iota'+\kappa}t\big)^{-1/y}-1}{p^{\iota+\kappa}\omega},\\
\tilde{f}_{k+1}(t)&=\frac{\left(1+p^{\iota'+\kappa}t-p^{u'+\kappa}\omega''\tilde{f}_k(t)g_1(\tilde{f}_k(t))\right)^{-1/y}-1}{p^{\iota+\kappa}\omega}\quad(k\geqslant 0).
\end{aligned}
\end{equation}

Let $\rho_k=(k+1)(u'-\iota')$. We claim that $\tilde{f}_k$ is a sequence of power series with
\[ \tilde{f}_0\in\mathbf{I}^n_0[\kappa-\rho_p(y)](\mathbb{Z}_p)\quad\text{and}\quad\tilde{f}_k\in\mathbf{I}^n_0[\tilde{\lambda}](\mathbb{Z}_p)\,\,(k\geqslant 1) \]
such that
\begin{equation}
\label{EquationSuccessiveApproximations}
\tilde{f}_{k+1}=\tilde{f}_k+p^{\rho_k}\tilde{F}_k
\end{equation}
for some $\tilde{F}_k\in\mathbf{I}^n_0[\tilde{\lambda}](\mathbb{Z}_p)$.

We prove this claim by induction on $k$. For $k=0$, $\pi^{-1/y}_{[\iota'+\kappa]}\in\mathbf{I}^1_{\iota+\kappa}[\kappa-\rho_p(y)](\mathbb{Z}_p)$ implies that $\tilde{f}_0\in\mathbf{I}^n_0[\kappa-\rho_p(y)](\mathbb{Z}_p)$. Then $g_1(\tilde{f}_0(t))\in\mathbf{I}_0[\tilde{\lambda}](\mathbb{Z}_p)$ and $\omega''\tilde{f}_0(t)g_1(\tilde{f}_0(t))\in\mathbf{I}^n_0[\tilde{\lambda}](\mathbb{Z}_p)$. Applying Lemma~\ref{RaiseToPower},
\begin{align*}
\Big(1+p^{\iota'+\kappa}t-p^{(u'+\kappa-\iota')+\iota'}\cdot{}&\omega''\tilde{f}_0(t)g_1(\tilde{f}_0(t))\Big)^{-1/y}\\
&=\big(1+p^{\iota'+\kappa}t\big)^{-1/y}+p^{u'+\kappa-\iota'+\iota}\tilde{b}(t)
\end{align*}
for some $\tilde{b}(t)\in\mathbf{I}^n_0[\tilde{\lambda}](\mathbb{Z}_p)$. We see that we can take $\tilde{F}_0(t)=\omega^{-1}\tilde{b}(t)$ in \eqref{EquationSuccessiveApproximations}.

Assume that \eqref{EquationSuccessiveApproximations} holds for some $k\in\mathbb{N}_0$; then clearly $\tilde{f}_{k+1}\in\mathbf{I}^n_0[\tilde{\lambda}](\mathbb{Z}_p)$. Moreover, according to Lemma~\ref{PowerSeriesTaylorExpansion}, we can write
\[ g_1(\tilde{f}_{k+1}(t))=g_1\big(\tilde{f}_k(t)+p^{\rho_k}\tilde{F}_k(t)\big)=g_1(\tilde{f}_k(t))+p^{\rho_k}g_2(t) \]
for some $g_2\in\mathbf{I}^n_{0,1}[\tilde{\lambda}](\mathbb{Z}_p)$. We can rearrange
\begin{align*}
\Big(1+p^{\iota'+\kappa}t-{}&p^{u'+\kappa}\omega''\tilde{f}_{k+1}(t)g_1(\tilde{f}_{k+1}(t))\Big)^{-1/y}\\
&=\Big[\big(1+p^{\iota'+\kappa}t-p^{u'+\kappa}\omega''\tilde{f}_k(t)g_1(\tilde{f}_k(t))\big)\\
&\qquad\qquad -p^{(u'+\kappa+\rho_k-\iota')+\iota'}\omega''\big(\tilde{F}_k(t)g_1(\tilde{f}_k(t))+\tilde{f}_{k+1}(t)g_2(t)\big)\Big]^{-1/y}
\end{align*}
Since $g_1(\tilde{f}_k(t))\in\mathbf{I}_0[\tilde{\lambda}](\mathbb{Z}_p)$, we have that
\begin{gather*}
1+p^{\iota'+\kappa}t-p^{u'+\kappa}\omega''\tilde{f}_k(t)g_1(\tilde{f}_k(t))\in\mathbf{I}^1_{\iota'+\kappa}[\tilde{\lambda}](\mathbb{Z}_p),\\
\tilde{F}_k(t)g_1(\tilde{f}_k(t))+\tilde{f}_{k+1}(t)g_2(t)\in\mathbf{I}^n_0[\tilde{\lambda}](\mathbb{Z}_p).
\end{gather*}
Applying Lemma~\ref{RaiseToPower}, we conclude that
\begin{align*}
&\left(1+p^{\iota'+\kappa}t-p^{u'+\kappa}\omega''\tilde{f}_{k+1}(t)g_1(\tilde{f}_{k+1}(t))\right)^{-1/y}\\
&\qquad\qquad=\Big(1+p^{\iota'+\kappa}t-p^{u'+\kappa}\omega''\tilde{f}_k(t)g_1(\tilde{f}_k(t))\Big)^{-1/y}+p^{u'+\kappa+\rho_k-\iota'+\iota}\tilde{b}_k(t)
\end{align*}
for some $\tilde{b}_k(t)\in\mathbf{I}^n_0[\tilde{\lambda}](\mathbb{Z}_p)$. We see that we can take $\tilde{F}_{k+1}(t)=\omega^{-1}\tilde{b}_k(t)$ in \eqref{EquationSuccessiveApproximations}, since $\rho_{k+1}=u'-\iota'+\rho_k$. This completes the inductive proof of \eqref{EquationSuccessiveApproximations}.

We now define
\begin{equation}
\label{DefinitiongTilde}
\tilde{g}(t)=\sum_{k=0}^{\infty}p^{\rho_k-\rho_0}\tilde{F}_k(t),\quad \tilde{f}(t)=\tilde{f}_0(t)+p^{u'-\iota'}\tilde{g}(t).
\end{equation}
In light of $u'>\iota'$, it is clear that the series converges and that $\tilde{g}(t),\tilde{f}(t)\in\mathbf{I}^n_0[\tilde{\lambda}](\mathbb{Z}_p)$. Moreover, for $r=|t|_p<p^{\tilde{\lambda}}$, we have \eqref{DefinitiongTilde} also as equalities of values, and $M_r\tilde{f}\doteq r$. We claim that $\tilde{f}$ has all desired properties; it is now immediate that \eqref{ImplicitFunctionClasses} holds.

Define $\check{F}_k=\sum_{\ell=k}^{\infty}p^{\rho_{\ell}-\rho_k}\tilde{F}_{\ell}$. Then
\[ \tilde{f}_k=\tilde{f}-p^{\rho_k}\check{F}_k \]
and $\check{F}_k\in\mathbf{I}^n_0[\tilde{\lambda}](\mathbb{Z}_p)$. The second equation in \eqref{Definitionfk} is equivalent to
\[ 1+p^{\iota+\kappa}\omega\tilde{f}_{k+1}(t)=\left(1+p^{\iota'+\kappa}t-p^{u'+\kappa}\omega''\tilde{f}_k(t)g_1(\tilde{f}_k(t))\right)^{-1/y}. \]
Applying Lemma~\ref{RaiseToPower} and Lemma~\ref{PowerSeriesTaylorExpansion} as above (with $\tilde{f}_k(t)$, $\tilde{f}(t)$, and $-\check{F}_k(t)$ in place of $\tilde{f}_{k+1}(t)$, $\tilde{f}_k(t)$, and $\tilde{F}_k(t)$, respectively), we can re-write the right-hand side of this equality to see that
\begin{align*}
1+{}&p^{\iota+\kappa}\omega\tilde{f}(t)\\
&=1+p^{\iota+\kappa}\omega\tilde{f}_{k+1}(t)+p^{\rho_{k+1}+\iota+\kappa}\omega\check{F}_{k+1}(t)\\
&=\left(1+p^{\iota'+\kappa}t-p^{u'+\kappa}\omega''\tilde{f}(t)g_1(\tilde{f}(t))\right)^{-1/y}+p^{\rho_{k+1}+\iota+\kappa}(\check{b}_k(t)+\omega\check{F}_{k+1}(t))
\end{align*}
for some $\check{b}_k(t)\in\mathbf{I}^n_0[\tilde{\lambda}](\mathbb{Z}_p)$. For $k$ large enough, this is possible only if
\[ 1+p^{\iota+\kappa}\omega\tilde{f}(t)=\left(1+p^{\iota'+\kappa}t-p^{u'+\kappa}\omega''\tilde{f}(t)g_1(\tilde{f}(t))\right)^{-1/y}. \]
This equality of series in $\mathbf{I}^1_{\iota+\kappa}(\mathbb{Z}_p)$ is equivalent to
\begin{gather*}
\left(1+p^{\iota+\kappa}\omega\tilde{f}(t)\right)^{-y}=1+p^{\iota'+\kappa}t-p^{u'+\kappa}\omega''\tilde{f}(t)g_1(\tilde{f}(t)),\\
f'\big(\tilde{f}(t)\big)=p^w\tilde{g}_0+p^w\omega'\big(1+p^{\iota'+\kappa}t\big).
\end{gather*}
According to Lemma~\ref{NumericalSubstitution}, the numerical substitution of $\tilde{f}(t)$ in $f'$ is justified for all $|t|_p<p^{\tilde{\lambda}}$.

We pass to characterizing the solutions to the congruence $f'(m)p^{-w}\equiv s'\pmod{p^j}$, that is,
\[ \tilde{g}_0+\omega'\big(1+p^{\iota+\kappa}\omega m\big)^{-y}+p^{u'+\kappa}mg_1(m)\equiv s'\pmod{p^j}, \]
where $j\geqslant\iota'+\kappa$. Recalling that $\pi^{-y}_{[\iota+\kappa]}\in\mathbf{I}^1_{\iota'+\kappa}(\mathbb{Z}_p)$, it is seen that solutions $m$ exist only if $s'=\tilde{g}_0+\omega'\big(1+p^{\iota'+\kappa}t\big)$ for some $t\in\mathbb{Z}_p$, in which case the above congruence becomes equivalent to
\[ f'(m)p^{-w}\equiv f'(\tilde{f}(t))p^{-w}\pmod{p^j}. \]
In light of $u'>\iota'$, the series $f'(t)p^{-w}=\sum_{k=0}^{\infty}a_k^{\sharp}t^k$ satisfies the conditions of Lemma~\ref{LocalInjectionLemma} for every $r<p^{\tilde{\lambda}}$, with $|a_1^{\sharp}|_p=p^{-(\iota'+\kappa)}$. In particular, an $m\in\mathbb{Z}_p$ is a solution of the above congruence if and only if
\[ m\equiv\tilde{f}(t)\pmod{p^{j'}}, \]
as announced.
\end{proof}

\begin{lemma}
\label{StatPhaseFourier}
Let $f\in\mathbf{F}(w,y,\kappa,\lambda,u,\omega,\omega')$, and assume that
\[ \min\big(n-w,\,\,u+\lfloor\lambda\rfloor\big)>\kappa+\iota',\quad\tilde{\lambda}=\min(\kappa-\rho_p(y),\lambda)>0. \]

Let
\begin{equation}
\label{EquationFourierTransform}
\hat{\mathbf{e}}_f(s)=\sum_{m\bmod{p^{n-w}}}e\left(\frac{f(m)p^{-w}-sm}{p^{n-w}}\right),
\end{equation}
and let $\varepsilon_{\lambda}=\lfloor\lambda\rfloor-\lceil\tilde{\lambda}\rceil$. Then, assuming additional conditions listed below if $p\in\{2,3\}$, there exists a power series $\breve{f}\in\mathbf{F}(\breve{w},y^{-1},\kappa,\tilde{\lambda},\breve{u},1,-\omega'\omega^{-1})$ with $\breve{f}'\in p^{w+\iota'+\kappa}\mathbf{I}^n_0[\tilde{\lambda}](\mathbb{Z}_p)$, where
\[ \breve{w}=w+\textnormal{ord}_py,\quad \breve{u}=u+\varepsilon_{\lambda}-\textnormal{ord}_py, \]
and an $\epsilon\in\mathbb{C}$, $|\epsilon|=1$ such that
\[ \hat{\mathbf{e}}_f(s)=\hat{\mathbf{e}}_f\Big(\tilde{g}_0+\omega'\big(1+p^{\iota'+\kappa}t\big)\Big)=\epsilon p^{(n-w+\iota'+\kappa)/2}e\left(\frac{\breve{f}(t)}{p^n}\right) \]
if $s=\tilde{g}_0+\omega'\big(1+p^{\iota'+\kappa}t\big)$ for some $t\in\mathbb{Z}_p$, and $\hat{\mathbf{e}}_f(s)=0$ otherwise. The unit $\epsilon$ depends only on $\omega$, $\omega'$, $y$, $p$, the parity of $n-w+\iota'+\kappa-\iota'(2)$, and, for $p=2$ only, on $p^{u-\kappa-\iota'}g'(0)$; in particular, it is independent of $s$.

In the case $p\in\{2,3\}$, we must make additional assumptions. Let $\nu\in\{0,1\}$ be the residue of $n-w+\iota'+\kappa-\iota'(2)$ modulo $2$, let $n_1=n-w-\kappa-\iota'-1$, let $\kappa+\iota'(y+1)=\nu+2\iota'(2)+\kappa_1$. Then assume additionally that
\[ \kappa+\iota'(y+1)\geqslant\nu+2\iota'(2),\quad n_1\geqslant 2\iota'(2)\nu,\quad n_1+\kappa_1\geqslant\iota'(3). \]
These assumptions are automatically satisfied if $\kappa\geqslant 1+\iota'(4)$ and $n-w>\iota'+\kappa+\iota'(12)$, or if $\kappa\geqslant 1+\iota'(12)$ and $n-w>\iota'+\kappa+\iota'(4)$.
\end{lemma}

\begin{proof}
In light of $\pi_{[\iota+\kappa]}^{-y}\in\mathbf{I}^1_{\iota'+\kappa}(\mathbb{Z}_p)$ and $u+\lceil\lambda\rceil>\iota'+\kappa$, we have that $p^{-w}f'(t)\in\mathbb{Z}_p+p^\mu t\mathbf{I}_0(\mathbb{Z}_p)$, with $\mu=\iota'+\kappa$. Write
\[ n-w+\iota'+\kappa-\iota'(2)=2j+\nu, \]
with $j\in\mathbb{N}$ and $\nu\in\{0,1\}$. Note that, under our assumptions,
\[ \iota'+\kappa\leqslant j<n-w, \]
as well as
\[ 2(n-w-j)+(\iota'+\kappa)\geqslant n-w+\iota'(2). \]

This shows that all conditions are satisfied for an application of Lemma~\ref{StatPhase} to \eqref{EquationFourierTransform}. According to Lemma~\ref{StatPhase}, the summation in \eqref{EquationFourierTransform} can be restricted to indices $m$ for which
\begin{equation}
\label{CongruenceCondition}
f'(m)p^{-w}\equiv s\pmod{p^j}.
\end{equation}

All conditions are also satisfied for an application of Lemma~\ref{ImplicitFunctionTheorem}. Let $\tilde{f}$ be the power series whose existence is established there, and let $j'=j-\iota'-\kappa$. According to Lemma~\ref{ImplicitFunctionTheorem}, we have that indices $m$ satisfying \eqref{CongruenceCondition} exist if and only if
\[ s=\tilde{g}_0+\omega'\big(1+p^{\iota'+\kappa}t\big) \]
for some $t\in\mathbb{Z}_p$ (with congruence classes of $s\bmod{p^j}$ for which \eqref{CongruenceCondition} is solvable in one-to-one correspondence with congruence classes of $t\bmod{p^{j'}}$), in which case an $m\bmod{p^{n-w}}$ satisfies \eqref{CongruenceCondition} if and only if $m=\tilde{f}(t)+p^{j'}q$ for some $q\bmod{p^{n-w-j'}}$.

Consider two functions $\breve{f}(t,q):\mathbb{Z}_p\times\mathbb{Z}_p\to\mathbb{Z}_p$ and $\breve{f}(t):\mathbb{Z}_p\to\mathbb{Z}_p$ defined by their pointwise values as
\begin{gather}
\breve{f}(t,q)=f\big(\tilde{f}(t)+p^{j'}q\big)-p^w\Big(\tilde{g}_0+\omega'\big(1+p^{\iota'+\kappa}t\big)\Big)\big(\tilde{f}(t)+p^{j'}q\big),\nonumber\\
\breve{f}(t)=\breve{f}(t,0)=f(\tilde{f}(t))-p^w\Big(\tilde{g}_0+\omega'\big(1+p^{\iota'+\kappa}t\big)\Big)\tilde{f}(t).\label{DefinitionFBreve}
\end{gather}
With this notation, we have proved so far that
\begin{equation}
\label{IntermediateExpression}
\hat{\mathbf{e}}_f(s)=\hat{\mathbf{e}}_f\Big(\tilde{g}_0+\omega'\big(1+p^{\iota'+\kappa}t\big)\Big)=\sum_{q\bmod{p^{n-w-j'}}}e\left(\frac{\breve{f}(t,q)}{p^n}\right)
\end{equation}
if $s=\tilde{g}_0+\omega'\big(1+p^{\iota'+\kappa}t\big)$ for some $t\in\mathbb{Z}_p$, and $\hat{\mathbf{e}}_f(s)=0$ otherwise.

Using the Taylor expansion \eqref{TaylorSeriesValues}, we obtain, for every $t,q\in\mathbb{Z}_p$, an equality of values
\begin{align*}
\breve{f}(t,q)&=\sum_{r=0}^{\infty}\frac{f^{(r)}(\tilde{f}(t))}{r!}\big(p^{j'}q\big)^r-p^w\Big(\tilde{g}_0+\omega'\big(1+p^{\iota'+\kappa}t\big)\Big)\big(\tilde{f}(t)+p^{j'}q\big)\\
&=\breve{f}(t)+\Big[f'(\tilde{f}(t))-p^w\Big(\tilde{g}_0+\omega'\big(1+p^{\iota'+\kappa}t\big)\Big)\Big]p^{j'}q+\sum_{r=2}^{\infty}\frac{f^{(r)}(\tilde{f}(t))}{r!}p^{rj'}q^r\\
&=\breve{f}(t)+\frac12f''(\tilde{f}(t))p^{2j'}q^2+\sum_{r=3}^{\infty}\frac{f^{(r)}(\tilde{f}(t))}{r!}p^{rj'}q^r,
\end{align*}
recalling the defining property \eqref{ImplicitFunctionDefiningProperty} of $\tilde{f}$. Note that
\begin{equation}
\label{RthDerivative} 
f^{(r)}(t)p^{-w}=\omega'(-y)_{r-1}\big(p^{\iota+\kappa}\omega\big)^{r-1}\big(1+p^{\iota+\kappa}\omega t\big)^{-y-r+1}+p^ug^{(r-1)}(t).
\end{equation}
Since $u+\lceil\lambda\rceil>\iota'+\kappa$, we have that
\begin{align*}
\nu_2:=\text{ord}_p\bigg(\frac12f''(\tilde{f}(t))p^{2j'}\bigg)&=2j'+w+\kappa+\iota'-\iota'(2)\\
&=2j+w-\kappa-\iota'-\iota'(2)\\
&=n-\nu-2\iota'(2).
\end{align*}

We now consider the remaining infinite sum $E$ in the Taylor expansion for $\breve{f}(t,q)$ and set conditions under which $\ord_pE\geqslant n$ holds. We have that, for $r\geqslant 3$,
\begin{align*}
\text{ord}_p&\bigg(\frac{f^{(r)}(\tilde{f}(t))}{r!}p^{rj'}q^r\bigg)-w\\
&\geqslant\min\big((r-1)\kappa+rj'+\iota'+\iota'(y+1)-\text{ord}_pr!,u+\lceil(r-1)\lambda\rceil+rj'-\text{ord}_pr\big).
\end{align*}
We now carefully (and tediously) examine each of the two terms in the above expression, which we denote by $\nu_{r,1}(p)$ and $\nu_{r,2}(p)$. (This examination is not pleasant. At the first reading, the reader is encouraged to skip it or consider the case $p\not\in\{2,3\}$, for which we will see that no further assumptions are needed.)

Denote temporarily $\iota''=\iota'(y+1)$. Using that $\text{ord}_pr!\leqslant (r-1)/(p-1)$, we have that, for $r\geqslant 3$,
\begin{align*}
\nu_{r,1}(p)&:=(r-1)\kappa+rj'+\iota'+\iota''-\text{ord}_pr!\\
&\geqslant \lceil(r-1)\big(\kappa+j'-\rho_p\big)\rceil+j'+\iota'+\iota''\\
&\geqslant\lceil 2\kappa+3j'-2\rho_p\rceil+\iota'+\iota''\\
&=2\kappa+3j'+\iota'+\iota''-\iota'(12).
\end{align*}
In fact, we have the slightly stronger estimate
\[ \nu_{r,1}(p)\geqslant 2\kappa+3j'+\iota'+\iota''-\iota'(6). \]
Namely, when $p=2$, this follows by direct verification for $r=3$ and from $\nu_{r,1}(2)\geqslant\lceil 3\kappa+4j'-3\rho_2\rceil+\iota'+\iota''$ and $\kappa\geqslant 2$ for $r\geqslant 4$.

On the other hand,
\begin{align*}
\nu_{r,2}(p)&:=u+\lceil (r-1)\lambda\rceil+rj'-\text{ord}_pr\\
&\geqslant u+rj'+\lceil (r-1)(\lambda-\rho_p)\rceil\\
&\geqslant u+3j'+\lceil 2(\lambda-\rho_p)\rceil.
\end{align*}
In fact, recalling also that $\lambda\in\mathbb{N}$ for $p=2$, we have that $\nu_{3,2}(2)=u+3j'+2\lambda$ and $\nu_{4,2}(2)=u+4j'+3\lambda-2$; since $\ord_2r\leqslant r-5$ for $r\geqslant 5$, we also have that $\nu_{r,2}(2)\geqslant u+5j'+4\lambda>\nu_{3,2}(2)$ for all $r\geqslant 5$, so that
\[ \nu_{r,2}(p)\geqslant u+3j'+\hat{\nu}_2,\quad
\hat{\nu}_2:=\begin{cases}
\lceil 2(\lambda-\rho_p)\rceil, &p\geqslant 3,\\
\min\left(2\lambda,j'+3\lambda-2\right), &p=2.
\end{cases} \]

Summing up, we have that
\[ \text{ord}_pE\geqslant\nu_E:=\min\big(2\kappa+3j'+\iota'+\iota''-\iota'(6),u+3j'+\hat{\nu}_2\big)+w. \]

This gives us two conditions that need to be met for $\nu_E\geqslant n$. The first is that
\begin{align*}
2\kappa&+3j'+\iota'+\iota''-\iota'(6)\\
&=3j-3\iota'-\kappa+\iota'+\iota''-\iota'(6)\\
&=\tfrac32\big(n-w+\iota'+\kappa-\iota'(2)\big)-\tfrac32\nu-3\iota'-\kappa+\iota'+\iota''-\iota'(6)\geqslant n-w,
\end{align*}
which is equivalent to
\[ \big(n-w-\iota'+\kappa-\iota'(2)\big)-\nu+2\iota''\geqslant 2\nu+\iota'(16\cdot 9). \]
By parity considerations, this inequality will be satisfied whenever
\begin{equation}
\label{NuECondition1}
(n-w)+\kappa+2\iota''\geqslant \iota'+2\nu+\iota'(32\cdot 9).
\end{equation}
In light of $n-w=\kappa+\iota'+1+n_1$, $n_1\geqslant 0$, the above is satisfied whenever
\[ 2\kappa+2\iota'(y+1)+1+n_1\geqslant 2\nu+\iota'(32\cdot 9), \]
and this is automatically satisfied for $p\not\in\{2,3\}$. For $p\in\{2,3\}$, substituting $\kappa+\iota''=\nu+2\iota'(2)+\kappa_1$, $\kappa_1\geqslant 0$, the above inequality reads as
\[ 1+n_1+2\kappa_1\geqslant \iota'(18), \]
which is trivially satisfied in light of the condition that $n_1+\kappa_1\geqslant\iota'(3)$.

The other condition for $\nu_E\geqslant n$ is that
\begin{gather*}
u+3j'+\hat{\nu}_2=u+\tfrac32\big(n-w+\iota'+\kappa-\iota'(2)\big)-\tfrac32\nu-3\iota'-3\kappa+\hat{\nu}_2\geqslant n-w\\
(n-w-3\iota'-3\kappa-3\iota'(2))-\nu+2u+2\hat{\nu}_2\geqslant 2\nu.
\end{gather*}
Again, by parity considerations, this inequality will be satisfied whenever
\begin{equation}
\label{NuECondition2}
(n-w)+2u+2\hat{\nu}_2\geqslant 3\iota'+3\kappa+2\nu+\iota'(8).
\end{equation}
We first comment how \eqref{NuECondition2} is always satisfied for $p\geqslant 3$. Indeed, for $p\geqslant 3$ we have that
\[ \hat{\nu}_2=\lceil 2(\lambda-\rho_p)\rceil\geqslant\lfloor\lambda\rfloor; \]
this is trivially true if $\lambda=\rho_p$ and follows from $2(\lambda-\rho_p)\geqslant\lambda$ if $\lambda\geqslant 2\rho_p$. The inequality \eqref{NuECondition2} now follows from $n-w\geqslant\iota'+\kappa+1$ and $u+\lfloor\lambda\rfloor\geqslant\iota'+\kappa+1$.

Verifying the condition \eqref{NuECondition2} for $p=2$ involves checking all cases. The above argument clearly applies if $\hat{\nu}_2=2\lambda$. If $\hat{\nu}_2=j'+3\lambda-2$, then (substituting for $j'$ as above), \eqref{NuECondition2} reads as
\begin{equation}
\label{VerySpecialCase}
\begin{gathered}
(n-w)+2u+2j-2\iota'-2\kappa+6\lambda-4\geqslant 3\iota'+3\kappa+2\nu+3\\
2(n-w)+2u+6\lambda\geqslant 4\iota'+4\kappa+3\nu+8,
\end{gathered}
\end{equation}
and this follows immediately in light of $n-w\geqslant\iota'+\kappa+1+2\iota'(2)\nu$.

\medskip

Having checked that $\nu_E\geqslant n$, we conclude that
\[ \breve{f}(t,q)\equiv\breve{f}(t)+\frac12f''(\tilde{f}(t))p^{2j'}q^2\pmod{p^n}. \]
Writing $\tilde{\nu}=n-\nu_2=\nu+2\iota'(2)$, the summation in the intermediate stationary phase expression \eqref{IntermediateExpression} becomes
\begin{align*}
\hat{\mathbf{e}}_f\Big(\tilde{g}_0+\omega'\big(1+p^{\iota'+\kappa}t\big)\Big)
&=\sum_{q\bmod{p^{n-w-j'}}}e\left(\frac{\breve{f}(t)+\frac12f''(\tilde{f}(t))p^{2j'}q^2}{p^n}\right)\\
&=p^{(n-w-j'-\tilde{\nu})}e\left(\frac{\breve{f}(t)}{p^n}\right)\sum_{q\bmod{p^{\tilde{\nu}}}}e\left(\frac{\frac12f''(\tilde{f}(t))p^{2j'-\nu_2}q^2}{p^{\tilde{\nu}}}\right).
\end{align*}

The remaining sum can be evaluated by Lemma~\ref{Gauss} as
\[ p^{(\tilde{\nu}+\iota'(2))/2}\epsilon\left(\frac12f''(\tilde{f}(t))p^{2j'-\nu_2},p^{\tilde{\nu}}\right), \]
where, for an odd prime $p$, $\epsilon\big(a,p^{\tilde{\nu}}\big)$ depends on $p$, the parity of $\tilde{\nu}$, and the class of $a\bmod p$ only, while for $p=2$ and $\tilde{\nu}\in\{2,3\}$, $\epsilon\big(a,2^{\tilde{\nu}}\big)$ also depends on the class of $a\bmod{2^{\tilde{\nu}}}$. We have already seen (compare \eqref{RthDerivative} for $r=2$) that, for an odd $p$ (when $\tilde{\nu}\in\{0,1\}$),
\[ a=\frac12f''(\tilde{f}(t))p^{2j'-\nu_2}\equiv\omega\omega'(-y/2)|y/2|_p\pmod{p^{\tilde{\nu}}}. \]

In the case $p=2$ (when $\tilde{\nu}\in\{2,3\}$), considering the power $\nu^{+}_2$ of $p$ in non-constant terms in \eqref{RthDerivative}, we find that
\begin{align*}
\nu^{+}_2-\nu_2
&\geqslant\min\big(\kappa+\iota'',u+\lceil 2\lambda\rceil+\iota'(2)-\kappa-\iota',u+\lceil 3\lambda\rceil-\kappa-\iota'\big)\\
&\geqslant\min\big(\nu+2\iota'(2),3\big)=\tilde{\nu},
\end{align*}
and therefore
\[ a=\frac12f''(\tilde{f}(t))p^{2j'-\nu_2}\equiv\omega\omega'(-y)|y|_p+p^{u-\kappa-\iota'}g'(0)\pmod{p^{\tilde{\nu}}} \]
in this case.

We conclude that, in any case,
\begin{align*}
\hat{\mathbf{e}}_f\Big(\tilde{g}_0+\omega'&\big(1+p^{\iota'+\kappa}t\big)\Big)\\
&=p^{\tilde{n}/2}\epsilon\left(\omega\omega'(-y/2)|y/2|_p+p^{u-\kappa-\iota'}g'(0),p^{\tilde{\nu}}\right)e\left(\frac{\breve{f}(t)}{p^n}\right),
\end{align*}
where
\begin{align*}
\tilde{n}&=2\big(n-w-j'-\tilde{\nu}\big)+\big(\tilde{\nu}+\iota'(2)\big)\\
&=2(n-w)-2j+2\iota'+2\kappa-\nu-\iota'(2)\\
&=n-w+\iota'+\kappa.
\end{align*}

Note that the equation \eqref{DefinitionFBreve} also defines $\breve{f}(t)$ as a formal power series; since $r_f=r_{f'}\geqslant p^{\tilde{\lambda}}$, $\tilde{f}(t)\in\mathbf{I}^n_0[\tilde{\lambda}](\mathbb{Z}_p)$, and, for every $r<p^{\tilde{\lambda}}$, $M_r\tilde{f}\doteq r$, the numerical substitution of $\tilde{f}(t)$ in \eqref{DefinitionFBreve} is allowed for $|t|_p<p^{\tilde{\lambda}}$. To complete the proof of Lemma~\ref{StatPhaseFourier}, it remains to prove that $\breve{f}$ belongs to the announced classes. According to the Chain Rule \eqref{ChainRuleEquation}, we obtain from \eqref{DefinitionFBreve}, \eqref{ImplicitFunctionDefiningProperty}, and \eqref{ImplicitFunctionClasses} that
\begin{align*}
\breve{f}'(t)&=f'(\tilde{f}(t))\tilde{f}'(t)-p^w\Big(\tilde{g}_0+\omega'\big(1+p^{\iota'+\kappa}t\big)\Big)\tilde{f}'(t)-\omega'p^{w+\iota'+\kappa}\tilde{f}(t)\\
&=-\omega'p^{w+\iota'+\kappa}\tilde{f}(t)\\
&=-\omega'\omega^{-1}p^{w+\ord_py}\big(1+p^{\iota'+\kappa}t\big)^{-1/y}+\omega'\omega^{-1}p^{w+\ord_py}-\omega'p^{w+u+\lfloor\lambda\rfloor}\tilde{g}(t).
\end{align*}
This clearly shows that $\breve{f}'\in p^{w+\iota'+\kappa}\mathbf{I}^n_0[\tilde{\lambda}](\mathbb{Z}_p)$. Recalling Definition~\ref{ClassesDefinition}, we have that, indeed, $\breve{f}\in\mathbf{F}(\breve{w},1/y,\kappa,\tilde{\lambda},\breve{u},1,-\omega'\omega^{-1})$ with
\[ \breve{w}=w+\ord_py,\quad \breve{u}=u+\lfloor\lambda\rfloor-\lceil\tilde{\lambda}\rceil-\text{ord}_py, \]
as announced.
\end{proof}

We would like to point out the following feature of \eqref{EquationFourierTransform}, which is a complete exponential sum modulo $p^{n-w}$. That such exponential sums reduce to sums over the (approximate) critical points is classical; however, in general, one also encounters contributions from singular critical points, and these can be very difficult to evaluate or estimate. An extremely important feature of our definition of the class $\mathbf{F}$ is that it guarantees that we never encounter singular points, while still being sufficiently broad to cover all cases of interest for the estimation of short character sums. It is this feature that allows for the handsome, compact looks of the result of Lemma~\ref{StatPhaseFourier}, the main thrust of whose proof is to explicate the $p$-adic implicit function $\tilde{f}$ in a neighborhood of a non-singular critical point and collect all contributions through explicit computations with $p$-adic Gaussians.

We also remark that many of the conditions included in Lemma~\ref{StatPhaseFourier} for $p\in\{2,3\}$ can be relaxed or altogether dropped by allowing higher-order terms and directly evaluating the resulting sums. For example, the condition $n-w>\iota'+\kappa$ actually suffices for \eqref{VerySpecialCase} in the case $p=2$. Namely, \eqref{VerySpecialCase} also holds if $n-w=\iota'+\kappa+1$ (when $\nu=0$), or if $u+\lambda\geqslant\iota'+\kappa+2$, or if $\lambda\geqslant 2$. In the remaining case $n-w=\iota'+\kappa+2$, $\nu=1$, $\tilde{\nu}=3$, $j'=0$, $\lambda=1$, $u=\iota'+\kappa$, we incur the extra term
\[ \frac{f^{(iv)}(\tilde{f}(t))}{4!}p^{4j'}q^4\equiv A_2p^{n-1}q^4\pmod{p^n}, \]
where $A_2=p^{u+w-n+1}g^{(iii)}(0)/4!\in\mathbb{Z}_2$, and the summation in \eqref{IntermediateExpression} becomes
\[ \hat{\mathbf{e}}_f\Big(\tilde{g}_0+\omega'\big(1+p^{\iota'+\kappa}t\big)\Big)=p^{n-w-3}e\left(\frac{\breve{f}(t)}{p^n}\right)\sum_{q\bmod 8}e\left(\frac{aq^2+4A_2q^4}8\right). \]
Since $4A_2q^4\equiv 4A_2q^2\pmod 8$, the inner sum can again be evaluated by Lemma~\ref{Gauss}, yielding Lemma~\ref{StatPhaseFourier} with only a change in the value of $\epsilon$. However, we chose the current formulation, which reflects conditions that guarantee a purely quadratic-term expansion at each stationary point.

In the applications of Lemma~\ref{StatPhaseFourier}, we will simply assume that $\kappa\geqslant 1+\iota'(4)$ and $n-w>\iota'+\kappa+\iota'(12)$; while these conditions can occasionally be somewhat relaxed, we will not be concerned with this aspect, which is anyway relevant for $p\in\{2,3\}$ only.

\medskip

We collect the fruits of our labor in the following summation formula.

\begin{theorem}[Summation Formula]
\label{SummationFormula}
Let $f\in\mathbf{F}(w,y,\kappa,\lambda,u,\omega,\omega')$, $n\in\mathbb{N}$, $B>0$, and a Schwarz function $h\in C^{\infty}_0(\mathbb{R})$ be given, and assume that
\begin{equation}
\label{BProcessCondition1}
\begin{gathered}
\kappa\geqslant 1+\iota'(4),\quad n-w>\kappa+\iota'+\iota'(12),\\
u+\lfloor\lambda\rfloor>\kappa+\iota',\quad\tilde{\lambda}=\min(\kappa-\rho_p(y),\lambda)>0.
\end{gathered}
\end{equation}
Let $\varepsilon_{\lambda}=\lfloor\lambda\rfloor-\lceil\tilde{\lambda}\rceil$. Then there exists a function
\begin{equation}
\label{DualFClasses}
\mathring{f}\in\mathbf{F}\big(w+\ord_py,y^{-1},\kappa,\tilde{\lambda},u+\varepsilon_{\lambda}-\ord_py,\omega'{}^{-1},-\omega^{-1}\big)
\end{equation}
depending on $f$ only and an $\epsilon\in\mathbb{C}$, $|\epsilon|=1$, such that
\[ \sum_{m\in\mathbb{Z}}e\left(\frac{f(m)}{p^n}\right)h\left(\frac mB\right)=\frac{\epsilon B}{p^{(n-w-\iota'-\kappa)/2}}\sum_{t\in\mathbb{Z}}e\left(\frac{\mathring{f}(t)}{p^n}\right)\hat{h}_{f,B}\left(\frac{t}{p^{n-w-\iota'-\kappa}/B}\right), \]
where $\hat{h}_{f,B}$ is a reflected translate of the Fourier transform $\hat{h}$ given by
\[ \hat{h}_{f,B}(t)=\hat{h}\left(-t-\frac{f'(0)}{p^n/B}\right). \]
\end{theorem}

\begin{proof}
Let $S$ denote the sum on the left-hand side of the equality to be proved. Since $e(f(t)/p^n)$ is periodic with period $p^{n-w}$, we have that
\begin{align}
S&=\sum_{m\bmod{p^{n-w}}}\sum_{q\in\mathbb{Z}}e\left(\frac{f\big(m+p^{n-w}q\big)}{p^n}\right)h\left(\frac{m+p^{n-w}q}B\right)\nonumber\\
&=\sum_{m\bmod{p^{n-w}}}e\left(\frac{f(m)p^{-w}}{p^{n-w}}\right)h^{\sharp}_B(m),\label{FirstExpressionOfS}
\end{align}
where
\[ h^{\sharp}_B(m)=\sum_{q\in\mathbb{Z}}h\left(\frac{m+p^{n-w}q}B\right) \]
is a $(\mathbb{Z}/p^{n-w}\mathbb{Z})$-periodic function. Applying Parseval's identity, we have that
\[ S=\frac1{p^{n-w}}\sum_{s\bmod{p^{n-w}}}\hat{\mathbf{e}}_f(s)\hat{h}^{\sharp}_B(-s), \]
where $\hat{\mathbf{e}}_f(s)$ is as in \eqref{EquationFourierTransform}, while, by unfolding,
\begin{align*}
\hat{h}^{\sharp}_B(s)&=\sum_{m\bmod{p^{n-w}}}h^{\sharp}_B(m)e\left(-\frac{sm}{p^{n-w}}\right)\\
&=\sum_{m\bmod{p^{n-w}}}\sum_{q\in\mathbb{Z}}h\left(\frac{m+p^{n-w}q}B\right)e\left(-\frac{s\big(m+p^{n-w}q\big)}{p^{n-w}}\right)\\
&=\sum_{m\in\mathbb{Z}}h\left(\frac mB\right)e\left(-\frac{sm}{p^{n-w}}\right).
\end{align*}
Applying the Poisson summation formula, we find that
\[ \hat{h}^{\sharp}_B(s)=\sum_{\sigma\in\mathbb{Z}}\int_{-\infty}^{\infty}h\left(\frac xB\right)e\left(-\frac{sx}{p^{n-w}}-\sigma x\right)\,\text{d}x=B\sum_{\sigma\in\mathbb{Z}}\hat{h}\left(\frac{s}{p^{n-w}/B}+B\sigma\right), \]
where $\hat{h}$ denotes the usual Fourier transform. Therefore,
\begin{equation}
\label{ExpressionOfS}
S=\frac{B}{p^{n-w}}\sum_{s\bmod{p^{n-w}}}\sum_{\sigma\in\mathbb{Z}}\hat{\mathbf{e}}_f(s)\hat{h}\left(-\frac{s}{p^{n-w}/B}+B\sigma\right).
\end{equation}

We make a remark that, while the expression of $S$ as \eqref{ExpressionOfS} can be reached in fewer steps, we have structured the above proof so as to emphasize the r\^ole of duality in passing from a long sum in \eqref{FirstExpressionOfS} to a short one or conversely.

We now crucially apply Lemma~\ref{StatPhaseFourier}. According to this lemma, which may be applied in light of \eqref{BProcessCondition1}, the Fourier transform $\hat{\mathbf{e}}_f(s)$ vanishes unless $s=\tilde{g}_0+\omega'\big(1+p^{\iota'+\kappa}t\big)$ for some $t\in\mathbb{Z}_p$, in which case $\hat{\mathbf{e}}_f(s)$ is given in terms of the exponential $e\big(\breve{f}(t)/p^n\big)$, with the function $\breve{f}$ as in the statement of Lemma~\ref{StatPhaseFourier}. Using this result, we obtain that
\begin{align*}
S&=\frac{B}{p^{n-w}}\epsilon p^{(n-w+\iota'+\kappa)/2}\sum_{t\bmod{p^{n-w-\iota'-\kappa}}}\sum_{\sigma\in\mathbb{Z}}\\
&\qquad\qquad\qquad e\left(\frac{\breve{f}(t)}{p^n}\right)\hat{h}\left(-\frac{\tilde{g}_0+\omega'\big(1+p^{\iota'+\kappa}t\big)}{p^{n-w}/B}+B\sigma\right)\\
&=\frac{\epsilon B}{p^{(n-w-\iota'-\kappa)/2}}\sum_{t\bmod{p^{n-w-\iota'-\kappa}}}\sum_{\sigma\in\mathbb{Z}}\\
&\qquad\qquad\qquad e\left(\frac{\breve{f}(\omega'{}^{-1}t)}{p^n}\right)\hat{h}\left(-\frac{\big(t-p^{n-w-\iota'-\kappa}\sigma\big)+\big(\tilde{g}_0+\omega'\big)p^{-\iota'-\kappa}}{p^{n-w-\iota'-\kappa}/B}\right)\\
&=\frac{\epsilon B}{p^{(n-w-\iota'-\kappa)/2}}\sum_{t\in\mathbb{Z}}e\left(\frac{\breve{f}(\omega'{}^{-1}t)}{p^n}\right)\hat{h}\left(-\frac{t+\big(\tilde{g}_0+\omega'\big)p^{-\iota'-\kappa}}{p^{n-w-\iota'-\kappa}/B}\right),
\end{align*}
by unfolding again. The statement of the theorem follows by setting
\[ \mathring{f}(t):=\breve{f}(\omega'{}^{-1}t)\in\mathbf{F}\big(\breve{w},y^{-1},\kappa,\tilde{\lambda},\breve{u},\omega'{}^{-1},-\omega^{-1}\big), \]
noting that $\big(\mathring{f}(t)\big)'=\omega'{}^{-1}\breve{f}'(\omega'{}^{-1}t)$ and recalling that $\tilde{g}_0+\omega'=f'(0)p^{-w}$.
\end{proof}

\begin{theorem}[$B$-process]
\label{Bprocess}
If $\big(k,\ell,r,\delta,(n_0,u_0,\kappa_0,\lambda_0)\big)$ is a $p$-adic exponent datum, then so is
\[ B\big(k,\ell,r,\delta,(n_0,u_0,\kappa_0,\lambda_0)\big)=\left(\ell-\frac12,k+\frac12,\tilde{r},\tilde{\delta},(\tilde{n}_0,\tilde{u}_0,\tilde{\kappa}_0,\tilde{\lambda}_0)\right), \]
where, denoting $\tilde{\lambda}=\min\big(\kappa-\rho_p(y),\lambda\big)$,
\[ \tilde{r}(y,p,\kappa,\lambda)=r\big(y^{-1},p,\kappa,\tilde{\lambda}\big),\quad\tilde{\delta}=\delta+\delta_{01}, \]
\begin{align*}
\tilde{\kappa}_0(y,p)&=\max\big(1+\iota'(4),\kappa_0(y^{-1},p),\lambda_0(y^{-1},p)+\rho_p(y)\big),\\
\tilde{\lambda}_0(y,p)&=\lambda_0(y^{-1},p),\\
\tilde{n}_0(y,p,\kappa,\lambda)&=\max\big(\kappa+\iota'+1+\iota'(12),\textnormal{ord}_py+n_0(y^{-1},p,\kappa,\tilde{\lambda})\big),\\
\tilde{u}_0(y,p,\kappa,\lambda)&=\max\big(\kappa-\lfloor\lambda\rfloor+\iota'+1,\textnormal{ord}_py+u_0(y^{-1},p,\kappa,\tilde{\lambda})+\lceil\tilde{\lambda}\rceil-\lfloor\lambda\rfloor,1\big).
\end{align*}
\end{theorem}

\begin{proof}
Let $f\in\mathbf{F}(w,y,\kappa,\lambda,u,\omega,\omega')$ and $0<B\leqslant p^{n-w-\iota'-\kappa}$ be given. Fix a Schwarz function $h\in C^{\infty}_0(\mathbb{R})$, and consider the sum
\[ S=\sum_{m\in\mathbb{Z}}e\left(\frac{f(m)}{p^n}\right)h\left(\frac mB\right), \]
with an eye to invoking Lemma~\ref{EquivalenceSharpSmooth}. According to Theorem~\ref{SummationFormula}, assuming that conditions \eqref{BProcessCondition1} hold, we have that
\[ S=\frac{\epsilon B}{p^{(n-w-\iota'-\kappa)/2}}\sum_{t\in\mathbb{Z}}e\left(\frac{\mathring{f}(t)}{p^n}\right)\hat{h}_{f,B}\left(\frac{t}{p^{n-w-\iota'-\kappa}/B}\right), \]
with $\hat{h}_{f,B}$ as in Theorem~\ref{SummationFormula}, $|\epsilon|=1$, and
\[ \mathring{f}\in\mathbf{F}\big(w+\ord_py,y^{-1},\kappa,\tilde{\lambda},u+\varepsilon_{\lambda}-\ord_py,\omega'{}^{-1},-\omega^{-1}\big). \]
We now estimate the sum on the right-hand side using the given $p$-adic exponent datum and Lemma~\ref{EquivalenceSharpSmooth}. We note that, importantly, the cutoff function $\hat{h}_{f,B}$ satisfies $\|\hat{h}_{f,B}\|_{\star}=\|\hat{h}\|_{\star}$, where $\|\cdot\|_{\star}$ is the (translation-invariant) quantity defined in \eqref{DefinitionHStar} which enters the condition $H_{\text{sm}}^{\sharp}(\delta)$ in Lemma~\ref{EquivalenceSharpSmooth}. The given exponent datum can be directly applied as long as
\[ 1\leqslant p^{n-w-\iota'-\kappa}/B\leqslant p^{(n-w-\iota'+\iota-\kappa)-\iota}, \]
which is trivially satisfied, and
\begin{equation}
\label{BProcessCondition2}
\begin{gathered}
n-w-\text{ord}_py\geqslant n_0\big(y^{-1},p,\kappa,\tilde{\lambda}\big),\quad\kappa\geqslant\kappa_0\big(y^{-1},p\big),\\ \tilde{\lambda}=\min\big(\kappa-\rho_p(y),\lambda\big)\geqslant\lambda_0\big(y^{-1},p\big),\quad u+\lfloor\lambda\rfloor-\lceil\tilde{\lambda}\rceil-\text{ord}_py\geqslant u_0\big(y^{-1},p,\kappa,\tilde{\lambda}\big).
\end{gathered}
\end{equation}
We thus find that
\begin{align*}
S&=\frac{\epsilon B}{p^{(n-w-\iota'-\kappa)/2}}\sum_{t\in\mathbb{Z}}e\left(\frac{\mathring{f}(t)}{p^n}\right)\hat{h}_{f,B}\left(\frac t{p^{n-w-\iota'-\kappa}/B}\right)\\
&\ll\|\hat{h}_{f,B}\|_{\star}\cdot\frac{B}{p^{(n-w-\iota'-\kappa)/2}}p^r\left(\frac{p^{(n-w-\iota'+\iota)-\kappa-\iota}}{p^{n-w-\iota'-\kappa}/B}\right)^k\left(\frac{p^{n-w-\iota'-\kappa}}B\right)^{\ell}\times\\
&\qquad\qquad\times\big(\log p^{(n-w-\iota'+\iota)-\kappa-\iota}\big)^{\delta}\\
&=\|\hat{h}\|_{\star}\cdot p^{r-(\iota'+\kappa)\ell+(\iota'+\kappa)/2}B^{1+k-\ell}\big(p^{n-w}\big)^{\ell-1/2}\big(\log p^{n-w-\kappa-\iota'}\big)^{\delta}\\
&=\|\hat{h}\|_{\star}\cdot p^{\tilde{r}}\left(\frac{p^{n-w-\kappa-\iota'}}B\right)^{\ell-1/2}B^{k+1/2}\big(\log p^{n-w-\kappa-\iota'}\big)^{\delta},
\end{align*}
with
\[ \tilde{r}=\tilde{r}(y,p,\kappa,\lambda)=r\big(y^{-1},p,\kappa,\tilde{\lambda}\big). \]

We now apply Lemma~\ref{EquivalenceSharpSmooth} again. Since the estimate proved above holds with a constant depending on the cutoff function $h$ only, we have according to the implication $H_{\text{sm}}(\delta)^{sq}\implies H(\delta+\delta_{1/2})^{sq}$ that, for every $M\in\mathbb{Z}$ and every $0<B\leqslant p^{n-w-\kappa-\iota'}$,
\[ \sum_{M<m\leqslant M+B}e\left(\frac{f(m)}B\right)\ll p^{\tilde{r}(y,p)}\left(\frac{p^{n-w-\kappa-\iota'}}B\right)^{\ell-1/2}B^{k+1/2}\big(\log p^{n-w-\kappa-\iota'}\big)^{\tilde{\delta}}, \]
with $\tilde{\delta}=\delta+\delta_{01}$ (with $\delta_{01}$ equal to the $\delta_{1/2}$ applied to the pair $(\ell-\tfrac12,k+\tfrac12)$) and a uniform implied constant. This estimate is valid as long as all conditions listed in \eqref{BProcessCondition1} and \eqref{BProcessCondition2} are satisfied; this gives us the $p$-adic exponent datum announced in the statement of the Theorem.
\end{proof}

In particular, applying Theorem~\ref{Bprocess} to the trivial $p$-adic exponent datum
\[ \omega_{01}=(0,1,0,0,(\kappa+\iota'+1,1,1+\iota'(2),\rho_p)), \]
we obtain the following important datum:
\begin{equation}
\label{ExponentDatumHalfHalf}
\omega_{1/2}=\Big(\tfrac12,\tfrac12,0,1,\big(\kappa+\iota'+1+\iota'(12),\max(\kappa-\lfloor\lambda\rfloor+\iota'+1,1),1+\iota'(4),\rho_p\big)\Big).
\end{equation}

\section{\texorpdfstring{$A$-process}{A-process}}

$A$-process relies on a procedure in which an estimate on the exponential sum \eqref{GeneralSum} is obtained by comparing it to sums obtained by replacing $f$ with its differences over pairs of points in appropriate $p$-adic neighborhoods (see \eqref{ResultOfWeylDifferencing} below); this has the effect of considerably reducing the modulus relative to the length of the summation. This estimate can be seen as an adaptation of the classical Weyl-van der Corput inequality. For clarity, we state the underlying inequality separately and in some generality.

\begin{lemma}
\label{WeylDifferencing}
Let $b:\mathbb{Z}\to\mathbb{C}$ be an arbitrary function such that $|b(t)|\ll 1$ for every $t\in\mathbb{Z}$. Let $M\in\mathbb{Z}$ and $B\in\mathbb{N}$, and let \[ S=\sum_{M<m\leqslant M+B}b(m). \]
Then, for every positive integer $0<H\leqslant B$,
\[ S^2\ll BH+H\sum_{0<|h|<\frac BH}\left|\sum_{m\in J(h)}b(m+hH)\overline{b(m)}\right|, \]
where
\[ J(h)=(M,M+B-hH]\cap (M-hH,M+B] \]
is an interval of length $|J(h)|=B-|h|H\leqslant B$.
\end{lemma}

\begin{proof}
Let $I(m)$ be an interval of the real axis depending on $m\in\mathbb{Z}$ defined as
\[ I(m)=\{t\in\mathbb{R}:M<m+tH\leqslant M+B\}=\left(\frac{M-m}H,\frac{M+B-m}H\right]. \]
Note that $|I(m)|=B/H$ for every $m\in\mathbb{Z}$. We can adapt Weyl's ``smoothing'' trick to write
\begin{align*}
&\sum_{M-H<m\leqslant M+B}\sum_{h\in I(m)}b(m+hH)\\
&\qquad=\sum_{M<m\leqslant M+B}b(m)\cdot\#\bigg\{(m_1,h):\begin{matrix}M-H<m_1\leqslant M+B,\\ h\in I(m_1),\,\,m=m_1+hH \end{matrix}\bigg\}\\
&\qquad=\sum_{M<m\leqslant M+B}b(m)\cdot\#\{h\in\mathbb{Z}:M-H<m-hH\leqslant M+B\}\\
&\qquad=\sum_{M<m\leqslant M+B}b(m)\left(\frac BH+\text{O}(1)\right)=\frac BHS+\text{O}(B).
\end{align*}
The second equality follows from an observation that, given $m\in(M,M+B]$, $m_1\in(M-H,M+B]$, and $h\in\mathbb{Z}$ such that $m=m_1+hH$, the condition that $h\in I(m_1)$ is automatically satisfied.

Applying the Cauchy-Schwarz inequality, we have that
\begin{align*}
\frac{B^2}{H^2}S^2
&\ll\left|\sum_{M-H<m\leqslant M+B}\sum_{h\in I(m)}b(m+hH)\right|^2+B^2\\
&\ll B\sum_{M-H<m\leqslant M+B}\left|\sum_{h\in I(m)}b(m+hH)\right|^2+B^2\\
&\ll\frac{B^3}{H}+B\sum_{M-H<m\leqslant M+B}\mathop{\sum\sum}\limits_{h_1,h_2\in I(m),\,\,h_1\neq h_2}b(m+h_1H)\overline{b(m+h_2H)}+B^2\\
&\ll\frac{B^3}{H}+B\mathop{\sum\sum\sum}_{\substack{M-H<m\leqslant M+B,\,\,0<|h|<\frac BH,\\g\in I(m),\,\,g+h\in I(m)}}b\big((m+gH)+hH)\overline{b(m+gH)}\\
&=\frac{B^3}H+B\sum_{0<|h|<\frac BH}\sum_{-\frac BH<g<\frac BH+1}\sum_{m\in J(h)}b(m+hH)\overline{b(m)}\\
&\ll\frac{B^3}H+\frac{B^2}H\sum_{0<|h|<\frac BH}\left|\sum_{m\in J(h)}b(m+hH)\overline{b(m)}\right|,
\end{align*}
with $J(h)$ as in the statement of the Lemma. This gives the desired inequality.
\end{proof}

The condition that $|b(t)|\ll 1$ is not essential and was introduced only with our application in mind. Following the proof practically verbatim with an extra application of the Cauchy-Schwarz inequality to estimate the error term from the smoothing, one can prove that, for every function $b:\mathbb{Z}\to\mathbb{C}$, we have
\[ S^2\ll H\sum_{0\leqslant |h|<\frac BH}\left|\sum_{m\in J(h)}b(m+hH)\overline{b(m)}\right|, \]
with the term $h=0$ accounting for the diagonal contribution; the statement of the lemma follows trivially when $|b(t)|\ll 1$.

We will use Lemma~\ref{WeylDifferencing} with $b(t)=e(f(t)/p^n)$ and with $H=p^{\chi}$ chosen as a power of $p$. The estimate we just proved reads as
\begin{equation}
\label{ResultOfWeylDifferencing}
S^2\ll BH+H\sum_{0<|h|<\frac BH}\left|\sum_{m\in J(h)}e\left(\frac{f(m+p^{\chi}h)-f(m)}{p^n}\right)\right|.
\end{equation}
In the following lemma, we consider the function appearing in the inner exponential sum.

\begin{lemma}
\label{AprocessClassesOfFunctions}
Let $f\in\mathbf{F}(w,y,\kappa,\lambda,u,\omega,\omega')$, and assume that
\[ u>\kappa-\lfloor\lambda\rfloor+\iota'\quad\text{and}\quad\lambda+\chi>\rho_p. \]
Let $\tilde{\varepsilon}(y)=1$ if $\textnormal{ord}_py<0$ and $\tilde{\varepsilon}(y)=0$ otherwise, and let
\[ \tilde{\lambda}=\min(\kappa-\tilde{\varepsilon}(y)\rho_p,\lambda),\quad \mu=\min\big(2\kappa+\iota'-\iota'(2)-\tilde{\varepsilon}(y),u+\lfloor 2\lambda-\rho_p\rfloor\big). \]
Let $\chi\geqslant 0$, $h\in\mathbb{Z}^{\times}_p$ be fixed. Then there exists a power series $g_1\in\mathbf{I}_0[\tilde{\lambda}](\mathbb{Z}_p)$ with $g_1'\in p^{\mu}\mathbf{I}_0[\tilde{\lambda}](\mathbb{Z}_p)$ such that the equality
\[ f_{\chi,h}(t):=f\big(t+p^{\chi}h\big)-f(t)=p^{\chi}hf'(t)+p^{2\chi+w}g_1(t) \]
holds for all $|t|_p<p^{\tilde{\lambda}}$. In particular,
\begin{align*}
f_{\chi,h}\in\mathbf{F}\Big(w+\chi&+\kappa+\iota',y+1,\kappa,\tilde{\lambda},\\
&\min\big(u+\lfloor\lambda\rfloor-\kappa-\iota',\chi+\kappa-\iota'(2)-\varepsilon(y)\big),\omega,\omega\omega'(-y)|y|_ph\Big).
\end{align*}

\end{lemma}

\begin{proof}
Since $r_f=r_{f'}\geqslant p^{\tilde{\lambda}}$, we have according to \eqref{TaylorSeriesValues} the equality of values
\[ f\big(t+p^{\chi}h\big)-f(t)=p^{\chi}hf'(t)+p^{2\chi}h^2\sum_{r=2}^{\infty}p^{(r-2)\chi}h^{r-2}\frac{f^{(r)}(t)}{r!} \]
for every $|t|_p<p^{\tilde{\lambda}}$.

We now consider the infinite sum of the series on the right-hand side as a formal sum. With $g(t)$ as in \eqref{ClassFDefinition} and writing $g'(t)=\sum_{j=0}^{\infty}g_jt^j\in\mathbf{I}_{0,1}[\lambda](\mathbb{Z}_p)$, the coefficient of the $r^{\text{th}}$ series with $t^j$ ($j\geqslant 0$) equals
\begin{align*}
a_{rj}&=\frac{h^{r-2}p^{w+(r-2)\chi}}{r!}\times\\
&\qquad\times\Big(\omega'\omega^{j+r-2}p^{\kappa+\iota'+(\iota+\kappa)(j+r-2)}\frac{(-y-1)_{j+r-2}}{j!}+p^ug_{j+r-2}(j+r-2)_{r-2}\Big).
\end{align*}
It follows that
\begin{align*}
\text{ord}_p(a_{rj})&\geqslant w+(r-2)\chi-\text{ord}_p(r!)\\
&\hphantom{{}\geqslant w}+\min\big(\kappa(j+r-1)-\tilde{\varepsilon}(y)\text{ord}_p(j!)+\iota',u+\lceil\lambda(j+r-1)\rceil\big)\\
&\geqslant w+\min\big(\lceil(\kappa+\chi-\rho_p)(r-2)-\rho_p\rceil+\lceil(\kappa-\tilde{\varepsilon}(y)\rho_p)j\rceil+\kappa+\iota',\\
&\mskip 150mu u+\lceil\lambda(j+1)+(r-2)(\lambda+\chi-\rho_p)-\rho_p\rceil\big).
\end{align*}
According to our discussion in \eqref{DoubleExchangeCriterion}, since $\min(\kappa,\lambda)+\chi>\rho_p$, the formal sum of power series converges to a power series $\tilde{g}_1(t)=\sum_{j=0}^{\infty}\tilde{g}_jt^j$ with
\[ \text{ord}_p\tilde{g}_j\geqslant w+\min(\lceil(\kappa-\tilde{\varepsilon}(y)\rho_p)j\rceil+\kappa+\iota'-\iota'(2),u+\lceil\lambda(j+1)-\rho_p\rceil); \]
moreover, we have that $\tilde{g}_1(t)\in p^w\mathbf{I}_0[\tilde{\lambda}](\mathbb{Z}_p)$, and the pointwise equality of values holds for all $|t|_p<p^{\tilde{\lambda}}$. Hence we can take $g_1(t)=h^2p^{-w}\tilde{g}_1(t)$. We also have that
\begin{align*}
\text{ord}_p((j+1)\tilde{g}_{j+1})\geqslant w+\min\big(\lceil(\kappa-\tilde{\varepsilon}(y)\rho_p)j\rceil+2\kappa+\iota'&-\iota'(2)-\tilde{\varepsilon}(y),\\
&u+\lceil\lambda(j+2)-\rho_p\rceil\big),
\end{align*}
so that $g_1'(t)\in p^\mu\mathbf{I}_0[\tilde{\lambda}](\mathbb{Z}_p)$ with $\mu=\min\big(2\kappa+\iota'-\iota'(2)-\tilde{\varepsilon}(y),u+\lfloor 2\lambda-\rho_p\rfloor\big)$, as announced.
\end{proof}

The following theorem establishes the $p$-adic $A$-process. Its statement may appear somewhat frightening, but this is due to our desire to work in full generality. We will see in section~\ref{LFunctionSection} how one obtains very concrete and easy to work with exponent data as long as one stays away from a finite number of primes and makes a concrete choice of $\kappa$. To keep the expressions manageable, we write $g(y^{\pm})$ to denote $\max(g(y),g(y^{-1}))$.

\begin{theorem}[$A$-process]
\label{Aprocess}
If $\big(k,\ell,r,\delta,(n_0,u_0,\kappa_0,\lambda_0)\big)$ is a $p$-adic exponent datum, then
\[ A\big(k,\ell,r,\delta,(n_0,u_0,\kappa_0,\lambda_0)\big)=\left(\frac k{2(k+1)},\frac{k+\ell+1}{2(k+1)},\tilde{r},\tilde{\delta},(\tilde{n}_0,\tilde{u}_0,\tilde{\kappa}_0,\tilde{\lambda}_0)\right) \]
is also a $p$-adic exponent datum.

Here, if $0<k\leqslant\frac12\leqslant\ell<1$, then, denoting $\tilde{\lambda}=\min\big(\kappa-\rho_p(y),\lambda\big)$,
\[ \tilde{r}(y,p,\kappa,\lambda)=\frac{r+k\big(1-\kappa-\min(\iota'(y+1),\iota'(y^{-1}+1))\big)}{2(k+1)},
\quad \tilde{\delta}=\frac{\max(1,\delta)}2, \]
as well as
{\allowdisplaybreaks
\begin{align*}
\tilde{\kappa}_0(y,p)&=\max\big(1+\iota'(4),\kappa_0(y^{\pm}+1,p),\rho_p(y)+\lambda_0(y^{\pm}+1,p),\rho_p(y)+2\rho_p\big),\\
\tilde{\lambda}_0(y,p)&=\max\big(\lambda_0(y^{\pm}+1,p),2\rho_p\big),\\
\tilde{u}_0(y,p,\kappa,\lambda)&=\max\big(1,u_0(y+1,p,\kappa,\tilde{\lambda})+\kappa-\lfloor\lambda\rfloor+\iota'(y),\\
&\qquad\qquad u_0(y^{-1}+1,p,\kappa,\tilde{\lambda})+\kappa+\lceil\tilde{\lambda}\rceil-\lfloor\lambda\rfloor-\lfloor\tilde{\lambda}\rfloor+\iota'(y),\\
&\qquad\qquad 2\kappa-\lfloor\lambda\rfloor-\lfloor\tilde{\lambda}\rfloor+\iota'(y(y+1))+1,\\
&\qquad\qquad 2\kappa+\lceil\tilde{\lambda}\rceil-\lfloor\lambda\rfloor-2\lfloor\tilde{\lambda}\rfloor+\iota'(y)+\iota'(y^{-1}+1)+1\big),\\
\tilde{n}_0(y,p,\kappa,\lambda)&=\kappa+\iota'(y)+\Big\lceil\max\Big(2\kappa+2\iota'(y^{\pm}+1)+2\iota'(12),\\
&\qquad\qquad  n_0(y^{\pm}+1,p,\kappa,\tilde{\lambda})+\kappa+\iota'(y^{\pm}+1)-1,\\
&\qquad\qquad \tfrac32n_0(y^{\pm}+1,p,\kappa,\tilde{\lambda})-\tfrac12\kappa-\tfrac12\iota'(y^{\pm}+1)-\tfrac32,\\
&\qquad\qquad n_0(y^{\pm}+1,p,\kappa,\tilde{\lambda})+\iota'(2)+\iota'(y^{\pm}+1)+\varepsilon(y^{\pm})-\lfloor\tilde{\lambda}\rfloor,\\
&\qquad\qquad \frac{2(r+\kappa+\iota'(y^{\pm}+1))+(k-1)}{1-\ell},\\
&\qquad\qquad \varepsilon_u\Big(\frac{2k+1-\ell}{k}\big(u_0(y^{\pm}+1,p,\kappa,\tilde{\lambda})-\kappa+\iota'(2)+\varepsilon(y^{\pm})\big)-\\
&\mskip 270mu -\frac{r-1}{k(k+1)}+\frac{\kappa+\iota'(y^{\pm}+1)}{k+1}\Big)\Big)\Big\rceil,
\end{align*}
}
where $\varepsilon_u=0$ if $u_0(y^{\pm}+1,p,\kappa,\tilde{\lambda})-\kappa+\iota'(2)+\varepsilon(y^{\pm})\leqslant 0$, and $\varepsilon_u=1$ otherwise.

If $k=0$, the above holds with
\begin{gather*}
\tilde{r}(y)=r(y)/2,\quad \tilde{\delta}=\delta/2,\\
\tilde{\kappa}_0=\kappa_0(y,p),\quad \tilde{\lambda}_0=\lambda_0(y,p),\quad \tilde{n}_0=n_0,\quad \tilde{u}_0=u_0.
\end{gather*}

If $\ell=1$, the above holds with
\begin{gather*}
\tilde{r}(y)=0,\quad \tilde{\delta}=\frac{k}{k+1},\quad \tilde{\kappa}_0=1+\iota'(4),\quad \tilde{\lambda}_0=\rho_p,\\
\tilde{n}_0=\kappa+\iota'(y)+1+\iota'(12),\quad \tilde{u}_0=\max\big(\kappa-\lfloor\lambda\rfloor+\iota'(y)+1,1\big).
\end{gather*}
\end{theorem}

\begin{proof}
Let $f\in\mathbf{F}(w,y,\kappa,\lambda,u,\omega,\omega')$, $M\in\mathbb{Z}$, and $0<B\leqslant p^{n-w-\kappa-\iota'}$ be given, and let
\[ S=\sum_{M<m\leqslant M+B}e\left(\frac{f(m)}{p^n}\right). \]
We consider the principal case $0<k\leqslant\frac12\leqslant\ell<1$; the complementary cases are easy and will be addressed at the end of the proof. Denote $\tilde{w}=w+\kappa+\iota'$, and let $\rho$ and $\sigma$ be real parameters, to be suitably chosen later. We seek to prove an estimate of the form
\begin{equation}
\label{AprocessDesiredEstimate}
S\ll p^{\tilde{r}}\left(\frac{p^{n-\tilde{w}}}B\right)^{\frac k{2(k+1)}}B^{\frac{k+\ell+1}{2(k+1)}}(\log p^{n-\tilde{w}})^{\tilde{\delta}}.
\end{equation}
The basic strategy is to estimate $S$ by applying the given $p$-adic exponent datum to the inner sum in \eqref{ResultOfWeylDifferencing}. For this purpose, we will choose $H$ to be a positive integer, in fact a power of $p$, satisfying
\begin{equation}
\label{ChoiceOfH}
H=p^{\chi}=p^{\sigma}\left(\frac{p^{n-\tilde{w}}}B\right)^{\frac{k}{k+1}}B^{\frac{\ell}{k+1}}
\end{equation}
for some $\sigma\in\mathbb{R}$ to be suitably chosen. It turns out that this strategy works well if $B$ is neither too small nor too large, in a sense which will be made precise.

To make the discussion easier to follow, we present the proof in two parts. The principal range for $B$, along with the easy case when $B$ is small, is treated in the first part of the proof. We will address the range when $B$ is large in the second part of the proof by using the summation formula of Theorem~\ref{SummationFormula} to shorten the sum down to the first range.

\medskip

\textbf{1. Range $1\leqslant B\leqslant p^{n-\tilde{w}-\rho}/H$.} If $1\leqslant B\leqslant H$, then we use the trivial bound $|S|\leqslant B$ to obtain
\[ |S|\leqslant B\leqslant (BH)^{1/2}=p^{\sigma/2}\left(\frac{p^{n-\tilde{w}}}B\right)^{\frac{k}{2(k+1)}}B^{\frac{k+\ell+1}{2(k+1)}}. \]
This suffices for \eqref{AprocessDesiredEstimate} as long as
\begin{equation}
\label{AprocessCondition1}
\tilde{r}\geqslant\sigma/2+o_p,
\end{equation}
where (here and on) we denote $o_p=\text{O}(1/\log p)$ and $0\leqslant o_p^{+}\ll 1/\log p$, so that $p^{o_p},p^{o_p^{+}}\asymp 1$.

We now consider the range $H\leqslant B\leqslant p^{n-\tilde{w}-\rho}/H$, which is of principal interest. The lower bound on $B$ implies that
\begin{gather}
H\geqslant p^{\sigma}\left(\frac{p^{n-\tilde{w}}}H\right)^{\frac{k}{k+1}}H^{\frac{\ell}{k+1}}=p^{\sigma}\big(p^{n-\tilde{w}}\big)^{\frac k{k+1}}H^{\frac{\ell-k}{k+1}},\nonumber\\
H\geqslant p^{\frac{\sigma}{2k+1-\ell}}\big(p^{n-\tilde{w}}\big)^{\frac k{2k+1-\ell}},\label{RangeForH}
\end{gather}
since we are assuming that $(k,\ell)\neq(0,1)$. The upper bound on $B$ can be equivalently written as
\begin{gather}
B\leqslant p^{n-\tilde{w}-\rho}p^{-\sigma}\left(\frac{p^{n-\tilde{w}}}B\right)^{-\frac k{k+1}}B^{-\frac{\ell}{k+1}}=p^{-(\rho+\sigma)}\big(p^{n-\tilde{w}}\big)^{\frac1{k+1}}B^{-\frac{\ell-k}{k+1}}\nonumber\\
\label{RangeForB}
B\leqslant p^{-(\rho+\sigma)\frac{k+1}{\ell+1}}\big(p^{n-\tilde{w}}\big)^{\frac1{\ell+1}}.
\end{gather}
We will assume that
\[ \rho\geqslant\hat{\kappa}, \]
where $\hat{\kappa}=\kappa+\iota'(y+1)$, thus ensuring that
\[ B\leqslant p^{n-\tilde{w}-\chi-\hat{\kappa}}. \]

We can rewrite \eqref{ResultOfWeylDifferencing}, the result of Weyl differencing (Lemma~\ref{WeylDifferencing}), as
\begin{equation}
\label{AprocessBaseFormula}
S^2\ll BH+H\sum_{0<|h|<\frac BH}\left|\sum_{m\in J(h)}e\left(\frac{f_{\chi,h}(m)}{p^n}\right)\right|,
\end{equation}
where $f_{\chi,h}(t)=f(t+p^{\chi}h)-f(t)$. According to Lemma~\ref{AprocessClassesOfFunctions}, assuming that
\begin{equation}
\label{ConditionsLemmaOnClasses}
u>\kappa-\lfloor\lambda\rfloor+\iota'(y),\quad \lambda+\chi+\ord_ph>\rho_p,
\end{equation}
we have that
\begin{align*}
f_{\chi,h}\in\mathbf{F}&\Big(w+\chi+\text{ord}_ph+\kappa+\iota',y+1,\kappa,\tilde{\lambda},\\
&\min\big(u+\lfloor\lambda\rfloor-\kappa-\iota',\chi+\text{ord}_ph+\kappa-\iota'(2)-\varepsilon(y)\big),\omega,\omega'h|h|_p(-y)|y|_p\Big).
\end{align*}

The inner sum $S(h)$ in \eqref{AprocessBaseFormula} will be estimated using an appropriate existing $p$-adic exponent datum. Write $\tilde{w}_{\chi}=\chi+\tilde{w}$, $h_p=|h|^{-1}_p=p^{\chi_p}$. We see that we can use the given $p$-adic exponent datum for those values of $h$ for which $\chi_p$ satisfies
\[ n-\tilde{w}_{\chi}-\chi_p\geqslant n_0:=n_0\big(y+1,p,\kappa,\tilde{\lambda}\big), \]
as long as all other conditions are satisfied. We separate the sum in \eqref{AprocessBaseFormula} into two appropriate ranges for $h$ as
\[ S^2\ll BH+H(S_1+S_2), \]
where
\[ S_1=\sum_{\substack{0<|h|<B/H\\ 0\leqslant\chi_p\leqslant n-\tilde{w}_{\chi}-n_0}}|S(h)|,\qquad S_2=\sum_{\substack{0<|h|<B/H\\ \chi_p>n-\tilde{w}_{\chi}-n_0}}|S(h)|. \]
We think of $BH$ and $HS_1$ as the two main terms in this estimate on $S^2$. All other terms we encounter will be estimated so as to be (essentially) majorized by upper bounds on one of them (as was already done in the case $B\leqslant H$).

\bigskip

We first estimate $S_1$. The inner sum $S(h)$ in $S_1$ can be estimated using the given $p$-adic exponent datum as long as
\begin{equation}
\label{AprocessCondition2}
\begin{gathered}
\kappa\geqslant\kappa_0(y+1,p),\quad \tilde{\lambda}=\min\big(\kappa-\rho_p(y),\lambda\big)\geqslant\lambda_0(y+1,p),\\
 u\geqslant u_0(y+1,p,\kappa,\tilde{\lambda})+\kappa-\lfloor\lambda\rfloor+\iota'(y),
 \end{gathered}
 \end{equation}
 as well as
\[ \chi\geqslant u_0(y+1,p,\kappa,\tilde{\lambda})-\kappa+\iota'(2)+\varepsilon(y). \]
The latter condition is trivially satisfied when the right-hand side is non-positive. If this is not the case, we still need this inequality only in the range \eqref{RangeForH}, so that it is satisfied whenever
\begin{equation}
\label{AprocessCondition2prime} 
n-\tilde{w}\geqslant\varepsilon_u\Big(\frac{2k+1-\ell}{k}\big(u_0(y+1,p,\kappa,\tilde{\lambda})-\kappa+\iota'(2)+\varepsilon(y)\big)-\frac{\sigma}k\Big).
\end{equation}
Writing $r=r(y+1,p,\kappa,\tilde{\lambda})$, we thus obtain the estimate
\begin{align*}
|S(h)|&\ll p^r\left(\frac{p^{n-\tilde{w}_{\chi}-\chi_p-\hat{\kappa}}}{|J(h)|}\right)^k|J(h)|^{\ell}(\log p^{n-\tilde{w}_{\chi}-\chi_p-\hat{\kappa}})^{\delta}\\
&\leqslant p^r\left(\frac{p^{n-\tilde{w}_{\chi}-\chi_p-\hat{\kappa}}}{B}\right)^kB^{\ell}(\log p^{n-\tilde{w}})^{\delta}
\end{align*}
valid for all $h$ appearing in $S_1$ for which $|J(h)|\leqslant p^{n-\tilde{w}_{\chi}-\chi_p-\hat{\kappa}}$, as well as the estimate
\[ |S(h)|\ll p^r\big(p^{n-\tilde{w}_{\chi}-\chi_p-\hat{\kappa}}\big)^{\ell}\left(\frac{|J(h)|}{p^{n-\tilde{w}_{\chi}-\chi_p-\hat{\kappa}}}+1\right)(\log p^{n-\tilde{w}})^{\delta}, \]
valid for all $h$ in $S_1$ regardless of the size of $|J(h)|$. Combining these estimates, we find that, assuming that \eqref{AprocessCondition2} and \eqref{AprocessCondition2prime} are satisfied, we have
\begin{align*}
HS_1&\ll p^rB\left(\frac{p^{n-\tilde{w}_{\chi}-\hat{\kappa}}}B\right)^kB^{\ell}(\log p^{n-\tilde{w}})^{\delta}\\
&\qquad\qquad+p^rH\sum_{\substack{0<|h|<\frac BH,\\ \chi_p>n-\tilde{w}_{\chi}-\hat{\kappa}-\log B/\log p}}\frac{B}{\big(p^{n-\tilde{w}_{\chi}-\chi_p-\hat{\kappa}}\big)^{1-\ell}}(\log p^{n-\tilde{w}})^{\delta}.
\end{align*}
The second term of this estimate is
\begin{align*}
&\leqslant p^r\frac{BH}{p^{(n-\tilde{w}_{\chi}-\hat{\kappa})(1-\ell)}}\sum_{\psi>n-\tilde{w}_{\chi}-\hat{\kappa}-\log B/\log p}p^{\psi(1-\ell)}\frac{B/H}{p^{\psi}}(\log p^{n-\tilde{w}})^{\delta}\\
&\ll p^r\frac{B^2}{p^{(n-\tilde{w}_{\chi}-\hat{\kappa})(1-\ell)}}\left(\frac{p^{n-\tilde{w}_{\chi}-\hat{\kappa}}}B\right)^{-\ell}(\log p^{n-\tilde{w}})^{\delta}\\
&=p^r\frac{B^{2+\ell}}{p^{n-\tilde{w}_{\chi}-\hat{\kappa}}}(\log p^{n-\tilde{w}})^{\delta}.
\end{align*}
In light of $B\leqslant p^{n-\tilde{w}_{\chi}-\hat{\kappa}}$, this term is $\ll p^rB^{1+\ell}(\log p^{n-\tilde{w}})^{\delta}$ and is absorbed in the first term of the estimate. Summing up, we have proved that, assuming \eqref{AprocessCondition2} and \eqref{AprocessCondition2prime},
\[ HS_1\ll p^r\left(\frac{p^{n-\tilde{w}_{\chi}-\hat{\kappa}}}B\right)^kB^{1+\ell}(\log p^{n-\tilde{w}})^{\delta}. \]

\bigskip

We now turn our attention to $S_2$, where we estimate the inner sum $S(h)$ using the $p$-adic exponential datum \eqref{ExponentDatumHalfHalf}. This is allowable as long as
\[ n-\tilde{w}_{\chi}-\chi_p\geqslant\hat{\kappa}+1+\iota'(12) \]
as well as
\begin{equation}
\label{AprocessCondition4}
\begin{gathered}
\kappa\geqslant 1+\iota'(4),\quad \tilde{\lambda}=\min\big(\kappa-\rho_p(y),\lambda\big)\geqslant\rho_p,\\
u\geqslant 2\kappa-\lfloor\lambda\rfloor-\lfloor\tilde{\lambda}\rfloor+\iota'(y(y+1))+1,
\end{gathered}
\end{equation}
and
\[ \chi+\chi_p\geqslant \iota'(2)+\iota'(y+1)+\varepsilon(y)+1-\lfloor\tilde{\lambda}\rfloor. \]
Note that the final condition is required for $\chi_p\geqslant n-\tilde{w}_{\chi}-n_0+1$, so that it is satisfied as long as
\begin{equation}
\label{AprocessCondition4Prime}
n-\tilde{w}-n_0\geqslant\iota'(2)+\iota'(y+1)+\varepsilon(y)-\lfloor\tilde{\lambda}\rfloor.
\end{equation}
Assuming this to be the case, we obtain the estimate
\[ |S(h)|\ll \big(p^{n-\tilde{w}_{\chi}-\chi_p-\hat{\kappa}}\big)^{1/2}\left(\frac{|J(h)|}{p^{n-\tilde{w}_\chi-\chi_p-\hat{\kappa}}}+1\right)\log p^{n-\tilde{w}_{\chi}-\chi_p-\hat{\kappa}}, \]
valid for all $h$ appearing in $S_2$ for which $\chi_p\leqslant n-\tilde{w}_{\chi}-\hat{\kappa}-1-\iota'(12)$. Estimating the remaining summands in $S_2$ trivially as $|S(h)|\leqslant B$, we thus find that
\begin{align*}
HS_2
&\ll H\sum_{\psi\geqslant n-\tilde{w}_{\chi}-n_0+1}\frac{B/H}{p^{\psi}}\left(\frac{B}{p^{(n-\tilde{w}_{\chi}-\hat{\kappa}-\psi)/2}}+p^{(n-\tilde{w}_{\chi}-\hat{\kappa}-\psi)/2}\right)\log p^{n-\tilde{w}}\\
&\qquad\qquad+H\sum_{\psi\geqslant n-\tilde{w}_{\chi}-\hat{\kappa}-\iota'(12)}\frac{B/H}{p^{\psi}}B\\
&\ll \frac{B^2}{p^{(n-\tilde{w}_{\chi}-\hat{\kappa})/2}}\frac1{p^{(n-\tilde{w}_{\chi}-n_0+1)/2}}\log p^{n-\tilde{w}}\\
&\qquad\qquad+Bp^{(n-\tilde{w}_{\chi}-\hat{\kappa})/2}\frac1{p^{3(n-\tilde{w}_{\chi}-n_0+1)/2}}\log p^{n-\tilde{w}}+\frac{B^2}{p^{n-\tilde{w}_{\chi}-\hat{\kappa}-\iota'(12)}}\\
&\ll BH\frac{Bp^{(n_0-1)/2+\hat{\kappa}/2}+Bp^{\iota'(12)+\hat{\kappa}}+p^{3(n_0-1)/2-\hat{\kappa}/2}}{p^{n-\tilde{w}}}\log p^{n-\tilde{w}}.
\end{align*}
We now arrange for our parameters to be such that this upper bound on $HS_2$ is no more than $BH\log p^{n-\tilde{w}}$. (Here and below, we sacrifice a small power of logarithm for no other reason but clarity.) We will initially do this for the entire range $H\leqslant B\leqslant p^{n-\tilde{w}_{\chi}-\rho}$; this range will be restricted in the second part of the proof, relaxing the conditions to be imposed. Letting $\mu=\max((n_0-1)/2+\hat{\kappa}/2,\iota'(12)+\hat{\kappa})$, the condition that $Bp^{\mu}\leqslant p^{n-\tilde{w}}$ is, in light of the range \eqref{RangeForB} for $B$, satisfied whenever
\begin{gather*}
p^{-(\rho+\sigma)\frac{k+1}{\ell+1}}\big(p^{n-\tilde{w}}\big)^{\frac1{\ell+1}}p^{\mu}\leqslant p^{n-\tilde{w}}\\
p^{-(\rho+\sigma)\frac{k+1}{\ell+1}}p^{\mu}\leqslant\big(p^{n-\tilde{w}}\big)^{\frac{\ell}{\ell+1}}.
\end{gather*}
Along with the condition that $p^{3(n_0-1)/2-\hat{\kappa}/2}\leqslant p^{n-\tilde{w}}$, we will have that $HS_2\ll BH\log p^{n-\tilde{w}}$ as long as \eqref{AprocessCondition4} and \eqref{AprocessCondition4Prime} hold as well as
\begin{equation}
\label{FirstRangeFornw}
\begin{gathered}
n-\tilde{w}\geqslant\frac{\ell+1}{2\ell}\big(\max(n_0-1+\hat{\kappa},2\iota'(12)+2\hat{\kappa})\big)-(\rho+\sigma)\frac{k+1}{\ell},\\ n-\tilde{w}\geqslant \tfrac32n_0-\tfrac12\hat{\kappa}-\tfrac32.
\end{gathered}
\end{equation}

\bigskip

Collecting all contributions from the estimations of $HS_1$ and $HS_2$, we find that, in the range under consideration, and assuming \eqref{ConditionsLemmaOnClasses}, \eqref{AprocessCondition2}, \eqref{AprocessCondition2prime}, \eqref{AprocessCondition4}, \eqref{AprocessCondition4Prime}, and \eqref{FirstRangeFornw}, we have that
\begin{equation}
\label{AprocessProvedSoFar}
\begin{aligned}
S^2&\ll BH+H(S_1+S_2)\\
&\ll BH\log p^{n-\tilde{w}}+p^r\left(\frac{p^{n-\tilde{w}-\hat{\kappa}}}B\right)^k\frac{B^{1+\ell}}{H^k}(\log p^{n-\tilde{w}})^{\delta}.
\end{aligned}
\end{equation}

An $H$ satisfying
\[ H^{k+1}=p^r\left(\frac{p^{n-\tilde{w}-\hat{\kappa}}}{B}\right)^kB^{\ell} \]
would be essentially optimal. We can't make this exact choice as we are bound by the condition that $H$ be a non-negative power of $p$. However, it is reasonable to seek an $H$ to be the power of $p$ for which
\begin{align*}
c_Hp^{\frac{r-1}{k+1}}\left(\frac{p^{n-\tilde{w}-\hat{\kappa}}}B\right)^{\frac{k}{k+1}}&B^{\frac{\ell}{k+1}}=H_{\ast}\\
&<H\leqslant H^{\ast}=c_Hp^{\frac{r+k}{k+1}}\left(\frac{p^{n-\tilde{w}-\hat{\kappa}}}B\right)^{\frac{k}{k+1}}B^{\frac{\ell}{k+1}},
\end{align*}
with a suitable choice of $c_H>0$, since $H^{\ast}=pH_{\ast}$. Such a choice is admissible as long as $H^{\ast}\geqslant 1$. There are several ways to ensure this; we find it convenient to invoke the condition \eqref{SecondRangeFornw} below, which will be imposed anyway. In light of this condition, we have that $n-\tilde{w}-\hat{\kappa}+1\geqslant (n_0-\hat{\kappa})+\max(n-\tilde{w}-n_0+1,\frac12(n_0-\hat{\kappa}-1))$, with the first of the two latter expressions $\geqslant\hat{\kappa}$, so that $n-\tilde{w}-\hat{\kappa}+1\geqslant (n_0-\hat{\kappa})+\max(\frac12(n-\tilde{w}-n_0+\hat{\kappa}+1),\frac12(n_0-\hat{\kappa}-1))\geqslant (n_0-\hat{\kappa})+\frac14(n-\tilde{w})$. It follows that $H^{\ast}\geqslant c_H\big(p^{r+(n_0-\hat{\kappa})k}(\log p^{n-\tilde{w}})^{\delta}\big)^{1/(k+1)}p^{(n-\tilde{w})/4(k+1)}/(\log p^{n-\tilde{w}})^{\delta}\geqslant 1$ for a sufficiently large $c_H>1$ (depending only on the initial $p$-adic exponent datum in Theorem~\ref{Aprocess}), since the second factor is trivially $\gg 1$, while the first factor is $\gg 1$ as seen (following Definition~\ref{ExponentPairsDefinition}) after \eqref{NoBetterThanSqrt}.

With such a choice of $H$, we have
\begin{equation}
\label{RangeForSigma}
\frac{r-1-k\hat{\kappa}}{k+1}+o_p^{+}<\sigma\leqslant\frac{r+k-k\hat{\kappa}}{k+1}+o_p^{+},
\end{equation}
as well as
\[ S^2\ll p^{\frac{r+k-k\hat{\kappa}}{k+1}}\left(\frac{p^{n-\tilde{w}}}B\right)^{\frac{k}{k+1}}B^{\frac{k+\ell+1}{k+1}}(\log p^{n-\tilde{w}})^{\max(1,\delta)}. \]
We see that this is allowable for \eqref{AprocessDesiredEstimate} as long as
\begin{equation}
\label{AprocessCondition5}
\tilde{r}\geqslant\frac{r+k(1-\kappa-\iota'(y+1))}{2(k+1)},\quad \tilde{\delta}\geqslant\frac{\max(1,\delta)}2.
\end{equation}

Note that the first of these two inequalities subsumes \eqref{AprocessCondition1}. Further, the first of the two conditions \eqref{ConditionsLemmaOnClasses} is subsumed in \eqref{AprocessCondition4}. The second condition can also be dispensed with if $\lambda\geqslant 2\rho_p$ or if $\chi>0$. We could ensure the latter by imposing a lower bound on $n-\tilde{w}$, but we keep things simple and make an innocuous assumption
\begin{equation}
\label{ConditionOnLambdaTilde}
\tilde{\lambda}\geqslant 2\rho_p
\end{equation}
to take place of \eqref{ConditionsLemmaOnClasses}, with a $\tilde{\lambda}$ in place of $\lambda$ with an eye on the second part of the proof.

Summing up, we have proved that the estimate \eqref{AprocessDesiredEstimate} holds for all $0<B\leqslant p^{n-\tilde{w}_{\chi}-\rho}$ (where $\rho\geqslant\hat{\kappa}$), assuming that all conditions listed in \eqref{AprocessCondition2}, \eqref{AprocessCondition2prime}, \eqref{AprocessCondition4}, \eqref{AprocessCondition4Prime}, \eqref{FirstRangeFornw}, \eqref{AprocessCondition5}, \eqref{ConditionOnLambdaTilde}, and \eqref{SecondRangeFornw} are met.

\bigskip

\textbf{2. The complementary range, split at $p^{(n-\tilde{w})/2}$, and conclusion.} The complementary range $p^{n-\tilde{w}_{\chi}-\rho}<B\leqslant p^{n-\tilde{w}}$ really \textit{should} be treated in a different way, for in this range the supposed second main term in \eqref{AprocessProvedSoFar} does not correctly capture the full contribution of the terms $|S(h)|$ to $HS_1$,  because the length of summation $|J(h)|$ in $S(h)$ is unfavorably large compared to the modulus $p^{n-\tilde{w}_{\chi}-\chi_p-\hat{\kappa}}$ already for $\chi_p=0$. This is in the nature of the method. The Weyl-van der Corput inequality (Lemma~\ref{WeylDifferencing}) has the effect of substantially reducing the modulus relative to the length of the summation; this is its intended purpose. But if the length of the summation, which remains $\asymp B$, is too large, then this effect goes too far.

One way to deal with the supplementary range, in which $B$ is rather large compared with the modulus $p^{n-\tilde{w}}$, is to apply the $p$-adic exponent datum \eqref{ExponentDatumHalfHalf} and then make sure that the resulting estimate $p^{(n-\tilde{w})/2}$ is no more than \eqref{AprocessDesiredEstimate}. (This approach should be compared to the application of the P\'olya-Vinogradov inequality in \cite{BarbanLinnikChudakov} to dispense with the range $x\gg q^{2/3}$ when estimating the sum $\sum_{n\leqslant x}\chi(n)$.) This turns out to work wonderfully for $p^{n-\tilde{w}_{\chi}}\leqslant  B\leqslant p^{n-\tilde{w}}$, requiring no adjustments to the final result, and not too badly for $p^{n-\tilde{w}_{\chi}-\rho}<B<p^{n-\tilde{w}_{\chi}}$, where the price to be paid is that one must require $\tilde{r}\geqslant(\rho+\sigma)/2+o_p\geqslant (\hat{\kappa}+\sigma)/2+o_p$, which increases the final upper bound by a factor of at least $p^{\hat{\kappa}/2}$. This would not be horrible (and it is certainly inconsequential if one is only concerned with a fixed prime $p$), but we can do substantially better. If we think about the proof of the datum \eqref{ExponentDatumHalfHalf}, we realize that it consists of an application of the summation formula of Theorem~\ref{SummationFormula}, followed by a trivial estimate of the resulting shortened sum. In this light, the range $p^{n-\tilde{w}_{\chi}}\leqslant B\leqslant p^{n-\tilde{w}}$ corresponds dually to the range $1\leqslant B\leqslant H$, in which our estimate \eqref{AprocessDesiredEstimate} was indeed obtained by the trivial bound. It thus becomes clear that, to avoid losses for $B<p^{n-\tilde{w}_{\chi}}$, we should follow the application of the summation formula not by the trivial estimate but by \textit{exactly the same estimates} that we used in the dual range $B>H$.

At this point, we reflect back on the range considered in the first part of the proof, choose
\[ \rho=\hat{\kappa}, \]
and instead claim \eqref{AprocessDesiredEstimate} \textit{for all $1\leqslant B\leqslant p^{(n-\tilde{w})/2}$ and only those $B$}. It suffices to establish \eqref{AprocessDesiredEstimate} for all $1\leqslant B\leqslant p^{(n-\tilde{w})/2}/b$, where $b>0$ is a suitably chosen large constant. Note that, for all $B$ in this interval,
\begin{align*}
p^{n-\tilde{w}_{\chi}-\hat{\kappa}}
&\geqslant\frac{p^{n-\tilde{w}-\hat{\kappa}}}{p^{\sigma}\big(p^{\frac{n-\tilde{w}}2}/b\big)^{\frac{k+\ell}{(k+1)}}}\\
&=p^{-\hat{\kappa}-\sigma}b^{\frac{k+\ell}{k+1}}\big(p^{n-\tilde{w}}\big)^{\frac{k+2-\ell}{2(k+1)}}=p^{-\hat{\kappa}-\sigma}b^{\frac{k+\ell}{k+1}}\big(p^{n-\tilde{w}}\big)^{\frac{1-\ell}{2(k+1)}}p^{(n-\tilde{w})/2}.
\end{align*}
We want to ensure that the left-hand side, which is a power of $p$, is at least $p^{(n-\tilde{w})/2}$; for this, it suffices to ensure that the right-hand is $>p^{(n-\tilde{w}-1)/2}$. Keeping in mind the range for $\sigma$ in \eqref{RangeForSigma} and adjusting the constant $b$ as necessary, we conclude
that the proof of the estimate \eqref{AprocessDesiredEstimate} in the first part covers the entire range $1\leqslant B\leqslant p^{(n-\tilde{w})/2}$ as long as
\begin{equation}
\label{AprocessCondition6}
n-\tilde{w}\geqslant\frac{2(k+1)}{1-\ell}\left(\hat{\kappa}+\frac{r+k-k\hat{\kappa}}{k+1}-\frac12\right)=\frac{2(r+\kappa+\iota'(y+1))+(k-1)}{1-\ell}.
\end{equation}
As we announced, this restriction of range also allows us to relax the condition \eqref{FirstRangeFornw} somewhat. In light of $B\leqslant p^{(n-\tilde{w})/2}$, the condition that $Bp^{\mu}\leqslant p^{n-\tilde{w}}$ is satisfied whenever $n-\tilde{w}\geqslant 2\mu$, so that we may replace \eqref{FirstRangeFornw} with
\begin{equation}
\label{SecondRangeFornw}
n-\tilde{w}\geqslant\max\big(n_0-1+\hat{\kappa},2\hat{\kappa}+2\iota'(12),\tfrac32(n_0-1)-\tfrac12\hat{\kappa}\big).
\end{equation}

\medskip

We now address the case when $B\geqslant p^{(n-\tilde{w})/2}$. Instead of $S$, consider
\[ S'=\sum_{m\in\mathbb{Z}}e\left(\frac{f(m)}{p^n}\right)h\left(\frac mB\right). \]
According to Theorem~\ref{SummationFormula}, assuming that
\begin{equation}
\label{AprocessConditionDual}
\begin{gathered}
\kappa\geqslant 1+\iota'(4),\quad n-\tilde{w}\geqslant 1+\iota'(12),\\
u+\lfloor\lambda\rfloor>\kappa+\iota',\quad\tilde{\lambda}=\min\big(\kappa-\rho_p(y),\lambda\big)>0,
\end{gathered}
\end{equation}
we have that
\[ S'=\frac{\epsilon B}{p^{(n-\tilde{w})/2}}\sum_{t\in\mathbb{Z}}e\left(\frac{\mathring{f}(t)}{p^n}\right)\hat{h}_{f,B}\left(\frac{t}{p^{n-\tilde{w}}/B}\right), \]
with $|\epsilon|=1$ and $\mathring{f}$ as in \eqref{DualFClasses}:
\[ \mathring{f}\in\mathbf{F}\big(w+\ord_py,y^{-1},\kappa,\tilde{\lambda},u+\varepsilon_{\lambda}-\ord_py,\omega'{}^{-1},-\omega^{-1}\big). \]
Note that $\tilde{w}':=(w+\ord_py)+\kappa+\iota=w+\kappa+\iota'=\tilde{w}$. Since $p^{n-\tilde{w}}/B\leqslant p^{(n-\tilde{w})/2}$, the first part of the proof shows that sharp-cutoff sums of $e(\mathring{f}(t)/p^n)$ of length no more than $p^{n-\tilde{w}}/B$ can be estimated as in \eqref{AprocessDesiredEstimate}, as long as $\mathring{f}$ satisfies all conditions accumulated in the process of proving this estimate. Refering to \eqref{AprocessCondition2}, \eqref{AprocessCondition2prime}, \eqref{AprocessCondition4}, \eqref{AprocessCondition4Prime}, \eqref{AprocessCondition5}, \eqref{AprocessCondition6} and \eqref{SecondRangeFornw}, we find that we require the following additional assumptions:
\begin{equation}
\label{AprocessDualCondition}
\begin{gathered}
\kappa\geqslant\kappa_0(y^{-1}+1,p),\quad \tilde{\lambda}\geqslant\lambda_0(y^{-1}+1,p),\\
u+\varepsilon_{\lambda}-\ord_py\geqslant u_0(y^{-1}+1,p,\kappa,\tilde{\lambda})+\kappa-\lfloor\tilde{\lambda}\rfloor+\iota(y),\\
n-\tilde{w}\geqslant\varepsilon_u\Big(\frac{2k+1-\ell}{k}\big(u_0(y^{-1}+1,p,\kappa,\tilde{\lambda})-\kappa+\iota'(2)+\varepsilon(y^{-1})\big)-\frac{\sigma'}k\Big),\\
u+\varepsilon_{\lambda}-\ord_py\geqslant 2\kappa-2\lfloor\tilde{\lambda}\rfloor+\iota'(y^{-1}(y^{-1}+1))+1,\\
n-\tilde{w}-n'_0\geqslant\iota'(2)+\iota'(y^{-1}+1)+\varepsilon(y^{-1})-\lfloor\tilde{\lambda}\rfloor\\
\tilde{r}\geqslant\frac{r+k(1-\kappa-\iota'(y^{-1}+1))}{2(k+1)},\\
n-\tilde{w}\geqslant\frac{2(r+\kappa+\iota'(y^{-1}+1))+(k-1)}{1-\ell},\\
n-\tilde{w}\geqslant\max\big(n'_0-1+\hat{\kappa}',2\hat{\kappa}'+2\iota'(12),\tfrac32(n'_0-1)-\tfrac12\hat{\kappa}'\big),
\end{gathered}
\end{equation}
where
\begin{gather*}
n'_0:=n'_0(y^{-1}+1,p,\kappa,\tilde{\lambda}),\quad \hat{\kappa}':=\kappa+\iota'(y^{-1}+1),\\
\frac{r-1-k\hat{\kappa}'}{k+1}+o_p^{+}<\sigma'\leqslant\frac{r+k-k\hat{\kappa}'}{k+1}+o_p^{+}.
\end{gather*}
Assuming that these hold, and in light of $\|\hat{h}_{f,B}\|_{\star}=\|\hat{h}\|_{\star}$, we can estimate $S'$, using also the implication $H(\delta)\implies H^{\sharp}_{\text{sm}}(\delta)$ of Lemma~\ref{EquivalenceSharpSmooth}, as
\[ S'\ll
\frac{B}{p^{(n-\tilde{w})/2}}
p^{\tilde{r}}B^{\frac{k}{2(k+1)}}\left(\frac{p^{n-\tilde{w}}}B\right)^{\frac{k+\ell+1}{2(k+1)}}
\big(\log p^{n-\tilde{w}}\big)^{\tilde{\delta}}, \]
with a uniform implied constant, depending on $h$ only. From this it follows that
\begin{align*}
S'&\ll p^{\tilde{r}}\big(p^{n-\tilde{w}}\big)^{\frac{\ell}{2(k+1)}}B^{\frac{2k+1-\ell}{2(k+1)}}\big(\log p^{n-\tilde{w}}\big)^{\tilde{\delta}}\\
&=p^{\tilde{r}}\big(p^{n-\tilde{w}}\big)^{\frac{k}{2(k+1)}}B^{\frac{\ell+1}{2(k+1)}}\big(\log p^{n-\tilde{w}}\big)^{\tilde{\delta}}\left(\frac{p^{n-\tilde{w}}}{B^2}\right)^{\frac{\ell-k}{2(k+1)}}\\
&\ll p^{\tilde{r}}\left(\frac{p^{n-\tilde{w}}}B\right)^{\frac{k}{2(k+1)}}B^{\frac{k+\ell+1}{2(k+1)}}\big(\log p^{n-\tilde{w}}\big)^{\tilde{\delta}}
\end{align*}
for all $B\geqslant p^{(n-\tilde{w})/2}$. Using the implication $H(\delta)\implies H_{\text{sm}}(\delta)$ of Lemma~\ref{EquivalenceSharpSmooth}, it follows from the first part of the proof that the same estimate also holds for all $B\leqslant p^{(n-\tilde{w})/2}$ and thus for all $1\leqslant B\leqslant p^{n-\tilde{w}}$. The same upper bound follows for $S$ in light of $H_{\text{sm}}(\delta)^{sq}\implies H(\delta)^{sq}$ of Lemma~\ref{EquivalenceSharpSmooth}, since no extra factor of $\delta_{1/2}$ appears in the power of the logarithm because the exponent pair $\big(\frac{k}{2(k+1)},\frac{k+\ell+1}{2(k+1)}\big)$ cannot equal $(\frac12,\frac12)$.

The stated $p$-adic exponent datum for $0<k\leqslant\frac12\leqslant\ell<1$ follows from collecting all conditions \eqref{AprocessCondition2}, \eqref{AprocessCondition2prime}, \eqref{AprocessCondition4}, \eqref{AprocessCondition4Prime}, \eqref{AprocessCondition5}, \eqref{ConditionOnLambdaTilde} \eqref{AprocessCondition6}, \eqref{SecondRangeFornw}, \eqref{AprocessConditionDual}, and \eqref{AprocessDualCondition}.

The remaining cases $k=0$ and $\ell=1$ follow directly by convex interpolation from bounds given by known $p$-adic exponent data. If $k=0$, we interpolate between the bound \eqref{SharpDef} for the given $p$-adic exponent datum and the trivial bound $S\ll B$ to obtain
\[ S\ll \left(p^rB^{\ell}\big(\log p^{n-\tilde{w}}\big)^{\delta}\right)^{1/2}B^{1/2}=p^{r/2}B^{\frac{\ell+1}2}\big(\log p^{n-\tilde{w}}\big)^{\delta/2}. \]
If $\ell=1$, we use convex interpolation between the bound \eqref{SharpDef} for the first nontrivial datum \eqref{ExponentDatumHalfHalf} and the trivial bound $S\ll B$ as follows:
\[ S\ll \left(p^{(n-\tilde{w})/2}\log p^{n-\tilde{w}}\right)^{\frac{k}{k+1}}B^{\frac1{k+1}}=\left(\frac{p^{n-\tilde{w}}}B\right)^{\frac{k}{2(k+1)}}B^{\frac{k+2}{2(k+1)}}\big(\log p^{n-\tilde{w}}\big)^{\frac{k}{k+1}}. \]
\end{proof}

We comment briefly on possible optimality of the obtained value of $\tilde{n}_0$. The condition that $\tilde{n}_0\geqslant (1+\epsilon)n_0$ with a fixed $\epsilon=\epsilon(k,\ell)>0$ appears essential to Weyl differencing method. We do not believe that, for example, a condition of the form $\tilde{n}_0\geqslant n_0+C(k,\ell)$ can suffice in general. Substantial effort was put into making $1+\epsilon$ as small as we could, but it is not clear that the factor of $\frac32$ is necessarily optimal.

While processes engaging some ``$q$-variant'' of the Weyl-van der Corput inequality have been used by previous authors, our approach in Theorem~\ref{Aprocess} is, to our knowledge, novel in a number of ways, including the use of the $\big(\frac12,\frac12)$ pair to reduce the required $\tilde{n}_0$ and of the summation formula to shorten the sum in the range $B\gg p^{(n-\tilde{w})/2}$ and obtain what are probably nearly optimal exponents, as well as the entire paradigm of the method applying to classes of $p$-adic analytic functions.

\section{Application to \texorpdfstring{$L$-functions}{L-functions}}
\label{LFunctionSection}

The relevance of the class $\mathbf{F}$ to Dirichlet $L$-functions stems from the following (in hindsight) simple Lemma~\ref{ParametrizationOfCharacters}. In a more elementary form, this line of reasoning seems to have been first used in the context of analysis of $L$-functions by Postnikov~\cite{Postnikov}.

Recall that the group $(\mathbb{Z}/p^n\mathbb{Z})^{\times}$ of invertible congruence classes modulo $p^n$ is cyclic for an odd prime $p$ and a product of the subgroup $\{\pm 1\}$ and a cyclic group of order $2^{n-2}$ if $p=2$ and $n\geqslant 2$ (we ignore the trivial case $p^n=2$ here). Let
\[ \kappa_1=1+\iota'(2), \]
and let $(\mathbb{Z}/p^n\mathbb{Z})^{\times}_1=\{a+p^n\mathbb{Z}:a\equiv 1\bmod{p^{\kappa_1}}\}$. We have that $(\mathbb{Z}/p^n\mathbb{Z})^{\times}=G_n\times(\mathbb{Z}/p^n\mathbb{Z})^{\times}_1$ with a subgroup $G_n\cong(\mathbb{Z}/p^{\kappa_1}\mathbb{Z})^{\times}$.

Let $\Gamma_n$ denote the set of all Dirichlet characters modulo $p^n$, and let $\Gamma_{n1}$ denote the set of all characters of the subgroup $(\mathbb{Z}/p^n\mathbb{Z})^{\times}_1$. We have the isomorphism of dual groups $\Gamma_n=\hat{G}_n\times\Gamma_{n1}$, and restriction to $1+p^{\kappa_1}\mathbb{Z}$ gives a natural surjection $\Gamma_n\twoheadrightarrow\Gamma_{n1}$.

\begin{lemma}
\label{ParametrizationOfCharacters}
Let $n\geqslant\kappa_1$ be given. For every $a=a_0p^{-n}\in p^{-n}\mathbb{Z}_p$,
\[ \chi_a\big(1+p^{\kappa_1}t\big)=e\big(a\log_p(1+p^{\kappa_1}t)\big)=e\left(\frac{a_0\log_p(1+p^{\kappa_1}t)}{p^n}\right) \]
defines a character $\chi_a\in\Gamma_{n1}$. Moreover, every character of $\Gamma_{n1}$ is of this form, and the correspondence $a\mapsto\chi_a$ induces an isomorphism $p^{-n}\mathbb{Z}_p/p^{-\kappa_1}\mathbb{Z}_p\cong\Gamma_{n1}$, with primitive characters being those corresponding to $p^{-n}\mathbb{Z}^{\times}_p/p^{-\kappa_1}\mathbb{Z}_p$.
\end{lemma}

\begin{proof}
We saw in section~\ref{Preliminaries} that the series $\lambda(x)=\log_p(1+x)$ has $r_{\lambda}=1$ and $M_r\lambda\doteq r$ for all $r<r_p$. Since
\[ \lambda(x+y+xy)=\lambda(x)+\lambda(y) \]
for every $x,y\in B_1$, it follows that $\chi_a$ is a multiplicative function $\chi_a:1+p^{\kappa_1}\mathbb{Z}\to S^1$. On the other hand, since $p^{\kappa_1}t\in B_{r_p}$ for every $t\in\mathbb{Z}$, we have that
\[ \ord_p\big(a\log_p(1+p^{\kappa_1}t)\big)=\ord_p\big(ap^{\kappa_1}t\big). \]
Note that $\chi_a\big(1+p^{\kappa_1}t\big)=1$ if and only if $a\log_p(1+p^{\kappa_1}t)\in\mathbb{Z}_p$. We see that $1+p^n\mathbb{Z}\subseteq\ker\chi_a$, so that $\chi_a$ is indeed a character of $(\mathbb{Z}/p^n\mathbb{Z})^{\times}_1$.

It is immediate that $a\mapsto\chi_a$ is a homomorphism of groups $p^{-n}\mathbb{Z}_p\to\Gamma_{n1}$. Moreover, we see that $\chi_a$ is the trivial character if and only if $ap^{\kappa_1}t\in\mathbb{Z}_p$ for every $t\in\mathbb{Z}$ (and in particular for $t=1$), that is, exactly when $a\in p^{-\kappa_1}\mathbb{Z}_p$, so that we have a monomorphism $p^{-n}\mathbb{Z}_p/p^{-\kappa_1}\mathbb{Z}_p\to\Gamma_{n1}$. This must be an isomorphism since $|p^{-n}\mathbb{Z}_p/p^{-\kappa_1}\mathbb{Z}_p|=|\Gamma_{n1}|=p^{n-\kappa_1}$; in particular, every character of $\Gamma_{n1}$ is of the form $\chi_a$ for some $a\in p^{-n}\mathbb{Z}_p$. Since the characters of $\Gamma_{n-1,1}$ are consequently of the form $\chi_a$ for some $a\in p^{-n+1}\mathbb{Z}_p$, the primitive characters of $\Gamma_{n1}$ correspond to $a\in p^{-n}\mathbb{Z}^{\times}_p$.
\end{proof}

Lemma~\ref{ParametrizationOfCharacters} presents a parametrization of the restrictions to $1+p^{\kappa_1}\mathbb{Z}$ of Dirichlet characters modulo $p^n$ by classes of $p$-adic rationals. The isomorphism exhibited in the Lemma extends to an isomorphism of inductive limits $\mathbb{Q}_p/p^{-\kappa_1}\mathbb{Z}_p\cong\Gamma^{(p)}_1$, with $\Gamma^{(p)}_1=\cup_{n=1}^{\infty}\Gamma_{n1}$ being the group of restrictions of all Dirichlet characters modulo all non-negative powers of $p$ to $1+p^{\kappa_1}\mathbb{Z}$.

Let $\chi$ be a primitive character modulo $q>1$ ($q=p^n$ in our case). The Dirichlet $L$-function $L(s,\chi)$ continues to an entire function and satisfies the functional equation
\[ \left(\frac q{\pi}\right)^{s/2}\Gamma\left(\frac{s+\varsigma}2\right)L(s,\chi)=\varepsilon(\chi)\left(\frac q{\pi}\right)^{(1-s)/2}\Gamma\left(\frac{1-s+\varsigma}2\right)L(1-s,\bar{\chi}), \]
where $\varsigma=0$ or $1$ according as $\chi$ is even or odd, and
\[ \varepsilon(\chi)=\frac{i^{-\varsigma}}{\sqrt{q}}\sum_{m\bmod q}\chi(m)e\left(\frac mq\right) \]
is a unit multiple of the normalized Gauss sum (see \cite[Theorem 4.15 on page 84]{IwaniecKowalski}). We will use the following standard expansion of $L(\frac12,\chi)$ in terms of short Dirichlet polynomials.

\begin{lemma}[Approximate functional equation]
\label{AFE}
Let $\chi$ be a primitive character modulo $q>1$, and let $A$ be a positive integer. Then
\[ L\left(\frac12,\chi\right)=\sum_{m=1}^{\infty}\frac{\chi(m)}{\sqrt{m}}V\left(\frac m{\sqrt{q}}\right)+\varepsilon(\chi)\sum_{m=1}^{\infty}\frac{\overline{\chi(m)}}{\sqrt{m}}V\left(\frac m{\sqrt{q}}\right), \]
where $V(y)$ is a smooth function of $y>0$ defined by
\[ V(y)=\frac1{2\pi i}\int_{(3)}y^{-u}\left(\cos\frac{\pi u}{4A}\right)^{-4A}\frac{\Gamma\left(\frac14+\frac{u+\varsigma}2\right)}{\Gamma\left(\frac14+\frac{\varsigma}2\right)}\,\frac{\textnormal{d}u}u, \]
and $V(y)$ and its derivatives satisfy the estimates
\[ y^aV^{(a)}(y)\ll(1+y)^{-A},\qquad y^aV^{(a)}(y)=\delta_{a0}+\text{O}(y^{1/6}). \]
\end{lemma}

\begin{proof}
This is an instance of \cite{IwaniecKowalski}, Theorem 5.3 and Proposition 5.4 on pages 98--100, with $G(u)=(\cos \pi u/4A)^{-4A}$ as on page 99.
\end{proof}

We now arrive at the theorem in which a $p$-adic exponent datum will be used to estimate the central value $L(1/2,\chi)$.

\begin{theorem}
\label{EstimateInTermsOfkl}
Suppose that $\big(k,\ell,r,\delta,(n_0,u_0,\kappa_0,\lambda_0)\big)$ is a $p$-adic exponent datum. Let $\delta'=1$ if $\ell=k+\frac12$, and $\delta'=0$ otherwise.

If $\ell\geqslant k+1/2$, then, for every $\kappa\geqslant\max\big(\kappa_0(1,p),1+\iota'(2)\big)$ and with $r=r(1,p,\kappa,\infty)$, and for every $n\geqslant\max\big(n_0(1,p,\kappa,\infty)+\kappa,2\kappa\big)$ and every primitive Dirichlet character $\chi$ modulo $p^n$,
\[ L\left(\frac12,\chi\right)\ll p^{r+\kappa(1-k-\ell)}\big(p^n\big)^{\left[\tfrac{k+\ell}2-\tfrac14\right]}(\log q)^{\delta+\delta'}.\]

If $\ell<k+1/2$, then, for every $\kappa\geqslant\max\big(\kappa_0(1,p),\lambda_0(1,p),1+\iota'(4)\big)$ and with $r=r(1,p,\kappa,\kappa)$, and for every $n\geqslant\max\big(n_0(1,p,\kappa,\kappa)+\kappa,2\kappa+1+\iota'(12)\big)$ and every primitive Dirichlet character $\chi$ modulo $p^n$,
\[ L\left(\frac12,\chi\right)\ll p^{r+\kappa(1-k-\ell)}\big(p^n\big)^{\left[\tfrac{k+\ell}2-\tfrac14\right]}(\log q)^{\delta}.\]
\end{theorem}

\begin{proof}
Using Lemma~\ref{AFE} with $A=2$, we can write $L(1/2,\chi)=S+\varepsilon(\chi)S'$, where
\[ S=\sum_{m=1}^{\infty}\frac{\chi(m)}{\sqrt{m}}V\left(\frac{m}{p^{n/2}}\right),\quad S'=\sum_{m=1}^{\infty}\frac{\overline{\chi(m)}}{\sqrt{m}}V\left(\frac m{p^{n/2}}\right). \]
We will prove an upper bound for $S$; the estimate on the sum $S'$ is exactly the same with $\chi$ replaced by $\bar{\chi}$. We first consider the case $\ell\geqslant k+1/2$.

For every $1\leqslant c\leqslant p^{\kappa}$ such that $p\nmid c$, fix an integer $c'$ with $cc'\equiv 1\pmod{p^n}$. We can decompose $S$ as
\begin{align*}
S&=\sum_{1\leqslant c\leqslant p^{\kappa}, p\nmid c}\sum_{m=0}^{\infty}\frac{\chi\big(c+p^{\kappa}m\big)}{\sqrt{c+p^{\kappa}m}}V\left(\frac{c+p^{\kappa}m}{p^{n/2}}\right)\\
&=\sum_{1\leqslant c\leqslant p^{\kappa}, p\nmid c}\chi(c)\sum_{m=1}^{\infty}\chi\big(1+p^{\kappa}c'm\big)W_c(m)+\text{O}\big(p^{\kappa/2}\big)
\end{align*}
with the cutoff function
\[ W_c(t)=\frac1{\sqrt{c+p^{\kappa}t}}V\left(\frac{c+p^{\kappa}t}{p^{n/2}}\right). \]
According to Lemma~\ref{ParametrizationOfCharacters}, the values of the primitive character $\chi$ modulo $p^n$ on $1+p^{\kappa_1}\mathbb{Z}$ are given by a character $\chi_a$ for some $a=a_0p^{-n}$, $a_0\in\mathbb{Z}^{\times}_p$. Since $\kappa\geqslant\kappa_1$, we can write
\begin{equation}
\label{SumAspAdic}
S=\sum_{1\leqslant c\leqslant p^{\kappa},p\nmid c}\chi(c)\sum_{m=1}^{\infty}e\left(\frac{a_0\log_p(1+p^{\kappa}c'm)}{p^n}\right)W_c(m)+\text{O}\big(p^{\kappa/2}\big).
\end{equation}

Recall from \eqref{FcClass} that the phase $f_c(t)=a_0\log_p(1+p^{\kappa}c't)$ belongs to the class $\mathbf{F}(\kappa,1,\kappa,\infty,\infty,c',a_0c')$. We estimate the inner sum $S(c)$ in \eqref{SumAspAdic} using a summation by parts argument similar to the one used in the proof of Lemma~\ref{EquivalenceSharpSmooth}. Let
\[ \tilde{S}(t)=\sum_{1\leqslant m\leqslant t}e\left(\frac{f_c(m)}{p^n}\right) \]
if $t\geqslant 1$, and $\tilde{S}(t)=0$ for $t<1$. Since
\[ n-\kappa\geqslant n_0,\quad \kappa\geqslant\kappa_0, \]
we can estimate $\tilde{S}(t)$ using the given $p$-adic exponential datum to find that
\begin{equation}
\label{TildeDatumConclusion}
\tilde{S}(t)\ll p^r\left(\frac{p^{n-2\kappa}}{t}\right)^kt^{\ell}(\log q)^{\delta}
\end{equation}
for $1\leqslant t\leqslant p^{n-2\kappa}$, and, more generally, for all $t\geqslant 0$,
\[ \tilde{S}(t)\ll p^r\left(p^{(n-2\kappa)k}t^{\ell-k}+\frac{t}{p^{(n-2\kappa)(1-\ell)}}\right)(\log q)^{\delta}. \]
Using summation by parts, we obtain
\begin{align*}
S(c)&=\int_{1-0}^{\infty}W_c(t)\,\text{d}\tilde{S}(t)=W_c(t)\tilde{S}(t)\bigg|_{1-0}^{\infty}-\int_{1}^{\infty}\tilde{S}(t)W'_c(t)\,\text{d}t\\
&\ll p^r(\log q)^{\delta}\int_1^{\infty}\left(p^{(n-2\kappa)k}t^{\ell-k}+\frac{t}{p^{(n-2\kappa)(1-\ell)}}\right)\times\\
&\mskip 80mu\times\left(\frac{p^{\kappa}}{(c+p^{\kappa}t)^{3/2}}\left|V\left(\frac{c+p^{\kappa}t}{p^{n/2}}\right)\right|+\frac{p^{\kappa-n/2}}{\sqrt{c+p^{\kappa}t}}\left|V'\left(\frac{c+p^{\kappa}t}{p^{n/2}}\right)\right|\right)\,\text{d}t.
\end{align*}
Introducing a substitution $t=(p^{n/2}\tau-c)p^{-\kappa}$, we find that
\begin{align*}
S(c)&\ll p^{r-n/4}(\log q)^{\delta}\int_{\tau_0}^{\infty}\left(p^{(n-2\kappa)(k+\ell)/2}\tau^{\ell-k}+p^{(n-2\kappa)(\ell-1/2)}\tau\right)\times\\
&\mskip 360mu\times\left(\frac{|V(\tau)|}{\tau^{3/2}}+\frac{|V'(\tau)|}{\sqrt{\tau}}\right)\,\text{d}\tau,
\end{align*}
where $\tau_0\geqslant p^{\kappa-n/2}$. Multiplying out the integrand, we obtain a sum of four improper integrals. In light of the asymptotic behavior of $V(\tau)$, all four of these integrals converge absolutely when extended to $(0,\infty)$ if $\ell-k>1/2$; in the case $\ell-k=1/2$, the same is true except that the integral of $|V(\tau)|\tau^{\ell-k-3/2}$ has a logarithmic singularity at zero. We thus have that
\begin{align*}
S(c)&\ll p^{r-n/4}(\log q)^{\delta}\left(p^{(n-2\kappa)(k+\ell)/2}(\log q)^{\delta'}+p^{(n-2\kappa)(\ell-1/2)}\right)\\
&\ll p^{r-(k+\ell)\kappa}\big(p^n\big)^{\left[\tfrac{k+\ell}2-\tfrac14\right]}(\log q)^{\delta+\delta'},
\end{align*}
since $(k+\ell)/2\geqslant\ell/2\geqslant\ell-1/2$.

Going back to \eqref{SumAspAdic}, we have that
\[ S\ll p^{r+\kappa(1-k-\ell)}\big(p^n\big)^{\left[\tfrac{k+\ell}2-\tfrac14\right]}(\log q)^{\delta+\delta'}+p^{\kappa/2}. \]
Since the estimate \eqref{TildeDatumConclusion} holds for all $1\leqslant t\leqslant p^{n-2\kappa}$, we know from \eqref{NoBetterThanSqrt} that its right-hand side is greater than $t^{1/2}$ throughout the same range. In particular, for $t=p^{n/2-\kappa}$, we find that
\[ p^r\big(p^{n/2-\kappa}\big)^{k+\ell}(\log q)^{\delta}\geqslant p^{n/4-\kappa/2}, \]
from which it follows that the first term dominates in our estimate of $S$. This completes the proof in the case $\ell\geqslant k+1/2$.

If $\ell<k+1/2$, we apply the $B$-process (Theorem~\ref{Bprocess}) to the given exponential datum and obtain a new datum
\[ B\big(k,\ell,r,\delta,(n_0,u_0,\kappa_0,\lambda_0)\big)=\left(\ell-\frac12,k+\frac12,\tilde{r},\tilde{\delta},(\tilde{n}_0,\tilde{u}_0,\tilde{\kappa}_0,\tilde{\lambda}_0)\right), \]
where
\begin{gather*}
\tilde{r}(1,p,\kappa,\infty)=r(1,p,\kappa,\kappa),\quad \tilde{\delta}=\delta+\delta_{01}=\delta,\\
\tilde{\kappa}_0(1,p)=\max\big(1+\iota'(4),\kappa_0(1,p),\lambda_0(1,p)\big),\\
\tilde{n}_0(1,p,\kappa,\infty)=\max(\kappa+1+\iota'(12),n_0(1,p,\kappa,\kappa)\big).
\end{gather*}
Since $k+1/2>(\ell-1/2)+1/2$, the first case of our theorem applies to this new $p$-adic exponent datum; this gives the stated result.
\end{proof}

We remark that the proof of Theorem~\ref{EstimateInTermsOfkl} applies verbatim to estimation of the values $L(1/2+it,\chi)$ at any point along the critical line. Using the appropriate approximate functional equation from \cite[Theorem 5.3 and Proposition 5.4]{IwaniecKowalski}, and denoting
\begin{gather*}
V_s(y)=\frac1{2\pi i}\int_{(3)}y^{-u}\left(\cos\frac{\pi u}{4A}\right)^{-4A}\frac{\Gamma\big(\frac12s+\frac{u+\varsigma}2\big)}{\Gamma\big(\frac12s+\frac{\varsigma}2\big)}\,\frac{\text{d}u}u,\\
W_c[s](t)=\frac1{(c+p^{\kappa}t)^s}V_s\left(\frac{c+p^{\kappa}t}{p^{n/2}}\right),
\end{gather*}
we find as above that
\begin{align*}
L\left(\frac12+it,\chi\right)&\ll\int_1^{\infty}|\tilde{S}(\tau)|\Big(\big|W_c[\tfrac12+it]'(\tau)\big|+\big|W_c[\tfrac12-it\big]'(\tau)\big|\Big)\,\text{d}\tau+p^{\kappa/2}\\
&\ll p^{r+\kappa-n/4}(\log q)^{\delta}\int_{\tau_0}^{\infty}\left(p^{(n-2\kappa)(k+\ell)/2}\tau^{\ell-k}+p^{(n-2\kappa)(\ell-1/2)}\tau\right)\times\\
&\qquad\qquad\times\left(\frac{(3+|t|)|V_{1/2+it}(\tau)|}{\tau^{3/2}}+\frac{|V'_{1/2+it}(\tau)|}{\sqrt{\tau}}\right)\,\text{d}\tau+p^{\kappa/2},
\end{align*}
with $\tau_0\geqslant p^{\kappa-n/2}$. Using the asymptotic $y^aV_{1/2+it}^{(a)}(y)\ll\big(1+y/\sqrt{3+|t|}\big)^{-A}$ and proceeding as above, we conclude that
\begin{align*}
L\left(\frac12+it,\chi\right)&\ll (3+|t|)^{\tfrac{\ell-k}2+\tfrac34}p^{r+\kappa-\tfrac n4+(n-2\kappa)\tfrac{k+\ell}2}(\log q)^{\delta+\delta'}\\
&\qquad +(3+|t|)^{\tfrac54}p^{r+\kappa-\tfrac n4+(n-2\kappa)\big(\ell-\tfrac12\big)}(\log q)^{\delta}+p^{\kappa/2}\\
&\ll (3+|t|)^{\tfrac54}p^{r+\kappa(1-k-\ell)}\big(p^n\big)^{\big[\tfrac{k+\ell}2-\tfrac14\big]}(\log q)^{\delta+\delta'}.
\end{align*}
In the remainder of this section, we describe explicit $p$-adic exponent data and apply them to estimation of the central value $L(1/2,\chi)$. The above bound (which is somewhat lossy for all nontrivial $(k,\ell)$ but conveniently compact), used with the same $p$-adic exponent data, then yields analogous estimates for $L(1/2+it,\chi)$ valid along the entire critical line with an explicit dependence in $t$, including the bound announced in the introduction.

\medskip

Using Theorem~\ref{EstimateInTermsOfkl}, we can obtain a subconvex estimate on $L(1/2,\chi)$ from every $p$-adic exponent datum in which $k+\ell<1$. We show how to obtain such $p$-adic exponent data by iterating the $A$-and $B$-processes (Theorems \ref{Aprocess} and \ref{Bprocess}). We have seen that the $p$-adic exponent data can take rather complicated forms in general, to account for all adjustments which need to be made at a finite number of special primes, possibly depending on $y$; the set of such primes was denoted by $P_0(y)$ in the definition of $p$-adic exponent data. We will, for simplicity, state our $p$-adic exponent data in their cleanest form, in which they are valid away from finitely many primes, and state the exceptions; however, for the application to Theorem~\ref{EstimateInTermsOfkl}, it is important to remember that the method does apply to every single prime without exception.

Applying the $A$-process to the datum $\omega_{1/2}$, we obtain with some labor the datum
\begin{align*}
A(\omega_{1/2})=AB(\omega_{01})&=\Big(\tfrac16,\tfrac23,\tfrac16(1-\kappa),\tfrac12,\\
&\qquad\qquad\big(\big\lceil\max(5\kappa-1,\varepsilon_u(\tfrac{13}3+\tfrac53\kappa-3\lfloor\tilde{\lambda}\rfloor)) \big\rceil,\\
&\qquad\qquad\hphantom{\big(}\max(2\kappa+\lceil\tilde{\lambda}\rceil-\lfloor\lambda\rfloor-2\lfloor\tilde{\lambda}\rfloor+1,1),1,2\rho_p\big)\Big)
\end{align*}
valid for all $p\not\in\{2,3\}$ such that $\ord_py=\ord_p(y+1)=0$,
\begin{align*}
A^2B(\omega_{01})&=\Big(\tfrac1{14},\tfrac{11}{14},\tfrac17(1-\kappa),\tfrac12,\\
&\qquad\big(\big\lceil\max(8\kappa-3,\varepsilon'_u(3\kappa-\tfrac92\lfloor\tilde{\lambda}\rfloor+5),\varepsilon_u(\tfrac{47}7\kappa+4\lceil\tilde{\lambda}\rceil-12\lfloor\tilde{\lambda}\rfloor+\tfrac{58}7))\big\rceil,\\
&\qquad\hphantom{\big(}\max(3\kappa+2\lceil\tilde{\lambda}\rceil-\lfloor\lambda\rfloor-4\lfloor\tilde{\lambda}\rfloor+1,1),1,2\rho_p\big)\Big)
\end{align*}
valid for all $p\not\in\{2,3\}$ such that $\ord_py=\ord_p(y+1)=\ord_p(y+2)=\ord_p(2y+1)=0$ ($\varepsilon'_u$ refers to the value of $\varepsilon_u$ in the previous datum),  and so on. We recall from the statement of Theorem~\ref{Aprocess} that, when constructing a new $p$-adic exponent datum $Aq$ from an existing datum $q=(k,\ell,r,\delta,(n_0,u_0,\kappa_0,\lambda_0))$ using the $A$-process, $\varepsilon_u$ is defined as $\varepsilon_u=0$ if $u_0(y^{\pm}+1,p,\kappa,\tilde{\lambda})-\kappa+\iota'(2)+\varepsilon(y^{\pm})\leqslant 0$, and $\varepsilon_u=1$ otherwise.

Our $p$-adic exponent data take an even simpler form if we restrict them to $\kappa=\tilde{\lambda}$, which is equivalent to $\lambda\geqslant\kappa$ and $\rho_p(y)=0$. Note that this condition is always sa\-tis\-fied in the cases needed for Theorem~\ref{EstimateInTermsOfkl} away from finitely many primes. Moreover, this condition ``propagates'' through the recursive $A$- and $B$-processes, since a pair $(\kappa,\tilde{\lambda})$ always satisfies the condition $\kappa=\tilde{\lambda}$ away from finitely many primes (possibly depending on $y$) if the pair $(\kappa,\lambda)$ does. Finally, note that, as shown below, with this restriction, every datum obtained from $\omega_{01}$ using $A$- and $B$-processes has $u_0=1$; in particular, this means that, away from finitely many primes, we always have $\varepsilon_u=0$ upon application of Theorem~\ref{Aprocess}. With this convenient restriction, we thus obtain the following $p$-adic exponent data:
{\allowdisplaybreaks
\begin{align*}
\omega_{01}[\kappa=\tilde{\lambda}]&=\big(0,1,0,0,(\kappa+1,1,1,\rho_p)\big),\\
&\qquad \ord_py=0,\\
B(\omega_{01})[\kappa=\tilde{\lambda}]&=\big(\tfrac12,\tfrac12,0,1,(\kappa+1,1,1,\rho_p)\big),\\
&\qquad p\not\in\{2,3\},\,\,\ord_py=0,\\
AB(\omega_{01})[\kappa=\tilde{\lambda}]&=\big(\tfrac16,\tfrac23,\tfrac16(1-\kappa),\tfrac12,(5\kappa-1,1,1,2\rho_p)\big),\\
&\qquad p\not\in\{2,3\},\,\,\ord_p\{y,y+1\}=0,\\
A^2B(\omega_{01})[\kappa=\tilde{\lambda}]&=\big(\tfrac1{14},\tfrac{11}{14},\tfrac17(1-\kappa),\tfrac12,(8\kappa-3,1,1,2\rho_p)\big),\\
&\qquad p\not\in\{2,3\},\,\,\ord_p\{y,y+1,y+2,2y+1\}=0,\\
A^3B(\omega_{01})[\kappa=\tilde{\lambda}]&=\big(\tfrac1{30},\tfrac{13}{15},\tfrac1{10}(1-\kappa),\tfrac12,(\lceil\tfrac {23}2\kappa-6\rceil,1,1,2\rho_p)\big),\\
&\qquad p\not\in\{2,3\},\,\,\ord_p\{y,y+1,y+2,y+3,\\
&\mskip 150mu 2y+1,2y+3,3y+1,3y+2\}=0,\\
BA^3B(\omega_{01})[\kappa=\tilde{\lambda}]&=\big(\tfrac{11}{30},\tfrac{8}{15},\tfrac1{10}(1-\kappa),\tfrac12,(\lceil\tfrac{23}2\kappa-6\rceil,1,1,2\rho_p)\big),\\
&\qquad p\not\in\{2,3\},\,\,\ord_p\{y,y+1,y+2,y+3,\\
&\mskip 150mu 2y+1,2y+3,3y+1,3y+2\}=0,\\
ABA^3B(\omega_{01})[\kappa=\tilde{\lambda}]&=\big(\tfrac{11}{82},\tfrac{57}{82},\tfrac7{41}(1-\kappa),\tfrac12,(\lceil\tfrac{71}4\kappa-\tfrac{39}4\rceil,1,1,2\rho_p)\big),\\
&\qquad p\not\in\{2,3\},\,\,\ord_p\{y,y+1,y+2,y+3,y+4,2y+1,2y+3,\\
&\mskip150mu 2y+5,3y+1,3y+2,3y+4,3y+4,3y+5,\\
&\mskip150mu 4y+1,4y+3,5y+2,5y+3\}=0.
\end{align*}
}

With a supply of $p$-adic exponent data, we can derive subconvex estimates on $L(1/2,\chi)$, reflect back on our method, and prove Theorem~\ref{MainResult}. Applying Theorem~\ref{EstimateInTermsOfkl} using the datum $AB(\omega_{01})$ and $\kappa=1$, we get that, for every $n\geqslant 4$ and every primitive Dirichlet character $\chi$ modulo $p^n$ with $p\not\in\{2,3\}$,
\[ L\left(\frac12,\chi\right)\ll p^{(n+1)/6}(\log p^n)^{3/2}, \]
which recovers the Weyl exponent $\theta=\tfrac16$ for a fixed $p$, as in \cite{BarbanLinnikChudakov} and \cite{FujiiGallagherMontgomery}, but with an explicit implied constant. Note that we cannot use special devices which allow one to precisely recover the Weyl exponent if one does not hope to iterate the process  (as in \cite{HeathBrownHybrid}). The estimate does improve upon the Burgess exponent $\theta=\frac{3}{16}$ for $n\geqslant 9$, although this is a minor point for us.

Note that, in the datum $AB(\omega_{01})$, $\tfrac16+\tfrac23=\tfrac56$. To improve upon the Weyl exponent, we need a $p$-adic exponent datum with $k+\ell<\tfrac56$. One such datum is provided by $ABA^3B(\omega_{01})$ above. Applying Theorem \ref{EstimateInTermsOfkl} with this datum and $\kappa=1$, we get that for every $n\geqslant 8$ and every primitive Dirichlet character $\chi$ modulo $p^n$ with $p\not\in\{2,3,5,7\}$,
\begin{equation}
\label{ExplicitEstimate}
L\left(\frac12,\chi\right)\ll p^{7/41}\big(p^n\big)^{27/164}(\log p^n)^{1/2}.
\end{equation}
This proves the main statement of Theorem~\ref{MainResult},
\begin{equation}
\label{AbstractEstimate}
L\left(\frac12,\chi\right)\ll p^r\big(p^n\big)^{\theta}(\log p^n)^{1/2}
\end{equation}
with $\theta=\frac{27}{164}<\frac16$ and $r=\frac7{41}$ for all primitive characters $\chi$ modulo $p^n$, $p\not\in\{2,3,5,7\}$, $n\geqslant 8$. Since $A$- and $B$-processes produce $p$-adic exponent data effective for every prime $p$ without exception, the same bound holds for all primitive characters $\chi$ modulo $p^n$ also in the case $p\in\{2,3,5,7\}$ for $n\geqslant n_0$, with different values of $r$ and $n_0$ (and so also with the same values of $r$ and $n_0$ by adjusting the implied constant). Further, a bound of the same form holds for all values of $n$ by adjusting the value of $r$. Finally, if $\chi$ is induced from a primitive charater $\chi_1$ modulo $p^{n_1}$, $0\leqslant n_1\leqslant n$, then $L(s,\chi)=L(s,\chi_1)$ if $n_1\geqslant 1$ and $L(s,\chi)=(1-p^{-s})L(s,\chi_1)$ if $n_1=0$, so that the statement follows also for non-primitive characters. This proves Theorem \ref{MainResult} for all Dirichlet characters to any prime power modulus with $\theta=\frac{27}{164}$.

Since $\tfrac{27}{164}\approx 0.1646<\tfrac16$, the estimate \eqref{ExplicitEstimate} breaks the Weyl exponent barrier for $n\geqslant n'_0$. As another benefit of our explicit calculations of full exponent data (including the values of $n_0$ and $r$), we can see that \eqref{ExplicitEstimate} improves on the Weyl exponent for all $n$ for which $\tfrac{7}{41}+\tfrac{27}{164}n<\tfrac16n$; this will be the case for all $n\geqslant 85$.

Note that no further improvement is obtained in \eqref{ExplicitEstimate} by taking a larger value of $\kappa$, and, equivalently, no harm is suffered by taking a smaller value of $\kappa$. This is in marked contrast to the works such as \cite{FujiiGallagherMontgomery,HeathBrownHybrid} in which the Weyl exponent is obtained, which essentially rely on a choice $\kappa>n/3+\text{O}(1)$. In our language, this ensures that, in appropriate ranges, $f_{\chi,h}(t)$ of Lemma \ref{AprocessClassesOfFunctions} is essentially a quadratic polynomial; this in turn allows for a sharper treatment of one special instance of the $A$-process but precludes iteration. It is essential for this iterative method to adopt the exactly opposite paradigm that $n$ is sufficiently \textit{large} compared to $\kappa$, so that $f(t)/p^n$ behaves like a $p$-adic analytic function, rather than sufficiently \textit{small} compared to $\kappa$ (which presents simplifications in special cases but can obstruct the view of the analogy). It is quite possible that better --- possibly substantially better --- values of $r$ and $n_0$ (but not $\theta$) in \eqref{AbstractEstimate} can be obtained by fixing the value of $\kappa$ in a range relative to $n$ so that, by the time the iteration of $A$- and $B$-processes reaches the final application of Weyl differencing, we do have $\kappa>n/3+\text{O}(1)$ and can obtain a sharper estimate. This would be a welcome development, but we felt that it would distract from the main thrust of this paper.

The above proof, relying on Theorem~\ref{EstimateInTermsOfkl}, applies verbatim to any $p$-adic exponent pair $(k,\ell)$ and shows that the bound \eqref{AbstractEstimate} holds with
\[ \theta=\frac{k+\ell}2-\frac14. \]
This brings to the fore the question of finding $p$-adic exponent pairs with $k+\ell$ as small as possible. It is immediate from Theorem~\ref{ABprocessesTheorem} that the set of $p$-adic exponent pairs we can construct from $(0,1)$ coincides with the set of (Archimedean) exponent pairs obtainable from $(0,1)$ by the classical $A$- and $B$-processes, for which we refer to \cite{GK}. For example, a further specific pair which improves on \eqref{ExplicitEstimate} is Phillips' exponent pair $ABA^3BA^2BA^2B(0,1)=(\tfrac{97}{696},\tfrac{480}{696})$ \cite{Phillips}, which gives $\theta=\tfrac{229}{1392}\approx 0{.}1645$.

The question of finding a value of $\theta$ as small as it is possible to obtain from the $A$- and $B$-processes was considered and solved by Rankin \cite{RankinExponent}. Rankin proved that there is a $\theta_0\approx 0{.}1645$ such that $\theta>\theta_0$ for every pair obtainable by $A$- and $B$-processes, and, conversely, for every $\theta_1>\theta_0$, there is an exponent pair obtainable from $(0,1)$ by $A$- and $B$-processes which yields $\theta\in(\theta_0,\theta_1)$. Our Theorem~\ref{ABprocessesTheorem} shows that the corresponding $p$-adic processes will yield a $p$-adic exponent pair with the same value of $\theta$; using Theorem~\ref{EstimateInTermsOfkl} with this pair, we obtain a proof of Theorem~\ref{MainResult} for any $\theta>\theta_0\approx 0{.}1645$.

\section{Acknowledgments}

These results were prompted by a question posed by Kannan Sounda\-ra\-ra\-jan during an open problems session in a workshop on Zeta Functions and Their Riemann Hypotheses, held at the Courant Institute for Mathematical Sciences, New York University in 2002. Part of this work was done while the author visited Centre Interfacultaire Bernoulli at the \'Ecole Polytechnique F\'ed\'erale de Lausanne for the National Science Foundation-supported special semester on Group Actions in Number Theory in 2011. Most central claims and the paper as a whole took their final shape while the author visited the Max Planck Institute for Mathematics in Bonn during the 2011-12 academic year. The support and exceptional research infrastructure of both the EPFL and the Max Planck Institute for Mathematics are most gratefully acknowledged. Finally, the author would like to thank Valentin Blomer, Gergely Harcos, and Peter Sarnak for helpful feedback.

\bigskip\bigskip

\bibliographystyle{amsplain}
\bibliography{references}

\end{document}